\UseRawInputEncoding
\documentclass[11pt, reqno]{amsart}
\usepackage{textcmds} 
\usepackage{amsmath, amssymb, amsfonts, amstext, verbatim, amsthm, mathrsfs}
\usepackage{microtype}
\usepackage[all]{xy}
\usepackage[modulo]{lineno}
\usepackage[usenames]{color}
\usepackage{aliascnt}
\usepackage{enumitem}
\usepackage{xspace}
\usepackage{amsfonts}
\usepackage{amssymb}
\usepackage[centertags]{amsmath}
\usepackage{amsthm}
\usepackage{rotating}
\usepackage[margin=3.5cm]{geometry}
\usepackage{dsfont}
\usepackage{bm}
\usepackage{color}
\usepackage{subfigure}
\usepackage{amsmath}
\usepackage{array}
\usepackage[all]{xy}
\usepackage{euscript}
\usepackage[T1]{fontenc}
\usepackage{mathbbol}
\usepackage{nccmath}
\usepackage{marginnote}
\usepackage{stmaryrd}
\usepackage{youngtab}
\usepackage{tikz}
\usepackage{blkarray}
\usepackage{comment}
 \usepackage{booktabs}

\usepackage{graphics,graphicx}  

\let\oldtocsection=\tocsection
\let\oldtocsubsection=\tocsubsection

\renewcommand{\tocsection}[2]{\hspace{0em}\oldtocsection{#1}{#2}}
\renewcommand{\tocsubsection}[2]{\hspace{1em}\oldtocsubsection{#1}{#2}}

\addtocontents{toc}{\protect\setcounter{tocdepth}{2}}

\usepackage[backref,colorlinks=true,linkcolor=black,citecolor=blue,urlcolor=blue,citebordercolor={0 0 1},urlbordercolor={0 0 1},linkbordercolor={0 0 1}]{hyperref} 

\tikzset{node distance=3cm, auto}

\makeatletter
\def\@secnumfont{\bfseries}
\def\section{\@startsection{section}{1}%
  \z@{.7\linespacing\@plus\linespacing}{.5\linespacing}%
  {\normalfont\Large\bfseries}}
\def\subsection{\@startsection{subsection}{2}%
  \z@{.5\linespacing\@plus.7\linespacing}{-.5em}%
  {\normalfont\large\bfseries}}
  \def\subsubsection{\@startsection{subsubsection}{3}%
  \z@{.5\linespacing\@plus.7\linespacing}{-.5em}%
  {\normalfont\bfseries}}
\makeatother

\theoremstyle{plain}
\newtheorem{thm}{Theorem}[subsection]
\newtheorem{thm*}[thm]{Theorem*}
\newtheorem{lemma}[thm]{Lemma}
\newtheorem{prop}[thm]{Proposition}
\newtheorem{cor}[thm]{Corollary}

\newtheorem{problem}[thm]{Problem}

\newtheorem{conjecture}[thm]{Conjecture}
\newtheorem{example}[thm]{Example}

\theoremstyle{definition}
\newtheorem{construction}[thm]{Construction}
\newtheorem{definition}[thm]{Definition}

\theoremstyle{remark}
\newtheorem{disclaimer}[thm]{Disclaimer}
\newtheorem{remark}[thm]{Remark}

\numberwithin{equation}{subsection} 

\numberwithin{figure}{section}
\numberwithin{table}{section}

     \renewcommand{\wr}{{{\rm w}}}

\newcommand{\NI}{{\noindent}}

\newcommand{\im}{{\rm im\,}}

\newcommand{\la}{{\lambda}}

\newcommand{\Z}{\mathbb{Z}}
\renewcommand{\H}{\mathbb{H}}

\newcommand{\R}{\mathbb{R}}
\newcommand{\Q}{\mathbb{Q}}
\newcommand{\C}{\mathbb{C}}
\newcommand{\CP}{\mathbb{CP}}

\newcommand{\K}{\mathbb{K}}
\newcommand{\lam}{\lambda}
\renewcommand{\sc}{\op{SC}}
\newcommand{\eps}{\varepsilon}
\newcommand{\calF}{\mathcal{F}}
\newcommand{\calA}{\mathcal{A}}
\newcommand{\calL}{\mathcal{L}}

\renewcommand{\bar}{\mathcal{B}}

\newcommand{\bdy}{\partial}

\newcommand{\Ai}{\mathcal{A}_\infty}
\newcommand{\Li}{\mathcal{L}_\infty}

\newcommand{\calM}{\mathcal{M}}

\newcommand{\wh}{\widehat}
\newcommand{\wt}{\widetilde}

\newcommand{\ovl}{\overline}
\newcommand{\op}[1]{{\operatorname{#1}}}

\newcommand{\Sh}{{\op{Sh}}}

\newcommand{\can}{{\op{can}}}
\newcommand{\sh}{\op{SH}}
\newcommand{\Lam}{\Lambda}
\newcommand{\Lamo}{\Lambda_{\geq 0}}
\newcommand{\chlin}{\op{CH}_{\op{lin}}}

\newcommand{\chlinsc}{\sc_{S^1,+}}

\newcommand{\cz}{{\op{CZ}}}
\newcommand{\ind}{\op{ind}}

\newcommand{\sss}{\vspace{2.5 mm}}
\renewcommand{\leqq}{\preceq}

\newcommand{\sign}{\diamondsuit}

\newcommand{\gapac}{\mathfrak{g}}

\newcommand{\bb}{\frak{b}}

\newcommand{\sk}{{\op{sk}}}

\renewcommand{\bar}{\mathcal{B}}
\newcommand{\id}{\mathbb{1}}

\newcommand{\ii}{\mathfrak{i}}
\newcommand{\jj}{\mathfrak{j}}

\newcommand{\sht}{\op{short}}
\newcommand{\lng}{\op{long}}

\newcommand{\leftgen}{\xi}
\newcommand{\rightgen}{\zeta}
\newcommand{\lonegen}{\tau}
\newcommand{\frakl}{\mathfrak{l}}
\newcommand{\frakr}{\mathfrak{r}}
\newcommand{\frakp}{\mathfrak{p}}
\newcommand{\shookrightarrow}{\overset{s}\hookrightarrow}

\newcommand{\calS}{\mathcal{S}}
\newcommand{\mult}{\kappa}
\newcommand{\gauss}{\op{Ga}}

\newcommand{\pvf}{T_{\op{poly}}}
\newcommand{\frakt}{\mathfrak{t}}
\newcommand{\vol}{\op{vol}}
\renewcommand{\qq}{d}

\makeatletter
\newcommand{\dashover}[2][\mathop]{#1{\mathpalette\df@over{{\dashfill}{#2}}}}
\newcommand{\fillover}[2][\mathop]{#1{\mathpalette\df@over{{\solidfill}{#2}}}}
\newcommand{\df@over}[2]{\df@@over#1#2}
\newcommand\df@@over[3]{%
  \vbox{
    \offinterlineskip
    \ialign{##\cr
      #2{#1}\cr
      \noalign{\kern1pt}
      $\m@th#1#3$\cr
    }
  }%
}
\newcommand{\dashfill}[1]{%
  \kern-.5pt
  \xleaders\hbox{\kern.5pt\vrule height.4pt width \dash@width{#1}\kern.5pt}\hfill
  \kern-.5pt
}
\newcommand{\dash@width}[1]{%
  \ifx#1\displaystyle
    2pt
  \else
    \ifx#1\textstyle
      1.5pt
    \else
      \ifx#1\scriptstyle
        1.25pt
      \else
        \ifx#1\scriptscriptstyle
          1pt
        \fi
      \fi
    \fi
  \fi
}
\newcommand{\solidfill}[1]{\leaders\hrule\hfill}
\makeatother

\begin{document}

\title{Computing higher symplectic capacities I} 

\author{Kyler Siegel}
\address{Columbia University Department of Mathematics, 2990 Broadway, 10027 NY, USA}
\email{kyler@math.columbia.edu}

\date{\today}

\maketitle

\begin{abstract}
We present recursive formulas which compute the recently defined ``higher symplectic capacities'' for all convex toric domains. 
In the special case of four-dimensional ellipsoids, we apply homological perturbation theory to the associated filtered $\mathcal{L}_\infty$ algebras and prove that the resulting structure coefficients count punctured pseudoholomorphic curves in cobordisms between ellipsoids. 
As sample applications, we produce new previously inaccessible obstructions for stabilized embeddings of ellipsoids and polydisks, and we give new counts of curves with tangency constraints. 
\end{abstract}

\tableofcontents

\section{Introduction}\label{sec:intro}

\subsection{Context}\label{subsec:context}

This paper is about two closely related problems in symplectic geometry:
\begin{enumerate}[label=(\alph*)]
\item
understanding when there is a symplectic embedding of one domain into another of the same dimension
\item 
counting punctured pseudoholomorphic curves in a given domain.
\end{enumerate}
Each of these questions has both a qualitative and quantitative version,
the former being more topological and the latter being more geometric in flavor.
The qualitative version of (a) asks for symplectic embeddings up to a suitable class of symplectic deformations, and the qualitative version of (b) seeks enumerative invariants which are independent of such deformations.
The quantitative version of (a) asks for symplectic embeddings on the nose, and the quantitative version of (b) seeks invariants which are potentially sensitive to the symplectic shape of a given domain.

In this paper we will focus on the quantitative theory, and by ``domain'' we have in mind star-shaped subdomains of $\C^n$ for some $n \geq 1$. 
One can also extend the discussion to a wider class of open symplectic manifolds such as Liouville domains\footnote{Recall that a Liouville domain is a compact symplectic manifold $(X,\omega)$, where $\omega = d\theta$ for a one-form $\theta$, such that the Liouville vector field $X_\theta$ characterized by $\omega(X_\theta,-) = \theta$ is outwardly transverse along the boundary of $X$.
} or nonexact symplectic manifolds with sufficiently nice boundary.
A class of particular importance in dynamics is given by the ellipsoids $E(a_1,\dots,a_n) \subset \C^n$ with area parameters $a_1,\dots,a_n \in \R_{> 0}$, defined by 
\begin{align*}
E(a_1,\dots,a_n) := \{(z_1,\dots,z_n) \in \C^n \;:\; \sum_{i=1}^n \pi |z_i|^2/a_i \leq 1\}.
\end{align*}
{We will typically assume that the area factors are ordered as $a_1 \leq \dots \leq a_n$ and are rationally independent, in which case $\bdy E(a_1,\dots,a_n)$ has precisely $n$ simple Reeb orbits, with actions $a_1,\dots,a_n$ respectively.

\subsubsection{Symplectic embeddings}

The motivation for (a) goes back to Gromov's celebrated nonsqueezing theorem \cite{gromov1985pseudo}, which
states that a large ball cannot be squeezed by a symplectomorphism into a narrow infinite cylinder.
Gromov proved this result using his newly minted theory of pseudoholomorphic curves, thereby giving the first non-classical obstructions for symplectic embeddings.
This kickstarted the search for an understanding of the ``fine structure'' of symplectic embeddings.
The following problem is still largely open for $n \geq 3$ and provides a useful metric for progress in this field:
\begin{problem}[ellipsoid embedding problem (EEP)]\label{problem:eep}
For which $a_1,\dots a_n$ and $a_1' \dots a_n'$ is there a symplectic embedding 
$E(a_1,\dots,a_n) \shookrightarrow E(a_1',\dots,a_n')$?
\end{problem}
\noindent Note that the only classical obstruction is the volume constraint 
\begin{align*}
\tfrac{1}{n!}a_1\dots a_n \leq \tfrac{1}{n!}a_1' \dots a_n',
\end{align*}
while Gromov's nonsqueezing theorem amounts to the inequality
\begin{align*}
\min(a_1,\dots,a_n) \leq \min(a_1',\dots,a_n').
\end{align*}

There have been a number of important contributions to Problem~\ref{problem:eep} and its cousins, and we mention here only a partial list of symplectic rigidity highlights:
\begin{itemize}
\item the construction by Ekeland--Hofer of an infinite sequence of symplectic capacities $$c_1^{\op{EH}}(X) \leq c_2^{\op{EH}}(X) \leq c_3^{\op{EH}}(X) \leq \dots$$ associated to a domain $X$ of any dimension, often giving stronger obstructions than Gromov's nonsqueezing theorem
\item the solution (or at least reduction to combinatorics) by McDuff \cite{McDuff_Hofer_conjecture} of the four-dimensional ellipsoid embedding problem $E(a,b) \shookrightarrow E(a',b')$
\item the construction by Hutchings \cite{Hutchings_quantitative_ECH} of the embedded contact homology (ECH) capacities $$c_1^{\op{ECH}}(X) \leq c_2^{\op{ECH}}(X) \leq c_3^{\op{ECH}}(X) \leq \dots$$ associated to a four-dimensional domain $X$,
with many strong applications to four-dimensional symplectic embedding problems (see e.g. \cite{choi2014symplectic,cristofaro2019symplectic} and the references therein).
 \end{itemize}
 
There have also been some important developments on the side of symplectic flexibility, including the advent of symplectic folding (see \cite{schlenk1999symplectic}), Guth's result on polydisk embeddings \cite{Guth_polydisks}, and its subsequent refinement by Hind \cite{hind2015some}.
For example, the latter together with \cite{Pelayo-Ngoc_hofer_question} produces a symplectic embedding
\begin{align}\label{Hind_folding_emb}
E(1,x,\infty) \shookrightarrow E(\tfrac{3x}{x+1},\tfrac{3x}{x+1},\infty)
\end{align}
for any $x \in \R_{\geq 1}$.
In particular, by \cite{hind2015some} there is an embedding\footnote{Note that we have $E(c,c,\infty) = B^4(c) \times \C$.}
$$E(1,\infty,\infty) \shookrightarrow E(c,c,\infty)$$
if $c \geq 3$, and in fact this is optimal by \cite{HK}.

Although the full solution to the higher dimensional ellipsoid embedding problem is still seemingly out of reach, the following special case of Problem~\ref{problem:eep} has recently gained popularity and probes to what extent gauge-theoretic obstructions in dimension four persist in higher dimensions.
\begin{problem}[stabilized ellipsoid embedding problem]\label{problem:SEEP}
Fix $N \in \Z_{\geq 1}$. For which $a,b$ and $a',b'$ is there a symplectic embedding 
$E(a,b) \times \C^N \shookrightarrow E(a',b') \times \C^N$?
\end{problem}
\noindent Note that $E(a,b) = a\cdot E(1,x)$ for $x = b/a$. Restricting to the case that the target is a stabilized round four-ball, we arrive at:
\begin{problem}[restricted stabilized ellipsoid embedding problem]\label{prob:res_stab_ell}
For $N \in \Z_{\geq 1}$, determine the function 
$$f_N(x) := \inf \{c \in \R_{> 0}\;:\; E(1,x) \times \C^N \shookrightarrow B^4(c) \times \C^N\}.$$
\end{problem}

The analogous unstabilized function $f_0(x)$ was determined by 
McDuff--Schlenk \cite{McDuff-Schlenk_embedding} and dubbed the ``Fibonacci staircase''.
Based on ~\eqref{Hind_folding_emb}, McDuff \cite{Mint} has given the following explicit conjecture for the form of $f_N(x)$ for $N \geq 1$:
\begin{conjecture}[restricted stabilized ellipsoid embedding conjecture]\label{conj:stab_ell}
For $x \in \R_{\geq 1}$ and $N \geq 1$, we have
\begin{align*}
f_N(x) = \begin{cases}
f_0(x) & \mbox{ if } x \leq \tau^4\\
\tfrac{3x}{x+1} & \mbox{ if } x > \tau^4.
\end{cases}
\end{align*}
\end{conjecture}
\noindent Here $\tau = \tfrac{1+\sqrt{5}}{2}$ denotes the golden ratio.
Combining several results, the following progress has been made:
\begin{thm}[\cite{HK, CGH, Ghost, Mint}]
Conjecture~\ref{conj:stab_ell} holds in the following cases:
\begin{itemize}
\item for all $x \leq \tau^4$
\item for $x \in 3\Z_{\geq 1} - 1$
\item for $x$ in a decreasing sequence of numbers $b_0 = 8,b_1=\frac{55}{8},b_2 = \frac{377}{55},\dots$
determined by the even index Fibonacci numbers, with $b_i \rightarrow \tau^4$.
\end{itemize}
\end{thm}
\noindent These techniques require a rather strong geometric control on SFT-type degenerations after neck-stretching, so they appear difficult to scale as the number of possibilities grows.
It has therefore been difficult to say whether Conjecture~\ref{conj:stab_ell} holds for other seemingly simple cases such as $x = 7$ or $x = \frac{19}{2}$.
One application of the results of this paper is a proof of many new cases of Conjecture~\ref{conj:stab_ell} - see Example~\ref{ex:new_seep_cases}.

\subsubsection{Enumerating punctured curves}

We now move to motivation (b).
Let $X^{2n}$ be a Liouville domain.
What does it mean to ``count'' punctured pseudoholomorphic curves in $X$? 
Since there is a naturally induced contact structure on $\bdy X$, we can speak of Reeb orbits in $\bdy X$.
After passing\footnote{Note that all almost complex structures and punctured pseudoholomorphic curves in this paper are in {\em completed} symplectic cobordisms, so we will sometimes omit explicit mention of this completion process without risk of ambiguity.} to the symplectic completion $\wh{X}$ of $X$ and picking an SFT-admissible almost complex structure $J$, 
we can consider pseudoholomorphic curves which are asymptotic to Reeb orbits at each of the punctures.
If we fix the genus $g$, homology class $A$, and  nondegenerate Reeb orbit asymptotes $\gamma_1,\dots,\gamma_k$, this gives rise to a moduli space $\calM$ of curves in $X$, of expected dimension
$$\ind\,\calM = (n-3)(2-2g-k) + \sum_{i=1}^k \cz(\gamma_i) + 2c_1([\omega]) \cdot A.$$
Note that the asymptotic condition at a puncture is formulated in such a way that these curves are proper.
We refer the reader to \cite[\S 3]{HSC} for more details on the geometric setup.

In principle we can try to count the elements of $\calM$, but several basic issues come to mind:
\begin{enumerate}
\item Typically, $\calM$ will have nonzero expected dimension, so it will not contain finitely many elements. For example, any convex domain $X \subset \C^2$ is dynamically convex (see \cite{HWZ_dyn_conv}), meaning that $\calM$ will have strictly positive dimension.
\item Even if we do get a finite count, it could depend on the choice of $J$, along with any other choices we make during the construction (c.f. ``wall-crossing'' phenomena). 
\end{enumerate}

Following \cite{HSC}, we can take care of (1) by imposing additional constraints to cut down the dimension of the moduli space $\calM$.
The most basic such constraint is to require our curves to pass through $k$ generic points $p_1,\dots,p_k \in X$,
with each $p_i$ cutting down the dimension by $2n - 2$.
A different constraint of the same codimension is given by picking a generic local divisor $D$ at a single point $p \in X$ and requiring curves to pass through $p$ with contact order $k$ to $D$.
This idea goes back to the work of Cieliebak--Mohnke \cite{CM1,CM2}, who considered degree one curves in $\CP^n$ satisfying such a local tangency constraint and used a neck-stretching argument to prove the Audin conjecture.
As explained in \cite{HSC}, unlike curves satisfying generic point constraints, curves with a local tangency constraint enjoy nice dimensional stability properties which 
makes them particularly relevant to Problem~\ref{problem:SEEP}.

As for (2), recall that in closed symplectic manifolds we can use moduli spaces of closed curves to define Gromov--Witten invariants, and in favorable cases these do enumerate honest curves. The key point in Gromov--Witten theory is that the relevant moduli spaces admit codimension two compactifications, which means we can hope to avoid boundary phenomena in one-parameter families.
In contrast, the SFT compactness theorem provides a codimension {\em one} compactification of $\calM$, 
and we typically encounter nontrivial pseudoholomorphic buildings in one-parameter families.
In such a situation we should generally replace counting invariants with homological invariants, and this is the purview of symplectic field theory \cite{EGH2000}.
At present we are interested in genus zero curves, so we are in the  setting of rational symplectic field theory (see \cite[\S 3,\S 5]{HSC}).
There is also an alternative approach to these invariants using Floer theory (see e.g. \cite[Rmk. 3.9]{HSC} and the references therein), although the full details have not yet appeared in the literature.

\sss

In the special case of an ellipsoid $X = E(a_1,\dots,a_n)$, the differentials and higher operations involved in the above homological invariants all vanish identically for degree parity reasons.
Moreover, it is proved in \cite{McDuffSiegel_counting} that a local tangency constraint can be replaced by removing a small neighborhood symplectomorphic to an infinitely skinny ellipsoid $E_\sk$, and then considering punctured curves in the resulting cobordism with an additional negative puncture.
We are thus in the framework of the following problem.
\begin{problem}[ellipsoidal cobordism curve counting problem]\label{problem:SBCP}
What is the count $\#\calM^J_{E(a_1',\dots,a_n') \setminus \eps E(a_1,\dots,a_n)}(\Gamma^+;\Gamma^-)$?
\end{problem}
\NI Let us explain this notation. Given $a_1,\dots, a_n$ and $a_1',\dots,a_n'$, let $\eps > 0$ be small enough that we have an inclusion $\eps E(a_1,\dots,a_n) \subset E(a_1',\dots,a_n')$. Let 
$E(a_1',\dots,a_n') \setminus \eps E(a_1,\dots,a_n)$ denote the corresponding complementary cobordism, and let $J$ be a generic SFT-admissible almost complex structure on  $E(a_1',\dots,a_n') \setminus \eps E(a_1,\dots,a_n)$.
For collections of Reeb orbits $\Gamma^+ = (\gamma^+_1,\dots,\gamma^+_k)$ in $\bdy E(a_1',\dots,a_n')$ and $\Gamma^- = (\gamma^-_1,\dots,\gamma^-_l)$ in $\bdy E(a_1,\dots,a_n)$, let $\calM^J_{E(a_1',\dots,a_n') \setminus \eps E(a_1,\dots,a_n)}(\Gamma^+;\Gamma^-)$ denote 
the moduli of genus zero punctured curves in $E(a_1',\dots,a_n') \setminus \eps E(a_1,\dots,a_n)$ with positive punctures asymptotic to $\Gamma^+$ and negative punctures asymptotic to $\Gamma^-$.
Assuming $\ind\;\calM = 0$, we then denote by $\#\calM^J_{E(a_1',\dots,a_n') \setminus \eps E(a_1,\dots,a_n)}(\Gamma^+;\Gamma^-)$ the signed count of curves in this moduli space.

Even after restricting to ellipsoids, Problem~\ref{problem:SBCP} is still not entirely well-defined. 
For one thing, the presence of index zero branched covers of trivial cylinders in the symplectizations of $\bdy E(a_1,\dots a_n)$ and $\bdy E(a_1',\dots,a_n')$ introduces certain ambiguities and prevents curve counts from being independent of $J$ and $\eps$. More severely, for arbitrary $\Gamma^+$ and $\Gamma^-$, curves of negative index can appear due to lack of transversality, which precludes any hope of naive curve counting.

\sss

Nevertheless, Problem~\ref{problem:SBCP} is not vacuous, as there are many cases in which we do get well-defined enumerative invariants.
Here is a first example:
\begin{example}\label{ex:well_def_count}
Consider the slightly perturbed four-ball $E(1,1+\delta)$ for $\delta > 0$ irrational and sufficiently small,\footnote{We could also take $a_1 = a_2 = 1$, after taking into account the necessary Morse--Bott modifications.} and let $\gamma_\sht$ and $\gamma_\lng$ denote the short and long simple Reeb orbits respectively of $\bdy E(1,1+\delta)$.
Put $E_\sk := E(1,x)$ for $x \gg 1$ sufficiently large,
and let $\eta_k$ denote the $k$-fold cover of the short simple Reeb orbit of $E_\sk$.
For $d \in \Z_{\geq 1}$, we take $\Gamma^+ = \underbrace{(\gamma_\lng,\dots,\gamma_\lng)}_d$ and 
$\Gamma^- = (\eta_{3d-1})$.
Then $\#\calM^J_{E(1,1+\delta) \setminus \eps E_\sk}(\Gamma^+;\Gamma^-)$ is finite and independent of $\delta,\eps$, and generic $J$.
\end{example}

In fact, in the above example, 
$\frac{1}{d!(3d-1)}\#\calM^J_{E(1,1+\delta) \setminus \eps E_\sk}$
 precisely coincides with the count $T_d$ of degree $d$ rational curves in $\CP^2$ satisfying a local tangency constraint of contact order $3d-1$.\footnote{The extra combinatorial factor $\tfrac{1}{(3d-1)d!}$ has to do with our conventions for handling asymptotic markers and orderings of punctures - see \S\ref{subsec:well_def_counts}.}
These counts were recently computed for all $d$ in \cite{McDuffSiegel_counting}, giving $T_1 = 1$, $T_2 = 1$, $T_3 = 4$, $T_4 = 26$, $T_5 = 127$, and so on.
The computation in \cite{McDuffSiegel_counting} is based on a recursive formula which reduces these counts to blowup Gromov--Witten invariants of $\CP^2$, which can in turn be computed e.g. by \cite{GP}. 

Example~\ref{ex:well_def_count} is a special case of the following, which appears as \cite[Prop. 3.3.6]{Ghost}
(see also \S\ref{subsec:well_def_counts} for further discussion and extensions):
\begin{prop}\label{prop:Mint_counts}
For $x \in \R_{> 1}$, assume that we have $3d - 1= k + \lfloor k/x\rfloor$ and that we cannot find decompositions $d = \sum_{i=1}^m d_i$ and $k = \sum_{i=1}^m k_i$
for $m \in \Z_{\geq 2}$ and $d_1,\dots,d_m \in \Z_{\geq 1}$ such that $$3d_i - 1 = k_i + \lfloor k_i/x \rfloor$$ for $i = 1,\dots,m$.
Put $\Gamma^+ = \underbrace{(\gamma_\lng,\dots,\gamma_\lng)}_d$ and 
$\Gamma^- = (\gamma_{\sht;k})$, where $\gamma_{\sht;k}$ denotes the $k$-fold cover of $\gamma_\sht$. 
Then $\#\calM^J_{E(1,1+) \setminus \eps E(1,x)}(\Gamma^+;\Gamma^-)$ is finite and independent of $\eps$ and generic $J$.
\end{prop}
\begin{remark}
The count in Proposition~\ref{prop:Mint_counts} is equivalent to the count of degree $d$ rational curves in $\CP^2 \setminus \eps E(1,x)$ with one negative puncture asymptotic to $\gamma_{\sht;k}$.
Note that the condition $3d-1 = k + \lfloor k/x \rfloor$ is equivalent to having $\ind\; \calM^J_{E(1,1+)\setminus \eps E(1,x)}(\Gamma^+;\Gamma^-) = 0$.
\end{remark}

\begin{definition}
In the context of Proposition~\ref{prop:Mint_counts}, we put $$T_{d;1,x} := \tfrac{1}{d!k}\#\calM^J_{E(1,1+) \setminus \eps E(1,x)}(\Gamma^+;\Gamma^-).$$
\end{definition}
\noindent As far as we are aware, these counts have not previously been computed except for some special cases.
Note that we have $T_d = T_{d;1,x}$ for $x \gg d$.

\subsubsection{Obstructions from curves}

We now elaborate on the connection between (a) and (b).
Given $2n$-dimensional Liouville domains $X$ and $X'$, the basic strategy for obstructing symplectic embeddings
$X \shookrightarrow X'$ is as follows:
\begin{enumerate}
\item use a ``homological framework'' to argue that if such an embedding existed, there would have to be a punctured curve $u$ in $X' \setminus X$ with some predetermined positive Reeb orbit asymptotics $\Gamma^+ = (\gamma_1^+,\dots \gamma^+_{k})$ and negative Reeb orbit asymptotics $\Gamma^- = (\gamma_1^-,\dots,\gamma_{l}^-)$
\item
apply Stokes' theorem together with nonnegativity of energy to get an inequality of the form $0 \leq E(u) = \sum_{i=1}^{k}\calA(\gamma_i^+) - \sum_{j=1}^{l}\calA(\gamma_j^-)$
\end{enumerate}
We refer the reader to \cite{HSC} for the precise definition of energy and so on.
Roughly speaking, ``homological framework'' will be formalized using the following perspective:
\begin{itemize}
\item
cylinders in $X' \setminus X$ are encoded using action-filtered linearized contact homology $\chlin(X)$ (or alternatively action-filtered positive $S^1$-equivariant symplectic cochains $\sc_{S^1,+}(X)$), and these give the same\footnote{Strictly speaking it is an open conjecture that these obstructions coincide with those defined by Ekeland--Hofer. See \cite[Conj. 1.9]{Gutt-Hu} for evidence of this conjecture and more details. Note that we will not make any use of the original definition of the Ekeland--Hofer capacities in this paper.} obstructions as the Ekeland--Hofer capacities
\item
spheres in $X' \setminus X$ with several positive ends and one negative end are encoded using the action-filtered $\Li$ structure on $\chlin(X)$ 
(or alternatively on $\sc_{S^1,+}(X)$), and these give the obstructions from \cite{HSC}.\footnote{See Remark~\ref{rmk:on_sft_trans} for a discussion of our transversality assumptions.} 
\end{itemize}
A folklore question asks whether all nonclassical symplectic embedding obstructions are given by some pseudoholomorphic curve as in the above strategy. In principle such a curve could have higher genus and/or more than one negative end\footnote{Here we are excluding {\em anchors} (see e.g. \cite[\S 3]{HSC}), which behave essentially differently from negative ends.}, necessitating a more refined homological framework such as higher genus SFT. However, we do not know of any framework for defining dimensionally stable obstructions which involves such curves (see the discussion in \cite[\S5.4]{HSC}).

\sss

In order to implement the above strategy, we need to compute the filtered $\Li$ algebras $\chlin(X)$ and $\chlin(X')$.
Following \cite{HSC}, we can then try to read off obstructions using their bar complex spectral invariants (this is reviewed in \S\ref{sec:emb_obstrs}).
However, we also need a canonical way of referencing homology classes in $\chlin(X)$ and $\chlin(X')$.
In \cite{HSC}, local tangency constraints accomplish this task, and the capacities $\gapac_\bb(X)$ and $\gapac_\bb(X')$ give the corresponding ``coordinate-free'' bar complex spectral invariants of $X$ and $X'$ respectively (this is reviewed in \S\ref{subsec:spectral_capacities} below).
Alternatively, if we can compute the filtered $\Li$ homomorphism $\chlin(X') \rightarrow \chlin(X)$ induced by the complementary cobordism $X' \setminus X$, 
we can read off obstructions directly via Stokes' theorem, since any homologically nontrivial structure coefficient must be represented by some curve or building.

In the special case of ellipsoids $X = E(a_1,\dots,a_n)$ and $X' = E(a_1',\dots,a_n')$, the filtered $\Li$ algebras $\chlin(X)$ and $\chlin(X')$ are trivial to compute, since the differentials and all higher $\Li$ operations vanish for degree parity reasons.
However, computing the cobordism map $\chlin(X') \rightarrow \chlin(X)$ essentially amounts to counting all punctured spheres in $X' \setminus X$ with several positive ends and one negative end, and this is an intricate enumerative problem, a special case of Problem~\ref{problem:SBCP}.
For instance, as pointed out by McDuff, Conjecture~\ref{conj:stab_ell} would follow from the existence of some very specific curves:
\begin{thm}[\cite{Mint}]\label{lem:suff_cond_for_RSEP}
Let $x = p/q$ for $p,q,d \in \Z_{\geq 1}$ with $p + q = 3d$ and $\gcd(p,q)  =1$.
If $T_{d;1,x} \neq 0$, then we have $f_N(x) \geq \frac{3x}{x+1}$ for all $N \in \Z_{\geq 0}$.
\end{thm}
\begin{remark}
One can also check that numbers of the form $p/q$ as in Lemma~\ref{lem:suff_cond_for_RSEP} are dense in $\R_{\geq 1}$, and hence are sufficient to prove Conjecture~\ref{conj:stab_ell} for all $x$.
\end{remark}

\subsection{From geometry to algebra}

We first recall the class of convex toric domains.
Let $\mu: \C^n \rightarrow \R_{\geq 0}^n$ denote the moment map for the standard $\mathbb{T}^n$ action on $\C^n$,
given explicitly by 
$$
\mu(z_1,\dots,z_n) = (\pi |z_1|^2,\dots,\pi |z_n|^2).
$$
Note that the fiber $\mu^{-1}(p)$ over a point $p \in \R^n_{>0}$ is a smooth $n$-dimensional torus, while the fiber over a point $p \in \R_{\geq 0} \setminus \R_{> 0}$ is a torus of strictly lower dimension. Following \cite{choi2014symplectic,hutchings2016beyond,Gutt-Hu}, we make the following definition:
\begin{definition}\label{def:cvt}
A {\bf convex toric domain} is a subdomain of $\C^n$ of the form $X_{\Omega} := \mu^{-1}(\Omega)$, where $\Omega \subset \R_{\geq 0}$ is a subset
 such that $$\wh{\Omega} := \{(x_1,\dots, x_n) \in \R^n \; : \; (|x_1|,\dots,|x_n|) \in \Omega\} \subset \R^n$$ is compact and convex.
\end{definition}
\noindent This class includes the following examples:
\begin{itemize}
\item
the ellipsoid $E(a_1,\dots,a_n)$ is of the form $X_{\Omega_{E(a_1,\dots,a_n)}}$ with moment map image 
$$\Omega_{E(a_1,\dots,a_n)} := \op{Conv} \left((0,\dots,0),(a_1,0,\dots,0),(0,\dots,0,a_n)\right) \subset \R_{\geq 0}^n$$
\item
the polydisk $P(a_1,\dots,a_n) := B^2(a_1) \times \dots \times B^2(a_n)$ is of the form $X_{\Omega_{P(a_1,\dots,a_n)}}$
with moment map image
$$\Omega_{P(a_1,\dots,a_n)} := [0,a_1] \times \dots \times [0,a_n] \subset \R_{\geq 0}^n.$$
\end{itemize}
\noindent Compared to arbitrary Liouville domains, the extra torus symmetry present for convex toric domains makes their Reeb dynamics and pseudoholomorphic curve moduli spaces more amenable to analysis.
We can also view convex toric domains as partial compactifications of smooth Lagrangian torus fibrations, making them natural objects of study in mirror symmetry (see \S\ref{subsec:ms} below).

\sss

In \S\ref{sec:family}, we define an explicit filtered $\Li$ algebra $V^{2n}_\Omega$ associated to any subset $\Omega \subset \R_{\geq 0}^n$ for which $X^{2n}_\Omega$ is a convex toric domain. More precisely, $V^{2n}_{\Omega}$ is a certain differential graded Lie algebra $V^{2n}$ which is independent of $\Omega$, equipped with an $\Omega$-dependent filtration.
 If $X^{2n}_{\Omega'}$ is another convex toric domain associated to a subset $\Omega' \subset \R_{\geq 0}^n$, by slight abuse of notation we take the ``identity map'' $\id: V_{\Omega'} \rightarrow V_{\Omega}$ to be the (possibly unfiltered) $\Li$ homomorphism sending each generator of $V_{\Omega'}$ to the corresponding generator in $V_{\Omega}$.
The following theorem provides a complete algebraic model for the filtered $\Li$ algebra $\chlin(X_\Omega)$ and for the filtered $\Li$ homomorphism $\Xi: \chlin(X_{\Omega'}) \rightarrow \chlin(X_\Omega)$ induced by a symplectic embedding $X_\Omega \shookrightarrow X_{\Omega'}$: 
\begin{thm}[\cite{chscII}]\label{thm:chscII_main}
Let $X_{\Omega}$ and $X_{\Omega'}$ be $2n$-dimensional convex toric domains, and suppose there is a symplectic embedding $X_{\Omega} \times \C^{N} \shookrightarrow X_{\Omega'} \times \C^{N}$ for some $N \geq 0$.
Then there exist inverse filtered $\Li$ homotopy equivalences $F_{\Omega}: V_{\Omega} \rightarrow \chlin(X_\Omega)$ and $G_{\Omega'}: \chlin(X_{\Omega'}) \rightarrow V_{\Omega'}$ such that $\Xi$ is unfiltered $\Li$ homotopic to $F_{\Omega} \circ \id \circ G_{\Omega'}$.
\end{thm}
\noindent It will also be convenient to formulate the following ``model-independent'' version:
\begin{cor}\label{cor:model_indep}
Let $X_{\Omega}$ and $X_{\Omega'}$ be $2n$-dimensional convex toric domains, and suppose there is a symplectic embedding $X_{\Omega} \times \C^{N} \shookrightarrow X_{\Omega'} \times \C^{N}$ for some $N \geq 0$.
Then there exists a filtered $\Li$ homomorphism
$Q: V_{\Omega'} \rightarrow V_\Omega$ which is unfiltered $\Li$ homotopic to the identity.
\end{cor}

Our symplectic embedding obstructions will follow from Corollary~\ref{cor:model_indep} by applying a kind of ``filtered $\Li$ calculus''.
In \S\ref{sec:V_can}, we construct a canonical model for the unfiltered $\Li$ algebra $V_\Omega$.
Namely, we put $V^\can_\Omega := H(V_\Omega)$, and we recursively construct inverse $\Li$ homotopy equivalences
$\Phi_\Omega: V_\Omega \rightarrow V^\can_\Omega$ and $\Psi_\Omega: V^\can_\Omega \rightarrow V_\Omega$.
In \S\ref{sec:emb_obstrs}, we compute the homology of the bar complex of $V_\Omega$ and explain how to extract the capacities $\gapac_\bb$ from \cite{HSC}.
In the special case of a four-dimensional ellipsoid, we put 
$V_{a,b} := V_{\Omega_{E(a,b)}}$, $\Phi_{a,b} := \Phi_{\Omega_{E(a,b)}}$, and $\Psi_{a,b} := \Psi_{\Omega_{E(a,b)}}$, and we show that $V^{a,b}$ is a canonical model for the {\em filtered} $\Li$ algebra $V_{a,b}$.
In \S\ref{sec:enum_imps}, we prove that the combinatorially defined maps $\Phi_{a,b}$ and $\Psi_{a,b}$ can be used to compute enumerative invariants.

\subsection{Main results}

\subsubsection{Symplectic embeddings}

It is conjectured in \cite{HSC} that the capacities $\gapac_\bb$ give a complete set of obstructions for Problem~\ref{problem:SEEP}.
For the restricted version, namely Problem~\ref{prob:res_stab_ell}, the following corollary gives a purely combinatorial criterion.
We note that there is a natural isomorphism of $\K$-modules $V^\can_\Omega \cong \K\langle A_1,A_2,A_3,\dots\rangle$, where there at most one generator $A_i$ in each degree (see \S\ref{subsec:homology_V}).
The following theorem is proved in \S\ref{subsec:ellipsoids}:
\begin{thm}\label{thm:a_b_a'_b'_ineq}
Fix $a,b,a',b' \in \R_{> 0}$, and suppose that we have
$$\langle (\Phi_{a,b} \circ \id \circ \Psi_{a',b'})^k(A_{i_1},\dots,A_{i_k}),A_{i_1 + \dots + i_k + k - 1}\rangle \neq 0$$
for some $k,i_1,\dots,i_k \in \Z_{\geq 1}$.
Then if there exists a symplectic embedding
$E(a,b) \times \C^N \shookrightarrow E(a',b') \times \C^N$ for some $N \in \Z_{\geq 0}$, we must have:
\begin{align}\label{ineq:comb_crit}
\sum_{j=1}^k\min_{\substack{s+t=i_j\\ s,t \in \Z_{\geq 0}}}\max\{a's,b't\} \geq \min_{\substack{s+t=i_1+\dots + i_k + k-1\\ s,t \in \Z_{\geq 0}}}\max\{as,bt\}.
\end{align}
\end{thm}
\NI 
A geometric interpretation of the above expression is initiated in \S\ref{subsec:homology_V}, where it is observed that $\min\limits_{i+j=k}\max\{ia,jb\}$ is equal to the $k$th smallest action of a Reeb orbit in $\bdy E(a,b)$.
For example, in the case $E(a',b') = E(c,c+\delta)$ for $\delta > 0$ sufficiently small and $i_1 = \dots = i_k = 2$,
the left hand side of \eqref{ineq:comb_crit} becomes $kc(1+\delta)$.
Meanwhile, in the case $E(a,b) = E(1,p/q+\delta')$ for $\delta' > 0$ sufficiently small, with $p,q  \in \Z_{\geq 1}$ satisfying $p + q = 3k$, the right hand side of \eqref{ineq:comb_crit} becomes $p$, with the minimum occurring for $(s,t) = (p,q-1)$.
With these choices, the inequality \eqref{ineq:comb_crit} amounts to $k(c+\delta) \geq p$, or equivalently $c+\delta \geq \tfrac{3x}{x+1}$ for $x = p/q$. This leads to:
\begin{cor}\label{cor:comb_crit_for_rseep}
Suppose we have $x = p/q$ with $p,q,d \in \Z_{\geq 1}$ such that $p + q = 3d$ and $x \geq \tau^4$. Then Conjecture~\ref{conj:stab_ell} holds at the value $x$ provided that we have
$$ \langle (\Phi_{1,x} \circ \Psi_{1,1})(\odot^d A_2),A_{3d-1}\rangle \neq 0.$$
\end{cor}
\noindent We do not prove in this paper that the criterion in Corollary~\ref{cor:comb_crit_for_rseep} holds in general. However, for any given $x$ the relevant structure coefficients can easily be computed with the aid of a computer.\footnote{A Python implementation is available on the author's website.}
For example, by the sample computations in \S\ref{sec:enum_imps} we have:
\begin{cor}
Conjecture~\ref{conj:stab_ell} holds for each of the $x$ values appearing in Table~\ref{table:pq}.
\end{cor} 

\sss

For Liouville domains which are not necessarily ellipsoids, it turns out that we can sometimes extract stronger obstructions from Corollary~\ref{cor:model_indep} than those visible to the capacities $\gapac_\bb$.
As observed by Hutchings \cite{hutchings2016beyond}, a phenomenon which is similar in spirit occurs for ECH capacities.
In \S\ref{subsec:beyond}, we illustrate this phenomenon with the two following examples, which are (are far as we are aware) new for $a < 2$:
\begin{thm}\label{prop:poly_into_cube}
Given a symplectic embedding $P(1,a)\times \C^{N} \shookrightarrow P(c,c) \times \C^{N}$ for $a \geq 1$ and $N \geq 0$, we must have
$c \geq \min(a,2)$.
Moreover, this is sharp for $N \geq 1$.
\end{thm}

\begin{thm}\label{prop:poly_into_ball}
Given a symplectic embedding $P(1,a) \times \C^{N} \shookrightarrow B^4(c) \times \C^{N}$ for $a \geq 1$ and $N \geq 0$, we must have
$c \geq \min(a+1,3)$. Moreover, this is sharp for $N \geq 1$.
\end{thm}

\subsubsection{Enumerative geometry}

Let $E(a',b') \setminus E(a,b)$ be a cobordism between two four-dimensional ellipsoids.
In principle, we can now read off curve counts in this cobordism in using the structure coefficients of the induced $\Li$ map $\Xi: \chlin(E(a',b')) \rightarrow \chlin(E(a,b))$.
A priori, the filtered $\Li$ homomorphisms $F_\Omega$ and $G_{\Omega'}$ given by Theorem~\ref{thm:chscII_main}  are inexplicit, arising from certain auxiliary SFT cobordism maps.
In \S\ref{sec:enum_imps} we characterize these maps using techniques from embedded contact homology, proving the following result:
\begin{thm}
For $X_\Omega = E(a,b)$ and $X_{\Omega'} = E(a',b')$, we can take
$F_{\Omega} = \Phi_{a,b}$ and $G_{\Omega'} = \Psi_{a',b'}$ in Theorem~\ref{thm:chscII_main}.
\end{thm}
\begin{definition}
In the context of Proposition~\ref{prop:Mint_counts}, we put
$$S_{d;1,x} := \tfrac{1}{d!k}\langle \Phi_{1,x} \circ \Psi_{1,1}(\odot^d A_2),A_{3d-1}\rangle.$$
\end{definition}
\begin{cor}\label{cor:T_equals_S}
We have $T_{d;1,x} = S_{d;1,x}$.
\end{cor}
\begin{example}
Using Corollary~\ref{cor:T_equals_S} and the recursive construction of $\Psi_{a,b}$ given in \S\ref{subsec:phi_and_psi},
we get a recursive formula for the numbers $S_d$ which is completely different from 
(and much simpler than) the recursive algorithm for $T_d$ given in \cite{McDuffSiegel_counting}.
With the aid of a computer, we have independently verified Corollary~\ref{cor:T_equals_S} for $d = 1,\dots,9$. Putting $S_{d} := S_{d;1,x}$ for $x$ sufficiently large, we have:
\[
\footnotesize
\begin{array}{llll}
S_{1} = 1, &\;\;
S_{2} = 1, &\;\;
S_{3} = 4, &\;\;
S_{4} = 26\\
S_{5} = 217,&\;\;
S_{6} = 2110,&\;\;
S_{7} = 22744, &\;\;
S_{8} = 264057,\\
S_{9} = 3242395, &\;\;
S_{10} = 41596252,&\;\;
S_{11} = 552733376,&\;\;
S_{12} = 7559811021,\\
S_{13} = 105919629403, &\;\;
S_{14} = 1514674166755., &\;\;
S_{15} = 22043665219240, &\;\;
S_{16} = 325734154669786,
\end{array}
\]
and so on.
\end{example}
\begin{example}\label{ex:new_seep_cases}
The paper \cite{Ghost} proves the restricted stabilized ellipsoid conjecture for $x = 55/8$ by showing (in our notation) that
$T_{21;1,55/8} \geq 3$.
By a computer calculation using Corollary~\ref{cor:T_equals_S}, we get precisely $S_{21;1,55/8} = 3$.
See Table~\ref{table:pq} for many more computations of this nature.
\end{example}
\begin{disclaimer}[On SFT transversality]\label{rmk:on_sft_trans}
In general, as in \cite{HSC}, we work in a suitable virtual perturbation framework in order to define the above symplectic field theoretic invariants, without invoking any specific properties of the particular scheme used (see e.g. \cite[Rmk. 3.1]{HSC}). In fact, for the enumerative invariants discussed in \S\ref{sec:enum_imps}, the relevant moduli spaces are regular for any generic choice of admissible almost complex structure, and hence are counted in the classical sense. 
Moreover, thanks to the favorable Conley--Zehnder index behavior of Reeb orbits in a fully rounded convex toric domain (see \S\ref{subsec:rounding}), most of the moduli spaces involved in the proof of Theorem~\ref{thm:chscII_main} can made made regular within a classical perturbation framework (see \cite{chscII} for details).
By contrast, in the absence of virtual perturbations, the naive moduli spaces involved in e.g. the cobordism map $\Xi: \chlin(E(a',b')) \rightarrow \chlin(E(a,b))$ are often far from regular. 
\end{disclaimer}

\section*{Acknowledgements}
{
I am highly grateful to Dan Cristofaro-Gardiner and Dusa McDuff for 
their input and interest in this project. I also thank Mohammed Abouzaid for helpful discussions.
}

\section{A family of filtered $\Li$ algebras}\label{sec:family}

In this section, after recalling some background on $\Li$ algebras and setting up notation in \S\ref{subsec:recollections}, we define the DGLA $V$ in \S\ref{subsec:DGLA_V}, and endow it with its family of filtrations in \S\ref{subsec:V_a_b}. Lastly, as a prelude to \S\ref{sec:V_can}, in \S\ref{subsec:homology_V} we compute the linear spectral invariants of $V_\Omega$ and show that they recover the Ekeland--Hofer capacities of $X_\Omega$.

\subsection{$\Li$ recollections}\label{subsec:recollections}

Here we briefly recall some basic notions about $\Li$ algebras in order to set our conventions for signs, gradings, and filtrations. We refer the reader to \cite[\S2]{HSC} and the references therein for more details.

Let $\K$ be a fixed field containing $\Q$, which we will usually take to be $\Q$ itself.
Let $V$ be a $\Z$-graded $\K$-module. 
For $k \in \Z_{\geq 1}$, let $\otimes^k V$ denote the $k$-fold tensor product (over $\K$) of $V$, and 
let $\odot^k V = \otimes^kV / \Sigma^k$ denote the $k$-fold symmetric tensor product of $V$, i.e. we quotient by the signed action of the permutation group.
For an elementary tensor $v_1 \otimes \dots \otimes v_k \in \otimes^k V$, we will denote its image in $\odot^k V$ by $v_1 \odot \dots \odot v_k$, and the signs of the permutation action are such that permuting adjacent elements $v,v'$ ``costs'' the sign $(-1)^{|v| |v'|}$, e.g. we have
\begin{align*}
v_1 \odot v_2 \odot v_3 \odot v_4 = (-1)^{|v_2| |v_3|} v_1 \odot v_3 \odot v_2 \odot v_4.
\end{align*}
Let $sV$ denote the graded $\K$ module given by shifting the gradings of $V$ down by one.

Let $\ovl{S}V = \sum_{i=1}^\infty \odot^i V$ denote the (reduced) symmetric tensor coalgebra on $V$,
where the coproduct is given by
\begin{align*}
\Delta(v_1 \odot ... \odot v_k) := \sum_{i=1}^{k-1}\sum_{\sigma \in \Sh(i,k-i)}\sign(\sigma,V;v_1\dots,v_k)(v_{\sigma(1)} \odot \dots \odot v_{\sigma(i)}) \otimes (v_{\sigma(i+1)} \odot ... \odot v_{\sigma(k)}).
\end{align*}
Here $\Sh(i,k-i)$ denotes the subset of permutations $\sigma \in \Sigma_k$ satisfying 
$\sigma(1) < ... < \sigma(i)$ and $\sigma(i+1) < \dots < \sigma(k)$,
and the Koszul-type signs are defined by
$$\sign(\sigma,V;v_1,\dots,v_n) = (-1)^{\{|v_i| |v_j|\;:\; 1 \leq i < j \leq n,\; \sigma(i) > \sigma(j)\}}.$$

An $\Li$ structure on $V$ is typically defined to be a coderivation $\wh{\ell}: \ovl{S}(sV) \rightarrow \ovl{S}(sV)$ of degree $+1$ satisfying $\wh{\ell} \circ \wh{\ell} = 0$.
This is equivalent to a sequence of degree $+1$ graded symmetric maps $\ell^k: \otimes^k (sV) \rightarrow sV$ (or alternatively, a sequence of graded skew-symmetric maps\footnote{Note that we will also implicitly identify maps $\odot^k V \rightarrow V$ with multilinear maps with $k$ inputs and one output in $V$.} $\otimes^k V \rightarrow V$) for $k \in \Z_{\geq 1}$ satisfying the $\Li$ structure equations 
\begin{align*}
\sum_{k=1}^n\sum_{\sigma \in \Sh(k,n-k)}\sign(\sigma,sV;v_1,\dots,v_n)\ell^{n-k+1}\left(\ell^k(v_{\sigma(1)}\odot \dots \odot v_{\sigma(k)}) \odot v_{\sigma(k+1)} \odot \dots \odot v_{\sigma(n)}\right) = 0
\end{align*}
for any collection of inputs $v_1,\dots,v_n \in sV$.

However, in order to mostly suppress the grading shifts from the notation, by slight abuse of standard terminology we will work with the following modified definition:
\begin{definition}
An $\Li$ algebra is a graded $\K$-module $V$ together with a coderivation $\wh{\ell}: \ovl{S}V \rightarrow \ovl{S}V$ of degree $+1$ such that $\wh{\ell} \circ \wh{\ell} = 0$.
\end{definition}
\noindent With this convention, the {\bf bar complex} $\bar V$ is by definition simply the chain complex $(\ovl{S}V,\wh{\ell})$. 
Given the maps $\ell^1,\ell^2,\ell^3,\dots$, we recover the bar complex differential $\wh{\ell}$ via the extension formula
\begin{align*}
 \wh{\ell}(v_1 \odot \dots \odot v_n) &:= \sum_{k=1}^n\sum_{\sigma \in \Sh(k,n-k)}\sign(\sigma,V;v_1,\dots,v_n)\ell^k(v_{\sigma(1)}\odot \dots \odot v_{\sigma(k)}) \odot v_{\sigma(k+1)} \odot \dots \odot v_{\sigma(n)}.
 \end{align*}

Similarly:
\begin{definition}\label{def:L-inf_homo}
 An $\Li$ homomorphism $\Phi: V \rightarrow W$ between $\Li$ algebras $V$ and $W$ is by definition a degree $0$ coalgebra map $\wh{\Phi}: \ovl{S}V \rightarrow \ovl{S}W$ such that $\wh{\ell}_W \circ \wh{\Phi} = \wh{\Phi} \circ \wh{\ell}_V$. 
\end{definition}
\NI This can alternatively be described by a sequence of degree $0$ graded symmetric maps $\Phi^k: \otimes^k V \rightarrow W$ for $k \in \Z_{\geq 1}$ satisfying the $\Li$ homomorphism equations, and we recover $\wh{\Phi}$ from the maps $\Phi^1,\Phi^2,\Phi^3,\dots$ via the extension formula
$$\wh{\Phi}(v_1 \odot \dots \odot v_n) := \sum_{\substack{k \geq 1\\ i_1 + \dots + i_k = n}}\sum_{\sigma \in \Sh(n;i_1,\dots,i_k)} \sign(\sigma,V;v_1, \dots ,v_n)(\Phi^{i_1} \odot \dots \odot \Phi^{i_k})(v_{\sigma(1)}\odot \dots \odot v_{\sigma(n)}).$$
Here $\Sh(n;i_1,\dots,i_k)$ denotes subset of permutations $\sigma \in \Sigma_n$ satisfying
$\sigma(1) < \dots < \sigma(i_1)$, $\sigma(i_1+1) < \dots < \sigma(i_1 + i_2)$, $\dots$, $\sigma(i_1+\dots+i_{k-1}+1) < \dots < \sigma(i_1 + \dots + i_k)$.
Given an $\Li$ homomorphism $\Phi: V \rightarrow W$, we will switch freely between its representation as a sequence of maps $\Phi^1,\Phi^2,\Phi^3,\dots$ and its representation as a chain map $\wh{\Phi}: \bar  V \rightarrow \bar W$.
Similarly, a chain homotopy between two $\Li$ homomorphisms $\Phi,\Psi: V \rightarrow W$ is defined such that there is an induced chain homotopy between the chain maps $\wh{\Phi},\wh{\Psi}: \bar V \rightarrow \bar 
W$ (see \cite[\S2.1.3]{HSC}).
\begin{remark}
Most of the $\Li$ algebras appearing in this paper will in fact be differential graded Lie algebras (DGLAs),\footnote{See \cite[Rmk. 2.6]{HSC} for the relationship to typical DGLA grading and sign conventions.}
i.e. $\ell^k \equiv 0$ for $k \geq 3$.
In this case we will often use $\bdy$ to denote the differential $\ell^1(-)$ and $[-,-]$ to denote the bracket $\ell^2(-,-)$.
However, the corresponding $\Li$ homomorphisms will nevertheless tend to have infinitely many nonzero terms.
\end{remark}

\subsection{The differential graded Lie algebra $V$}\label{subsec:DGLA_V}

We now introduce our main protagonist, first without any filtration.
For each $n \in \Z_{\geq 1}$, we will define a differential graded Lie algebra (DGLA) $V^{2n}$ over $\K$.
According to Theorem~\ref{thm:chscII_main}, $V^{2n}$ is an $\Li$ model for $\chlin(X)$ (or alternatively for $\sc_{S^1,+}(X)$) when $X$ is a $2n$-dimensional convex toric domain in $\C^n$. For ease of exposition we will mostly focus on the case $n =2$, and by default we put $V = V^{4}$; the higher dimensional analogues of $V$ and $V_\Omega$ will be described in \S\ref{subsec:ms}.

\begin{definition}\label{def:V} As a $\K$-module, the DGLA $V$ has generators:
\begin{itemize}
\item
$\alpha_{i,j}$ for each $i,j \in \Z_{\geq 1}$, of degree $|\alpha_{i,j}| = -1-2i-2j$
\item 
$\beta_{i,j}$ for each $i,j \in \Z_{\geq 0}$ not both $0$, of degree $|\beta_{i,j}| = -2-2i-2j$.
\end{itemize}
The differential is of the form:
\begin{itemize}
\item $\bdy \alpha_{i,j} = j\beta_{i-1,j} -  i\beta_{i,j-1}$ 
\item $\bdy \beta_{i,j} = 0$.
\end{itemize}
The bracket is given by:
\begin{itemize}
\item $[\alpha_{i,j},\alpha_{k,l}] = (il-jk)\alpha_{i+k,j+l}$
\item $[\alpha_{i,j},\beta_{k,l}] = [\beta_{k,l},\alpha_{i,j}] = (il-jk)\beta_{i+k,j+l}$
\item $[\beta_{i,j},\beta_{k,l}] = 0$.
\end{itemize}
\end{definition}

According to Theorem~\ref{thm:chscII_main}, $V$ is an $\Li$ model for $\chlin(X)$ (or alternatively $\sc_{S^1,+}(X)$) whenever $X$ is a four-dimensional convex toric domain in $\C^2$. As it turns out, in the unfiltered setting the bracket carries essentially no information. Indeed, by the computation in \S\ref{subsec:homology_V} below, the homology of $V$ is concentrated in even degrees. Then by standard homological perturbation theory techniques, $V$ has a canonical $\Li$ model all of whose operations are trivial (see \S\ref{sec:V_can} for more details). However, the situation will be quite different in the presence of filtrations.

\subsection{The filtered differential graded Lie algebra $V_{\Omega}$}\label{subsec:V_a_b}

We now the equip $\Li$ algebra $V^{2n}$ with a family of filtrations which will give rise to rich combinatorial structures.
For each convex toric domain $X_\Omega \subset \C^n$ (see Definition~\ref{def:cvt}), we define the filtered DGLA $V_\Omega^{2n}$ which after forgetting the filtration is simply $V^{2n}$. 
We again assume by default that $X_\Omega$ is four-dimensional, corresponding to $V_\Omega = V^4_\Omega$.

By a {\em filtration} $\calF$ on $V$ we mean:
\begin{itemize}
\item
submodules $\calF_{\leq r}V \subset \calF_{\leq r'}V \subset V$ for all $0 < r \leq r' < \infty$
\item 
$\ell^k(v_1,\dots,v_k) \in \calF_{\leq r_1 + \dots + r_k}V$ whenever $v_i \in \calF_{\leq r_i}V$ for $i = 1,\dots,k$.
\end{itemize}
We will work primarily in the category of filtered $\Li$ algebras as in \cite[\S2.2]{HSC}.
This means that all structure maps must preserve filtrations, which is a rather strict condition.
For example, a {\em filtered $\Li$ homomorphism} $\Phi: V \rightarrow W$ between filtered $\Li$ algebras satisfies 
\begin{align}
\Phi^k(v_1,\dots,v_k) \in \calF_{\leq r_1 + \dots + r_k}V
\end{align}
 whenever $v_i \in \calF_{\leq r_i}V$ for $i = 1,\dots,k$.	
Similarly, we have a notion of filtered $\Li$ homotopy between filtered $\Li$ homomorphisms, and a corresponding notion of filtered $\Li$ homotopy equivalence between filtered $\Li$ algebras.

We can succinctly define a filtration on $V$ by endowing each basis element $v \in V$ with an ``action'', which we denote by $\calA(v) \in \R_{\geq 0}$.
For a nontrivial $\K$-linear combination of basis elements of $V$, we then put
$$\calA(c_1v_1 + \dots + c_mv_m) := \max\{ \calA(v_i)\;:\; c_i \neq 0\},$$
and we define $\calF_{\leq r}V$ to be the span of all basis elements in $V$ with action at most $r$.\footnote{Note that in terms of action, ``filtration preserving'' really means ``action nonincreasing''.}
In the geometric interpretation of $V_\Omega$ provided by Theorem~\ref{thm:chscII_main}, the basis elements of $V_\Omega$ roughly correspond to Reeb orbits in $\bdy X_{\wt{\Omega}}$, where $X_{\wt{\Omega}}$ is the ``fully rounded'' version of $X_\Omega$ (see \S\ref{subsec:rounding}), and $\calA$ corresponds to the symplectic action functional.

Observe that if $V$ is a filtered $\Li$ algebra, then its bar complex $\bar V$ naturally becomes a filtered chain complex.
Namely, we define the action of an elementary tensor $v_1 \odot \dots \odot v_k \in \bar V$ by
$$\calA(v_1\odot \dots \odot v_k) := \sum_{i=1}^k \calA(v_i).$$ 

\begin{remark}\label{rmk:filt_vers_Nov}
Recall from \cite[\S2.2]{HSC} that there is a close connection between $\Li$ algebras over the universal Novikov ring 
\begin{align*}
\Lamo := \left\{ \sum_{i = 1}^{\infty} c_iT^{a_i}\;:\; c_i \in \K,\; a_i \in \R_{\geq 0},\; \lim_{i \rightarrow \infty}a_i = +\infty\right\}
\end{align*}
and filtered $\Li$ algebras over $\K$ in the above sense.
In particular, given a filtered $\Li$ algebra we can define an $\Li$ algebra over $\Lamo$ by using the filtration to determine the $T$-exponents.
Since this procedure forgets the actions of the generators, we will find it more convenient in this paper to work directly with filtered $\Li$ algebras over $\K$.
Alternatively, we could adopt the $\Li$ augmentation framework of \cite{HSC} in order to recover the lost information.
\end{remark}

\sss

Now let $X_{\Omega} \subset \C^n$ be a convex toric domain with corresponding moment map image $\Omega \subset \R_{\geq 0}^n$. 	
\begin{definition}[\cite{Gutt-Hu}]\label{def:dual_norm}
We define a norm $||-||_\Omega^*$ on $\R^n$ by 
$$||v||_{\Omega}^* := \max\{ \langle v,w\rangle\;:\; w \in \wh{\Omega}\}$$
for $v \in \R^n$.
\end{definition}
\noindent As pointed out in \cite{Gutt-Hu}, if $||-||_\Omega$ denotes the norm on $\R^n$ whose unit ball is $\wh{\Omega}$, then $||-||_\Omega^*$ is the dual norm on $\R^n$ after identifying $(\R^n)^*$ with $\R^n$ via the Euclidean inner product. 

Restricting to the case that $X_{\Omega}$ is a four-dimensional convex toric domain, we now define $V_{\Omega}$ as follows.
\begin{definition}\label{def:the_filtration}
The filtered $\Li$ algebra $V_{\Omega}$ has underlying unfiltered $\Li$ algebra $V$,
and its filtration $\calF_{\Omega}$ determined by the following action values for its generators:
\begin{itemize}
\item $\calA_{\Omega}(\alpha_{i,j}) = ||(i,j)||_{\Omega}^*$ for each $i,j \in \Z_{\geq 1}$
\item $\calA_{\Omega}(\beta_{i,j}) = ||(i,j)||_{\Omega}^*$ for each $i,j \in \Z_{\geq 0}$ not both $0$.
\end{itemize}
\end{definition}
\begin{lemma}
This defines a valid filtered $\Li$ algebra.
\end{lemma}
\begin{proof}
We take for granted that $V$ satisfies the $\Li$ relations, which can be easily checked.
To see that the filtration is valid, we need to check that the differential and bracket preserve the filtration.
It suffices to check that we have
$$\max\{||(i-1,j)||_\Omega^*,||(i,j-1)||_\Omega^*\} \leq ||(i,j)||_\Omega^*$$
and
$$||(i+k,j+l)||_{\Omega}^* \leq ||(i,j)||_\Omega^* + ||(k,l)||_{\Omega}^*.$$
The second inequality follows directly from the triangle inequality.
The first inequality follows after observing that $||-||_\Omega^*$ satisfies the symmetries
$||(x,y)||_\Omega^* = ||(-x,y)||_\Omega^* = ||(x,-y)||_\Omega^*$, and hence we have
$||(x,yt)||_\Omega^* \leq ||(x,y)||_\Omega^*$
and $||(tx,y)||_\Omega^* \leq ||(x,y)||_\Omega^*$ whenever $t \in [0,1]$.
\end{proof}
We will sometimes denote the basis elements of $V_\Omega$ by $\alpha_{i,j}^{\Omega},\beta_{i,j}^\Omega$ if we wish to make explicit which filtration is being used.
Since four-dimensional ellipsoids play a special role in this paper, we also introduce the shorthand
$V_{a,b} := V_{\Omega_{E(a,b)}}$, denoting the corresponding generators by $\alpha_{i,j}^{a,b}$ and $\beta_{i,j}^{a,b}$.

\subsection{The homology of $V_\Omega$ and its linear spectral invariants}\label{subsec:homology_V}

One of our main goals is to extract embedding obstructions from Corollary~\ref{cor:model_indep}.
As a warmup, we consider what happens at the linear level. From the point of view of curves, this corresponds to using cylinders rather than spheres with several positive punctures.
We arrive at the following much weaker statement:
\begin{cor}\label{cor:cyls_4dell}
In the context of Corollary~\ref{cor:model_indep}, the identity map $H(V_{\Omega'}) \rightarrow H(V_{\Omega})$ is filtration preserving.
\end{cor}

\NI As we now explain, from this statement we naturally recover the capacities from \cite{Gutt-Hu}, which conjecturally agree with the Ekeland--Hofer capacities.

Observe that that we do {\em not} necessarily have $\calA(\alpha^{\Omega'}_{i,j}) \geq \calA(\alpha^{\Omega}_{i,j})$ and $\calA(\beta^{\Omega'}_{i,j}) \geq \calA(\beta^{\Omega}_{i,j})$.
Indeed, what we have is a filtered chain map $V_{\Omega'} \rightarrow V_{\Omega}$ which is
unfiltered chain homotopic to the identity, and this chain map is not necessarily the identity on the nose. 
What we do have is the following picture, which is familiar from the study of spectral invariants in symplectic geometry.
For a homology class $A \in H(V_{\Omega})$, we put 
\begin{align}
\calA_{\Omega}(A) := \min \{ \calA(v)\;:\; \bdy v = 0,\; [v] = A\text{ for some } v \in V_{\Omega}\}.
\end{align}
Then in the context of Corollary~\ref{cor:cyls_4dell}, we must have
$$ \calA_{\Omega'}(A) \geq \calA_{\Omega}(A)$$
for any homology class $A \in H(V)$. 
By way of terminology, we will say that $\calA_{\Omega}(A)$ is the {\bf spectral invariant} of $V_{\Omega}$ in the homology class $A$.

The homology of $V$ can be computed as follows. 
For convenience, let us make a simple change of basis by putting, for all $i,j$,
\begin{align*}
\ovl{\alpha}_{i,j} := (i-1)!(j-1)!\alpha_{i,j},\;\;\;\;\; \ovl{\beta}_{i,j} := i!j!\beta_{i,j}.
\end{align*}
In this basis, we have
\begin{align*}
\bdy \ovl{\alpha}_{i,j} =  \ovl{\beta}_{i-1,j} - \ovl{\beta}_{i,j-1},\;\;\;\;\; \bdy \ovl{\beta}_{i,j} = 0.
\end{align*}
From this we see that $H(V)$ is one-dimensional in degrees $-4,-6,-8,\dots$, and trivial otherwise. This means that, for $q \in \Z_{\geq 1}$, the cycles $\ovl{\beta}_{q,0},\ovl{\beta}_{q-1,1},\dots, \ovl{\beta}_{0,q}$ are all homologous and represent the unique nontrivial class $A_q$
(modulo scaling by elements of $\K^*$) in $H^{-2-2q}(V)$.

Now observe that, for $q \in \Z_{\geq 1}$, we have
$$\calA_\Omega(A_q) = \min_{\substack{i,j \in \Z_{\geq 0}\\ i+j=q}}||(i,j)||_\Omega^*.$$
Note that this coincides with the expression for $c_q(X_\Omega)$ from \cite[Thm. 1.6]{Gutt-Hu}.
In the special case of $E(a,b)$, we get the expression
$$\calA_{a,b}(A_q) = \min_{\substack{i,j \in \Z_{\geq 0}\\ i+j=q}}\max\{ia,jb\},$$
and one can check that this is precisely the $q$th Ekeland--Hofer capacity of $E(a,b)$. Alternatively, this is the $q$th smallest element of the infinite array $$(ic\;:\; i \in \Z_{\geq 1},\; c \in \{a,b\}).$$
We summarize this subsection in the following proposition.
\begin{prop}
For $q \in \Z_{\geq 1}$, we have $H(V) = \K\langle A_1,A_2,A_3,\dots\rangle$ with $|A_q| = -2-2q$
and $\calA_{\Omega}(A_q) = \min\limits_{\substack{i,j \in \Z_{\geq 0}\\ i+j=q}}||(i,j)||_\Omega^*$.
	Given a symplectic embedding $X_\Omega \shookrightarrow X_{\Omega'}$, we must have $\calA_{\Omega'}(A_q) \geq \calA_{\Omega}(A_q)$ for all $q$.
\end{prop}

\section{Computing the canonical model of $V_{a,b}$}\label{sec:V_can}

In order to extract the full power of Corollary~\ref{cor:model_indep}, we need to study the bar complex of $V_{\Omega}$, and in particular to understand its homology and corresponding spectral invariants.
We first compute in \S\ref{subsec:some_hpt} the homology $H(\bar V_{\Omega})$ as an unfiltered $\K$-module. 
To better understand the role of the filtration, we then seek to find a canonical\footnote{Note that this is often called the {\em minimal model} in the literature, but as pointed in \cite[Rmk. 2.3.1]{Fukaya-deformation}, this leads to confusion with the notion of minimal model from rational homotopy. In fact, these are essentially Koszul dual notions.} model for $V_\Omega$ as a filtered $\Li$ algebra.
This turns out to fail for general $V_\Omega$, but we succeed in the ellipsoid case $V_{a,b}$,
and in \S\ref{subsec:phi_and_psi} we recursively construct maps $\Phi_{a,b}$ and $\Psi_{a,b}$ which give a filtered canonical model for $V_{a,b}$.

\subsection{Filtered homological perturbation theory}\label{subsec:some_hpt}

We begin with some general observations. Firstly, as a $\K$-module up to isomorphism, $H(\bar V)$ does not depend on the filtration. If we ignore the filtration and view $V$ as an unfiltered $\Li$ algebra over $\K$, a standard corollary of the homological perturbation lemma (see e.g. \cite[\S2.3]{Fukaya-deformation}) states that $V$ is $\Li$ homotopy equivalent to an $\Li$ algebra $V^\can$ whose underlying chain complex is $H(V)$, with trivial differential.
In particular, by basic functoriality properties of the bar construction, we get an isomorphism of $\K$-modules $H(\bar V) \cong H(\bar V^\can)$. 

Note that in principle $V^\can$ could have many nontrivial higher $\Li$ operations, even if $V$ has only a differential and a bracket (c.f. Massey products). 
However, our computation in \S\ref{subsec:homology_V} shows that $H(V)$ is supported in even degrees, whereas the $\Li$ operations on $V^\can$ all have degree $+1$.
It follows that all of the $\Li$ operations on $V^\can$ are automatically trivial for degree reasons.
Thus $H(\bar V^\can) = \bar V^\can$, and we have:
\begin{prop}\label{prop:unfilt_bar_homology}
For $V$ the $\Li$ algebra from Definition~\ref{def:V}, $H(\bar V)$ is abstractly isomorphic as a $\K$-module to the reduced polynomial algebra $\ovl{S}\K\langle A_1,A_2,A_3,\dots\rangle$ on formal variables $A_q$ of degree $|A_q| = -2-2q$ for $q \in \Z_{\geq 1}$.
\end{prop}

Homological perturbation theory (HPT) in fact produces $\Li$ homomorphisms $\Phi: V \rightarrow V^\can$ and $\Psi: V^\can \rightarrow V$ such that $\Phi \circ \Psi$ and $\Psi \circ \Phi$ are both $\Li$ homotopic to the identity. These maps are constructed recursively, or can be described more directly as sums over decorated trees (sometimes interpreted as Feynman diagrams).
The construction of $\Phi$ and $\Psi$ above is based on the following ground inputs:
\begin{itemize}
\item a chain map $\Psi^1: V^\can \rightarrow V$
\item a chain map $\Phi^1: V \rightarrow V^\can$ such that $\Phi^1 \circ \Psi^1 = \id$
\item a chain homotopy $h^1$ between $\Psi^1 \circ \Phi^1$ and $\id$.
\end{itemize}
See e.g. \cite{kontsevich2003deformation,markl2004homotopy,seidelbook,Fukaya-deformation} for the general strategy and history.
We note that while explicit (as opposed to obstruction theoretic) recursive and tree-counting formulas for $\Phi$ and $\Psi$ in the analogous $\Ai$ case are given in \cite{markl_2006},
the explicit formulas for the $\Li$ case appear to be somewhat more subtle and we could not find them in the literature. We will take a more direct approach in \S\ref{subsec:phi_and_psi} below applied to the $\Li$ algebra $V$.

\sss

Recall that we want to understand $V_\Omega$ as a {\em filtered} $\Li$ algebra.
We expect that if all of the ground inputs are filtration preserving, then the resulting $\Li$ homomorphisms $\Phi_\Omega$ and $\Psi_\Omega$ will be filtered $\Li$ homotopy equivalences. 
However, if the homology of $V_\Omega$ when viewed as a $\Lamo$-module (c.f. Remark~\ref{rmk:filt_vers_Nov}) has nontrivial $T$-torsion, then it will not be possible to find such ground inputs.
Fortunately, in the special case of $V_{a,b}$ we can indeed find ground inputs $\Psi_{a,b}^1,\Phi_{a,b}^1,h_{a,b}^1$ which are filtration preserving, e.g. by putting
\begin{itemize}
\item
$\Phi_{a,b}^1(\ovl{\alpha}_{i,j}) = 0$
\item
$\Phi_{a,b}^1(\ovl{\beta}_{i,j}) = A_{i+j}$
\item $\Psi_{a,b}^1(A_q) = \ovl{\beta}_{\ii(q),\jj(q)}$
\item $h_{a,b}^1(\ovl{\alpha}_{i,j}) = 0$
\item $h_{a,b}^1(\ovl{\beta}_{i,j}) = \left( \ovl{\alpha}_{\ii(q)+1,\jj(q)} + \dots + \ovl{\alpha}_{\ii(q) + \jj(q),1}\right)  - \left( \ovl{\alpha}_{i+1,j} + \dots + \ovl{\alpha}_{i+j,1} \right)$ for $q = i+j$.
\end{itemize}
Here the pair $(\ii(q),\jj(q)) \in \Z_{\geq 1}^2\setminus \{(0,0)\}$ is defined as follows.
\begin{definition}\label{def:argmin}
 Given a four-dimensional convex toric domain $X_\Omega \subset \C^2$ and $q \in \Z_{\geq 1}$, we put 
\begin{align}
(\ii(q),\jj(q)) := \underset{\substack{(i,j) \in \Z_{\geq 0}^2,\\ i+j = q}}{\op{argmin}} ||(i,j)||_\Omega^*
\end{align}
That is, $(\ii(q),\jj(q))$ is the pair $(i,j) \in \Z_{\geq 0}$ with $i + j = q$ for which $||(i,j)||_\Omega^*$ is minimal.
\end{definition}
\NI 
Note that the pair $(\ii(q),\jj(q))$ depends quite sensitively on $\Omega$, although we suppress this dependence from the notation to avoid clutter.
In order to avoid borderline cases, for simplicity we will typically assume that there is a unique minimizer in Definition~\ref{def:argmin}. In the case of the four-dimensional ellipsoid $E(a,b)$, we achieve this by implicitly replacing $b$ with $b+\delta$ for $\delta > 0$ sufficiently small.

We omit the proof of the following lemma since we will not explicitly need it below:
\begin{lemma}\label{lem:h}
With the above definitions, $h_{a,b}^1$ is filtration preserving and satisfies $h^1_{a,b} \circ \bdy + \bdy \circ h^1_{a,b} = \id - \Psi^1_{a,b} \circ \Phi^1_{a,b}.$
\end{lemma}

\subsection{Construction of the maps $\Phi_{a,b}$ and $\Psi_{a,b}$}\label{subsec:phi_and_psi}

Taking the discussion from the previous subsection as motivation, we now proceed to directly construct maps $\Phi_{a,b}$ and $\Psi_{a,b}$ realizing a canonical model for the filtered $\Li$ algebra $V_{a,b}$.
\begin{construction}\label{phi_psi_constr}
For any fixed $a,b \in \R_{> 0}$ and constants $C_{q;a,b} \in \K^*$, $q \in \Z_{\geq 1}$, we recursively define maps $\Phi^k_{a,b}: \odot^k V_{a,b} \rightarrow V^\can_{a,b}$ and $\Psi^k_{a,b}: \odot^k V^\can_{a,b} \rightarrow V_{a,b}$, $k \in \Z_{\geq 1}$,
 by the following properties:
\begin{enumerate}
\item for $q \in \Z_{\geq 1}$ we have $\Psi_{a,b}^1(A_q) = C_{q;a,b}\,\beta_{\ii(q),\jj(q)}$
\item $\Psi_{a,b}^k \equiv 0$ for $k \geq 2$
\end{enumerate}
and
\begin{enumerate}
\item 
for $(i,j) \in \Z_{\geq 1}^2\setminus\{(0,0)\}$ we have $\Phi_{a,b}^1(\beta_{i,j}) = \frac{\ii(q)!\jj(q)!}{i!j!C_{q;a,b}}A_q$ with $q = i + j$
\item
$\Phi_{a,b}^k = 0$ if any of the inputs is $\alpha_{i,j}$ for some $i,j$
 \item for $k \geq 2$, we have $\Phi_{a,b}^k(\beta_{\ii(q_1),\jj(q_1)},\dots,\beta_{\ii(q_k),\jj(q_k)}) = 0$ for any $q_1,\dots,q_k \in \Z_{\geq 1}$ 
\item 
for $k \geq 2$ and $(i_1,j_1),\dots,(i_k,j_k) \in \Z_{\geq 0}^2 \setminus \{(0,0)\}$, we have
\begin{align*}
  j_1\Phi_{a,b}^k(\beta_{i_1-1,j_1},\beta_{i_2,j_2},\dots,\beta_{i_k,j_k}) - i_1\Phi_{a,b}^k(\beta_{i_1,j_1-1},\beta_{i_2,j_2},\dots,\beta_{i_k,j_k}) \\+ \sum_{m=2}^{k}(i_1j_m -j_1i_m)\Phi_{a,b}^{k-1}(\beta_{i_1+i_m,j_1+j_m},\beta_{i_2,j_2},\dots,\wh{\beta_{i_m,j_m}},\dots,\beta_{i_k,j_k}) = 0.
 \end{align*}
\end{enumerate}
\end{construction}

\begin{remark}
For the time being we leave the arbitrary constants $C_{q;a,b}$ unspecified. They will not affect the embedding obstructions described in \S\ref{sec:emb_obstrs}. However, they will play a role in the enumerative invariants discussed in \S\ref{sec:enum_imps}, and we will nail down a choice in \S\ref{subsec:rounding}.
\end{remark}

\begin{definition}
We will say that a basis element $\alpha_{i,j}$ or $\beta_{i,j}$ is {\bf action minimal} if we have $(i,j) = (\ii(q),\jj(q))$ for $q = i + j$.	
\end{definition}

Note that (3) states that $\Phi_{a,b}^k$ vanishes whenever all of its input basis elements are action minimal.
This is the main place where dependence on $a,b$ enters.
 Also, (4) is a direct translation of the $\Li$ homomorphism relations for $\Phi_{a,b}$.

Using (4), we can iteratively modify the inputs until they are all action minimal.
Namely, given inputs $\beta_{i_1,j_1},\dots,\beta_{i_k,j_k}$ which are not all action minimal, assume without loss of generality that $(i_1,j_1)$ is not action minimal.
In the case $i_1 < \ii(i_1+j_1)$, we compute $\Phi_{a,b}^k(\beta_{i_1,j_1},\dots,\beta_{i_k,j_k})$ recursively via
\begin{align}\label{eq:first_rec}
 \Phi_{a,b}^k(\beta_{i_1,j_1},\beta_{i_2,j_2},\dots,\beta_{i_k,j_k}) = \tfrac{i_1+1}{j_1}\Phi_{a,b}^k(\beta_{i_1+1,j_1-1},\beta_{i_2,j_2},\dots,\beta_{i_k,j_k})\nonumber \\- \tfrac{1}{j_1}\sum_{m=2}^{k}([i_1+1]j_m -j_1i_m)\Phi_{a,b}^{k-1}(\beta_{i_1+i_m+1,j_1+j_m},\beta_{i_2,j_2},\dots,\wh{\beta_{i_m,j_m}},\dots,\beta_{i_k,j_k}).
\end{align}
Similarly, if $i_1 > \ii(i_1+j_1)$, we compute $\Phi_{a,b}^k(\beta_{i_1,j_1},\dots,\beta_{i_k,j_k})$ recursively via
\begin{align}\label{eq:sec_rec}
 \Phi_{a,b}^k(\beta_{i_1,j_1},\beta_{i_2,j_2},\dots,\beta_{i_k,j_k}) =
 \tfrac{j_1+1}{i_1}\Phi_{a,b}^k(\beta_{i_1-1,j_1+1},\beta_{i_2,j_2},\dots,\beta_{i_k,j_k})\nonumber
 \\+ \tfrac{1}{i_1}\sum_{m=2}^{k}(i_1j_m -[j_1+1]i_m)\Phi_{a,b}^{k-1}(\beta_{i_1+i_m,j_1+j_m+1},\beta_{i_2,j_2},\dots,\wh{\beta_{i_m,j_m}},\dots,\beta_{i_k,j_k}).
\end{align}

\sss

Here are some example computations:
\begin{example}\label{exPhi2}
We compute $\Phi^2_{1,R}(\beta_{1,1},\beta_{1,1})$ for $R \gg 1$.
We have
\begin{align*}
\Phi^2_{1,R}(\beta_{1,1},\beta_{1,1}) &= 2\Phi^2_{1,R}(\beta_{2,0},\beta_{1,1}) - (2\cdot 1 - 1 \cdot 1)\Phi^1_{1,R}(\beta_{3,2})\\
\Phi^1_{1,R}(\beta_{3,2}) &= \tfrac{5!0!}{3!2!C_{5;1,R}}A_5 = \tfrac{10}{C_{5;1,R}}A_5\\
\Phi^2_{1,R}(\beta_{2,0},\beta_{1,1}) &= \Phi_{1,R}(\beta_{1,1},\beta_{2,0}) = 2\Phi^2_{1,R}(\beta_{2,0},\beta_{2,0}) - (2\cdot 0 - 1\cdot2)\Phi^1_{1,R}(\beta_{4,1})\\
\Phi^1_{1,R}(\beta_{4,1}) &= \tfrac{5!0!}{4!1!C_{5;1,R}} = \tfrac{5}{C_{5;1,R}}A_5.
\end{align*}
We have $\Phi^2_{1,R}(\beta_{2,0},\beta_{2,0}) = 0$ (using (3) in Construction~\ref{subsec:phi_and_psi} and the fact that $\beta_{2,0}$ is action minimal), so we get
\begin{align*}
\Phi^2_{1,R}(\beta_{1,1},\beta_{1,1}) = \tfrac{20}{C_{5;1,R}}A_5 - \tfrac{10}{C_{5;1,R}}A_5 = \tfrac{10}{C_{5;1,R}}A_5.
\end{align*}
\end{example}

\begin{example}
We compute $\Phi^3_{1,R}(\beta_{1,1},\beta_{1,1},\beta_{1,1})$ for $R \gg 1$.
We have
\begin{align*}
\Phi^3_{1,R}(\beta_{1,1},\beta_{1,1},\beta_{1,1}) &= 2\Phi^3_{1,R}(\beta_{2,0},\beta_{1,1},\beta_{1,1}) - 2\Phi^2_{1,R}(\beta_{3,2},\beta_{1,1})\\
\Phi^3_{1,R}(\beta_{2,0},\beta_{1,1},\beta_{1,1}) &= 2\Phi^3_{1,R}(\beta_{2,0},\beta_{2,0},\beta_{1,1}) +2\Phi^2_{1,R}(\beta_{4,1},\beta_{1,1}) - \Phi^2_{1,R}(\beta_{3,2}, \beta_{2,0})\\
\Phi^3_{1,R}(\beta_{2,0},\beta_{2,0},\beta_{1,1}) &= 2\Phi^3_{1,R}(\beta_{2,0},\beta_{2,0},\beta_{2,0})+ 4\Phi^2_{1,R}(\beta_{4,1},\beta_{2,0}) = 4\Phi^2_{1,R}(\beta_{4,1},\beta_{2,0})\\
\Phi^2_{1,R}(\beta_{4,1},\beta_{1,1}) &= 2\Phi^2_{1,R}(\beta_{4,1},\beta_{2,0}) + 2\Phi^1_{1,R}(\beta_{6,2})\\
\Phi^2_{1,R}(\beta_{4,1},\beta_{2,0}) &= 5\Phi^2_{1,R}(\beta_{5,0},\beta_{2,0}) + 2\Phi^1_{1,R}(\beta_{7,1}) = 2\Phi^1_{1,R}(\beta_{7,1})\\
\Phi^2_{1,R}(\beta_{3,2},\beta_{1,1}) &= 2\Phi^2_{1,R}(\beta_{3,2},\beta_{2,0}) - \Phi^1_{1,R}(\beta_{5,3})\\
2\Phi^2_{1,R}(\beta_{3,2},\beta_{2,0}) &= 4\Phi^2_{1,R}(\beta_{4,1},\beta_{2,0}) + 4\Phi^1_{1,R}(\beta_{6,2})\\
\Phi^1_{1,R}(\beta_{6,2}) &= \tfrac{8!0!}{6!2!C_{8;1,R}}A_8 = \tfrac{28}{C_{8;1,R}}A_8\\
\Phi^1_{1,R}(\beta_{7,1}) &= \tfrac{8!0!}{7!1!C_{8;1,R}}A_8 = \tfrac{8}{C_{8;1,R}}A_8\\
\Phi^1_{1,R}(\beta_{5,3}) &= \tfrac{8!0!}{5!3!C_{8;1,R}}A_8 = \tfrac{56}{C_{8;1,R}}A_8.
\end{align*}
Combining the above, we get
$\Phi_{1,R}^3(\beta_{1,1},\beta_{1,1},\beta_{1,1}) = \tfrac{192}{C_{8;1,R}}A_8$.
\end{example}

As visible in the above examples, there are choices involved as to what order we apply the recursion. For example, to compute
$\Phi^2_{1,R}(\beta_{4,1},\beta_{1,1})$, we can apply ~\eqref{eq:first_rec} to either $\beta_{4,1}$ or $\beta_{1,1}$.
It is not a priori obvious that the final answer is independent of these choices.
The following lemma alleviates this concern. 

\begin{lemma}\label{lem:am_rep}
Let $X_\Omega$ be any four-dimensional convex toric domain. 
Each class $A \in H(\bar V_\Omega)$ is uniquely represented by a cycle which is a linear combination of tensor products of action minimal basis elements in $V_\Omega$.
\end{lemma}
\begin{remark}
We warn the reader that the cycle provided by Lemma~\ref{lem:am_rep} is not necessarily the cycle of minimal action representing $A$.
For example, consider $\Omega_{P(1,1+\eps)}$ for $\eps > 0$ sufficiently small.
The element $4\beta_{2,0}\odot \beta_{2,0} + 10\beta_{5,0} \in \bar V_{\Omega_{P(1,1+\eps)}}$ which has action $5$, and each summand is a tensor product of action minimal basis elements, yet it is homologous (c.f. Example~\ref{exPhi2}) to $\beta_{1,1}\odot \beta_{1,1} \in \bar V_{\Omega_{P(1,1+\eps)}}$, which has action $4 + 2\eps < 5$.
\end{remark}
\begin{proof}[Proof of Lemma~\ref{lem:am_rep}]
Suppose that $y \in \bar V_{\Omega}$ is a linear combination of tensor products of action minimal basis elements, and that $y$ is nullhomologous, i.e. $y = \wh{\ell}(x)$ for some $x \in \bar V_{\Omega}$.
 It suffices to show that $y = 0$. 
 
For $k \in \Z_{\geq 1}$, let $\bar^{\leq k}V_{\Omega}$ denote the subcomplex consisting of linear combinations of elementary tensors of word length at most $k$.
We claim that the map $H(\bar^{\leq k}V_\Omega) \rightarrow H(\bar V_{\Omega})$ induced by the inclusion of $\bar^{\leq k}V$ into $\bar V$ is injective.
In other words, if an element in $\bar^{\leq k}V_\Omega$ is the differential of an element in $\bar V_\Omega$, then it is also the differential of an element in $\bar^{\leq k} V_\Omega$.
To justify this claim, let $V^\can_{\Omega}$ denote a canonical model for the $\Li$ algebra $V_\Omega$, ignoring filtrations.
This means that $V_\Omega^\can$ is an unfiltered $\Li$ algebra with underlying $\K$-module $H(V_{\Omega})$, with all $\Li$ operations necessarily trivial for degree reasons, and we have in particular a commutative diagram
\begin{align*}
\xymatrix{
H(\bar^{\leq k}V^\can_{\Omega}) \ar^{\cong}[r]\ar[d] & H(\bar^{\leq k}V_\Omega)\ar[d]\\
H(\bar V^\can_\Omega) \ar[r]^{\cong} & H(\bar V_\Omega)
}
\end{align*}
The left vertical arrow in induced by the inclusion $\bar^{\leq k}V_\Omega^\can \rightarrow \bar V_\Omega^\can$, and the homology level map is clearly injective since the differential on $\bar V_\Omega^\can$ is trivial.
It follows that the right vertical arrow is also injective, as desired.

Now put $m := \min \{k \in \Z_{\geq 1}\;:\; y \in \bar^{\leq k}V_\Omega\}$, and consider the quotient complex $\bar^{\leq m}V_\Omega / \bar^{\leq m-1}V_\Omega$.
By the earlier claim, we have $y = \wh{\ell}(x')$ for some $x' \in \bar^{\leq m}V_\Omega$.
This means that the class $[y] \in H(\bar^{\leq m}V_\Omega / \bar^{\leq m-1}V_\Omega)$ vanishes. 
On the other hand, note that $[y]$ is represented by the maximal word length part $\pi_m(y) \in \bar^{\leq m} V_{\Omega}$ of $y$ (here $\pi_m$ denotes the projection
$\ovl{S}V_{\Omega} \rightarrow \odot^m V_{\Omega}$),
and this is a linear combination of tensor products of action minimal basis elements.

Let $\calS \subset \Z^m_{\geq 1}$ denote the set of integer tuples $(i_1,\dots,i_m)$ such that $i_1 \leq \dots \leq i_m$,
and let $\K\langle \calS\rangle$ denote free $\K$-module spanned by these tuples.
We consider the $\K$-linear map 
$$Z: \bar^{\leq m}V_\Omega / \bar^{\leq m-1}V_\Omega \rightarrow \K\langle \calS\rangle$$
sending $v_{i_1,j_1}\odot\dots\odot v_{i_m,j_m}$ to $(i_1+j_1,\dots,i_m+j_m)$, where for each $k = 1,\dots,m$ we have either $v_{i_k,j_k} = \ovl{\alpha}_{i_k,j_k}$ or $v_{i_k,j_k} = \ovl{\beta}_{i_k,j_k}$.
Recall that for $(i,j) \in \Z_{\geq 1}^2$ we have $\bdy \ovl{\alpha}_{i,j} = \ovl{\beta}_{i-1,j} - \ovl{\beta}_{i,j-1}$, and note that the bracket term of $\bar^{\leq k}V_\Omega$ disappears when we pass to the quotient complex $\bar^{\leq m}V_\Omega / \bar^{\leq m-1}V_\Omega$.
It is then easy to check that all boundaries in $\bar^{\leq m}V_\Omega / \bar^{\leq m-1}V_\Omega$ are contained in the kernel of $Z$, whereas $Z(\pi_m(y)) \neq 0$ unless we have $\pi_m(y) = 0$ in $\bar^{\leq m}V_\Omega$.
This contradicts the definition of $m$ unless we have $y = 0$ in $\bar V_\Omega$, as desired.
\end{proof}

\begin{remark}
It would be interesting to find a more fundamental combinatorial formula for $\Phi^k$, e.g. in terms of lattice point counts in some lattice polytope. Such a description could give a more conceptual proof of Lemma~\ref{lem:am_rep} and also perhaps shed light on when the structure coefficients appearing in Corollary~\ref{cor:comb_crit_for_rseep} are nonzero.
\end{remark}

\sss 
 
The rest of this section is occupied with the following two lemmas which verify that $\Phi_{a,b}$ and $\Psi_{a,b}$ have the desired properties.
\begin{lemma}\label{lem:phi_and_psi}
Construction~\ref{phi_psi_constr} defines a valid $\Li$ homomorphism $\Phi_{a,b}: V_{a,b} \rightarrow V^\can_{a,b}$ and $\Psi_{a,b}: V^\can_{a,b} \rightarrow V_{a,b}$.
The induced bar homology maps $H(\wh{\Phi}_{a,b}): H(\bar V_{a,b}) \rightarrow H(\bar V^\can_{a,b})$ and $H(\wh{\Psi}_{a,b}): H(\bar V^\can_{a,b}) \rightarrow H(\bar V_{a,b})$ satisfy
\begin{align*}
H(\wh{\Phi}_{a,b}) \circ H(\wh{\Psi}_{a,b}) = \id\\
H(\wh{\Psi}_{a,b}) \circ H(\wh{\Phi}_{a,b}) = \id.
\end{align*}
\end{lemma}
\begin{remark}
A more natural formulation, which we expect could be achieved with a bit more effort, would state that $\Phi_{a,b} \circ \Psi_{a,b}$ and $\Psi_{a,b} \circ \Phi_{a,b}$ are homotopic to the identity as filtered $\Li$ homomorphisms. However, the above formulation in terms of bar complexes suffices for our intended applications.
\end{remark}

\begin{proof}[Proof of Lemma~\ref{lem:phi_and_psi}]
To see that $\Psi_{a,b}$ as defined is an $\Li$ homomorphism, observe that since all the operations in $V_{a,b}^\can$ are trivial, it suffices to show that $\ell_{V_{a,b}}^k$ vanishes on inputs of the form $\Psi^1_{a,b}(A_{q_1}) \odot \dots \odot \Psi^1_{a,b}(A_{q_k})$, but this is trivial since the operations on $V_{a,b}$ vanish when all of the inputs are $\beta$ generators.

To see that $\Phi_{a,b}$ is an $\Li$ homomorphism, we first check that it is a chain map. Since the differential of $V_{a,b}$ vanishes on $\beta$ basis elements, it suffices to check that $\Phi_{a,b}^1$ vanishes on terms of the form 
$\bdy \alpha_{i,j} = j\beta_{i-1,j} - i\beta_{i,j-1}$ for $i,j \in \Z_{\geq 1}$.
Putting $q := i + j - 1$, we have
\begin{align*}
\Phi_{a,b}^1(j\beta_{i-1,j} - i\beta_{i,j-1}) &= j\tfrac{\ii(q)!\jj(q)!}{(i-1)!j!C_{q;a,b}}A_{q} - i\tfrac{\ii(q)!\jj(q)!}{i!(j-1)!C_{q;a,b}}A_q = 0.
\end{align*}

More generally, since the $\Li$ operations on $V_{a,b}^\can$ are trivial, it suffices to check that $\Phi_{a,b}^k$ vanishes on inputs of the form $\wh{\ell}_{V_{a,b}}(v_1\odot\dots\odot v_k)$ for $v_1,\dots,v_k \in V_{a,b}$.
If amongst the inputs $v_1, \dots,v_k$ there are either two or more $\alpha$ basis elements, then each term in the above sum will contain at least one $\alpha$ basis element, and hence automatically vanishes. Similarly, if there are no $\beta$ inputs then the expression is automatically zero.
If there is exactly one $\alpha$, we get precisely the expression (4) above.

Next, we check that $H(\wh{\Phi}_{a,b})\circ H(\wh{\Psi}_{a,b}) = \id$. 
We have $$\wh{\Psi}_{a,b}(A_{q_1} \odot \dots \odot A_{q_k}) = C_{q_1;a,b}\dots C_{q_k;a,b} \beta_{\ii(q_1),\jj(q_1)}\odot \dots \odot \beta_{\ii(q_k),\jj(q_k)},$$ so by property (3) for $\Phi_{a,b}$ we have
\begin{align*}
(\wh{\Phi}_{a,b} \circ \wh{\Psi}_{a,b})(A_{q_1} \odot \dots \odot A_{q_k}) &= C_{q_1;a,b}\dots C_{q_k;a,b} \Phi^1_{a,b}(\beta_{\ii(q_1),\jj(q_1)}) \odot \dots \odot  \Phi^1_{a,b}(\beta_{\ii(q_k),\jj(q_k)}) \\&= A_{q_1} \odot \dots \odot A_{q_k}.
\end{align*}

Finally, we check that $H(\wh{\Psi}_{a,b}) \circ H(\wh{\Phi}_{a,b}) = \id$.
Since $H(\bar V_{a,b})$ and $H(\bar V_{a,b}^\can)$ have the same finite rank in each degree and $H(\wh{\Phi}_{a,b})$ is a left inverse to $H(\wh{\Psi}_{a,b})$
it follows immediately that $H(\wh{\Phi}_{a,b})$ and $H(\wh{\Psi}_{a,b})$ are both invertible and hence $H(\wh{\Phi}_{a,b})$ is also a right inverse to $H(\wh{\Psi}_{a,b})$.
\end{proof}

\begin{lemma}\label{lem:phi_is_filt_pres}
The $\Li$ homomorphisms $\Phi_{a,b}$ and $\Psi_{a,b}$ are filtration preserving.
\end{lemma}
\begin{proof}
The fact that $\Psi_{a,b}$ is filtration preserving is manifest. 
As for $\Phi_{a,b}$, given $\beta_{i_1,j_1},\dots,\beta_{i_k,j_k} \in V_{a,b}$, we need to verify the following action inequality:
\begin{align}\label{needed_ineq}
\calA_{a,b}(\Phi_{a,b}^k(\beta_{i_1,j_1},\dots,\beta_{i_k,j_k})) \leq \sum_{s=1}^k \calA_{a,b}(\beta_{i_s,j_s}).
\end{align}
The case $k = 1$ is clear. For $k \geq 2$, if each of the inputs $\beta_{i_1,j_1},\dots,\beta_{i_k,j_k}$ is action minimal then 
$\Phi_{a,b}^k(\beta_{i_1,j_1},\dots,\beta_{i_k,j_k}) = 0$ and there is nothing to check,
so may assume without loss of generality that $\beta_{i_1,j_1}$ is not action minimal.

We will further suppose that $i_1 < \ii(i_1+j_1)$, the case $i_1 > \ii(i_1+j_1)$ being closely analogous.
In order to recursively compute $\Phi_{a,b}^k(\beta_{i_1,j_1},\dots,\beta_{i_k,j_k})$, the next step is to apply \eqref{eq:first_rec} in order to write it as a linear combination of terms with strictly smaller $k$ or else strictly smaller $j_1$. 
We may assume by induction that we already know
\begin{align}\label{eq:ind1}
\calA_{a,b}(\Phi_{a,b}^k(\beta_{i_1+1,j_1-1},\beta_{i_2,j_2},\dots,\beta_{i_k,j_k})) \leq \calA_{a,b}(\beta_{i_1+1,j_1-1}) + \calA_{a,b}(\beta_{i_2,j_2}) + \dots + \calA_{a,b}(\beta_{i_k,j_k})
\end{align}
and, for $m = 2,\dots,k$,
\begin{align}\label{eq:ind2}
\calA_{a,b}(\Phi_{a,b}^{k-1}(\beta_{i_1+i_m+1,j_1+j_m},\beta_{i_2,j_2},\dots,\wh{\beta_{i_m,j_m}},\dots,\beta_{i_k,j_k}))\leq \nonumber\\\calA_{a,b}(\beta_{i_1+i_m+1,j_1+j_m}) + \calA_{a,b}(\beta_{i_2,j_2}) + \dots + \wh{\calA_{a,b}(\beta_{i_m,j_m})} + \dots + \calA_{a,b}(\beta_{i_k,j_k}).
\end{align}

We first observe that the right hand side of \eqref{eq:ind1} is at most the right hand side of \eqref{needed_ineq}.
Indeed, it suffices to show that $\calA_{a,b}(\beta_{i_1+1,j_1-1}) \leq \calA_{a,b}(\beta_{i_1,j_1})$. This follows from the assumption $i_1 < \ii(i_1+j_1)$, since the function 
$t \mapsto ||(\ii(q)+t,\jj(q)-t)||_{\Omega}^*$ is monotonically increasing with $t^2$.

In order to complete the proof, we need to show the right hand size of \eqref{eq:ind2} is at most the right hand side of \eqref{needed_ineq}. It suffices to establish the following inequality:
$$\calA_{a,b}(\beta_{i_1+1,j_1}) \leq \calA_{a,b}(\beta_{i_1,j_1}).\footnote{Note that in this case we evidently must have $\calA_{a,b}(\beta_{i_1+1,j_1}) = \calA_{a,b}(\beta_{i_1,j_1})$.}$$
Indeed, by the triangle inequality we then have
\begin{align*}
\calA_{a,b}(\beta_{i_1+i_m+1,j_1+j_m}) &\leq \calA_{a,b}(\beta_{i_1+1,j_1}) + \calA_{a,b}(\beta_{i_m,j_m})\\
& \leq \calA_{a,b}(\beta_{i_1,j_1}) + \calA_{a,b}(\beta_{i_m,j_m}),
\end{align*}
from which the desired inequality follows.

Suppose by contradiction that we have
$\calA_{a,b}(\beta_{i_1+1,j_1}) > \calA_{a,b}(\beta_{i_1,j_1})$, i.e.
$$\max\{[i_1+1]a,j_1b\} > \max\{i_1a,j_1b\}.$$
Then we must have $(i_1+1)a > j_1b$.
On the other hand, from the assumption $i_1 < \ii(i_1+j_1)$ we have
$\calA(\beta_{i_1+1,j_1-1}) \leq \calA(\beta_{i_1,j_1})$, i.e.
$$ \max\{[i_1+1]a,[j_1-1]b\} \leq \max\{i_1a,j_1b\},$$
which necessitates $(i_1+1)a \leq j_1b$, a contradiction.
\end{proof}

\begin{remark} \hspace{1cm}
\begin{enumerate}
	\item  We note that (3) in Construction~\ref{subsec:phi_and_psi} is often mandated purely for action reasons, but this is not always the case, e.g. in principle a term $\Phi^2_{1,1+\eps}(\beta_{1,0},\beta_{1,0}) = A_3$ would be permitted by action and index considerations. This is related to the ambiguities in punctured curve counts caused by index zero symplectization curves, which we further discuss in \S\ref{sec:enum_imps}.
\item The analogue of Lemma~\ref{lem:phi_is_filt_pres} does not hold for general $\Omega$. We can, however, iteratively apply analogues of ~\eqref{eq:first_rec} and ~\eqref{eq:sec_rec} to any cycle in $\bar V_{\Omega}$ in order to find its representative promised by Lemma~\ref{lem:am_rep}.
\end{enumerate}
\end{remark}

\begin{example}
Consider the filtered $\Li$ algebra $V_{\Omega_{P(a,b)}}$ associated with the polydisk $P(a,b)$ (as usual we assume $a \leq b$).
Note that basis elements $\beta_{i,j}$ are action minimal if and only if we have $j = 0$.
Now imagine that $\Phi_{\Omega_{P(a,b)}}$ and $\Psi_{\Omega_{P(a,b)}}$ really were filtered $\Li$ homotopy equivalences.
Then it would follow that we have a corresponding canonical filtered $\Li$ model $V^\can_{\Omega_{P(a,b)}}$ whose action filtration depends only on $a$. 
In particular, in this case we would not have any interesting capacities $\gapac_\bb(P(a,b))$ apart from functions of $a$.
As it turns out, whereas the linear spectral invariants of $P(a,b)$ indeed are just multiples of $a$, the bar complex spectral invariants are much richer than this and do depend on both $a$ and $b$ (c.f. \cite[Ex. 1.14]{HSC}).
\end{example}

\section{Symplectic embedding obstructions}\label{sec:emb_obstrs}

We begin this section by reviewing the capacities $\gapac_\bb$ from \cite{HSC} in \S\ref{subsec:spectral_capacities}. In \S\ref{subsec:persistent} we discuss the role of persistent homology and explain how it can be used to algorithmically extract obstructions. We then restrict to the case of ellipsoids in \S\ref{subsec:ellipsoids} and use the canonical model from the previous section to read off embedding obstructions. Finally, in \S\ref{subsec:beyond} we explore obstructions which lie beyond the bar complex spectral invariants, illustrating this technique in the context of polydisks. For concreteness we mostly stick to the case that $X$ is four-dimensional, although we expect all of the results in this section to have natural extensions to higher dimensions.

\subsection{Bar complex spectral invariants}\label{subsec:spectral_capacities}

Let $X_{\Omega}$ and $X_{\Omega'}$ be four-dimensional convex toric domains, and suppose we have a symplectic embedding $X_{\Omega} \shookrightarrow X_{\Omega'}$.
By Corollary~\ref{cor:model_indep}, we have a filtered $\Li$ homomorphism
$Q: V_{\Omega'} \rightarrow V_{\Omega}$ which is unfiltered $\Li$ homotopic to the identity. In particular, this means that the identity map
$\id: H(\bar V_{\Omega'}) \rightarrow H(\bar V_{\Omega})$ is filtration preserving,
so for any homology class $A \in H(\bar V)$ we have the inequality
$$ \calA_{\Omega}(A) \leq \calA_{\Omega'}(A).$$
Here we put
$$\calA_{\Omega}(A) := \min \{\calA_\Omega(x)\;:\; \bdy x = 0,\; [x] = A\text{ for some } x \in \bar V_{\Omega}\},$$
with $\calA_{\Omega'}(A)$ defined similarly

By Proposition~\ref{prop:unfilt_bar_homology}, we have abstract isomorphisms of $\K$-modules 
\begin{align}
H(\bar V_{\Omega}) \cong H(\bar V_{\Omega'}) \cong \ovl{S}\K\langle A_1,A_2,A_3,\dots \rangle,
\end{align}
where $\ovl{S}(-)$ denotes the reduced symmetric tensor algebra as in \S\ref{subsec:recollections}.
However, since $\ovl{S}\K\langle A_1,A_2,A_3,\dots \rangle$ is generally multidimensional in any given degree (unlike $H(V)$, c.f. \S\ref{subsec:homology_V}), we need to think more carefully about how to reference homology classes.
For example, in degree $-8$ we have $\K\langle A_3, A_1 \odot A_1\rangle$. Therefore any nonzero linear combination $c_1 A_3 + c_2 A_1 \odot A_1$ with $(c_1,c_2) \in \K^2 \setminus \{(0,0)\}$ gives a corresponding inequality $$\calA_{\Omega}(c_1 A_3 + c_2 A_1 \odot A_1) \leq \calA_{\Omega'}(c_1 A_3 + c_2 A_1 \odot A_1).$$
In fact the quantities $\calA_{\Omega}(A)$ and $\calA_{\Omega'}(A)$ are unaffected if we scale $A$ by an element of $\K^*$, so what we have is a family of spectral invariants indexed by a one-dimensional projective space $\K\mathbb{P}^1$.

One way to reference homology classes in $H(\bar V_\Omega)$ and $H(\bar V_{\Omega'})$  would be to simply describe cycles in terms of their representation in the basis $\{\alpha_{i,j}, \beta_{i,j}\}$. 
However, this is not the most natural approach in a general context, since it depends on our precise models $V_\Omega$ and $V_{\Omega'}$ and the fact that the cobordism map can be identified with the identity map.
Following \cite{HSC}, a more canonical approach is to use the $\Li$ homomorphism 
$$ \Xi_\sk: \chlin(X_{\Omega}) \rightarrow \chlin(E_\sk)$$ induced by the inclusion $E_\sk \shookrightarrow X_{\Omega}$,
where $E_\sk = \delta E(1,R)$ denotes a ``skinny ellipsoid'' for $\delta > 0$ sufficiently small and $R \gg 1$ sufficiently large.
This cobordism map makes sense for any Liouville domain $X$ and is uniquely determined up to $\Li$ homotopy, hence it induces a canonical way to refer to homology classes in $\bar \chlin(X)$,
via the inverse of the homology level map $H(\wh{\Xi}_\sk): H(\bar \chlin(X_\Omega))\rightarrow H(\bar \chlin(E_\sk))$. 
By the results in  \cite{McDuffSiegel_counting}, $\Xi_\sk$ can equivalently be defined by counting punctured curves in $X$ with local tangency constraints.

As a shorthand, for $\delta \ll 1$ sufficiently small and $R \gg 1$ sufficiently large we put $V_\sk := V_{\delta,\delta R}$,
and similarly $V_\sk^\can := V_{\delta,\delta R}^\can$, $\Phi_\sk := \Phi_{\delta,\delta R}$, $\Psi_\sk := \Psi_{\delta,\delta R}$, etc.
We have identifications of $\K$-modules $\chlin(E_\sk) \approx V_\sk^\can \approx \K[t]$, where 
$A_q \in V^\can_\sk$ corresponds to $t^{q-1}$ for $q \in \Z_{\geq 1}$. We then also have identifications of $\K$-modules
$\bar \chlin(E_\sk) \approx \bar V^\can_{\sk} \approx \ovl{S}\K[t].$\footnote{Recall that, with our grading conventions, the degree of the monomial $t^k$ in $\K[t]$ is $-4-2k$.}
Following \cite{HSC}, for $\bb \in \ovl{S}\K[t]$ we put
$$\gapac_\bb(X) := \calA_{\bar \chlin(X)}((H(\wh{\Xi}_\sk))^{-1}(\bb)).$$
Using $\Psi_{\sk}$, note that $\bb$ also corresponds to an element $\wh{\Psi}_{\sk}(\bb) \in \bar V_{\sk}$, which we can in turn identify with an element of $\bar V_\Omega$ whenever $X_\Omega$ is a four-dimensional convex toric domain.
Our algebraic formalism now computes the capacities $\gapac_\bb$ for this class of domains:
\begin{thm}\label{thm:gb_for_ctd}
If $X_\Omega$ is a four-dimensional convex toric domain, for any $\bb \in \ovl{S}\K[t]$ we have
$$\gapac_\bb(X_\Omega) = \calA_{\Omega}(H(\wh{\Psi}_{\sk}(\bb))).$$
\end{thm}

\subsection{The role of persistent homology}\label{subsec:persistent}

Given a symplectic embedding $X_\Omega \shookrightarrow X_{\Omega'}$, we have the inequality $\gapac_\bb(X_\Omega) \leq \gapac_\bb(X_{\Omega'})$ for each choice of $\bb \in \ovl{S}\K[t]$, and this gives a very large family of obstructions.
However, it turns out there is a much smaller collection of inequalities which determines all of the others. In fact, it suffices to check just finitely many choices of $\bb$ in each degree. This is particularly noteworthy since if we were to work over say $\K = \R$ there would be a priori uncountably many distinct elements $\bb \in \K[t]$, even after projectivizing.

Let $\frakp(d)$ denote the dimension of $H(\bar V_\Omega)$ in degree $\qq \in \Z$.
According to Proposition~\ref{prop:unfilt_bar_homology}, we have that $\frakp(\qq) = 0$ for $d$ odd, while in even degrees we have
$$\frakp(-4) = 1,\;\; \frakp(-6) = 1,\;\; \frakp(-8) = 2,\;\; \frakp(-10) = 2,\;\; \frakp(-12) = 4,\;\; \frakp(-14) = 4,$$
and so on.
According to \cite{zomorodian2005computing}, we can find a homogeneous basis for $\bar V_{\Omega}$, each element of which is either a left endpoint of a finite barcode, a right endpoint of a finite barcode, or a left endpoint of a semi-infinite barcode.
More precisely, for each $\qq \in \Z$ we have nonnegative integers $\frakl(\qq),\frakr(\qq)$\footnote{Note that $\frakl(\qq)$ and $\frakr(\qq)$ also implicitly depend on $\Omega$.} and a basis 
$$\leftgen^\Omega_{\qq;1},\dots,\leftgen^\Omega_{\qq;{\frakl(\qq)}},\rightgen^\Omega_{\qq;1},\dots,\rightgen^\Omega_{\qq;{\frakr(\qq)}},\lonegen^\Omega_{\qq;1},\dots,\lonegen^\Omega_{\qq;{\frakp(\qq)}}$$ for the degree $\qq$ part of $\bar V_{\Omega}$ such that:
\begin{itemize}
\item for each $i \in \{1,\dots,\frakp(\qq)\}$ we have $\wh{\ell}(\lonegen^\Omega_{\qq;i}) = 0$
\item for each $i \in \{1,\dots,\frakl(\qq)\}$ we have $\wh{\ell}(\leftgen^\Omega_{\qq;i}) = 0$
\item for each $i \in \{1,\dots,\frakr(\qq)\}$ we have
$\wh{\ell}(\rightgen_{\qq;i}^\Omega) = \leftgen^\Omega_{{\qq+1};j}$ for some unique ${j \in \{1,\dots,\frakl(\qq+1)\}}$ associated to the pair $(\qq,i)$
\end{itemize}
As explained in \cite{zomorodian2005computing}, finding these basis elements essentially reduces to finding the Smith normal forms of the matrices defining the chain complex $\bar V_{\Omega}$ in each degree, and there is an efficient algorithm for doing so.

In the above basis, each element $\lonegen^\Omega_{\qq;i}$ corresponds to the left endpoint of a semi-infinite barcode with endpoint at $\calA_{\Omega}(\lonegen^\Omega_{\qq;i}) \in \R_{\geq 0}$.
These form a basis for the homology of $\bar V_{\Omega}$.
In particular, the possible values of spectral invariants of $\bar V_{\Omega}$ associated to homology classes in degree $\qq$ are given by $$\{ \calA_{\Omega}(\lonegen^\Omega_{\qq;1}),\dots, \calA_{\Omega}(\lonegen^\Omega_{\qq;\frakp(\qq)}) \},$$
and we can assume without loss of generality that these actions appear in nondecreasing order.
Meanwhile, a pair of basis elements $(\rightgen^\Omega_{\qq;i},\leftgen^\Omega_{\qq+1;j})$ with $\wh{\ell}(\rightgen^\Omega_{\qq;i}) = \leftgen^\Omega_{{\qq+1};j}$ corresponds to a finite barcode with left endpoint at $\calA_{\Omega}(\leftgen^\Omega_{{\qq+1};j}) \in \R_{\geq 0}$ and right endpoint at $\calA_{\Omega}(\rightgen^\Omega_{\qq;i}) \in \R_{\geq 0}$.
The left endpoints of finite barcodes generate the torsion part of $H(\bar V_{\Omega})$ as a module over the Novikov ring $\Lamo$, but do not contribute to its homology over $\K$.

\sss

Fixing $q \in \Z_{\geq 1}$, we now consider how $\gapac_\bb(X_\Omega)$ changes as we vary $\bb$ in the degree $\qq :=-2-2q$ part of $\ovl{S}\K[t]$.
For $i \in \{1,\dots,\frakp(\qq)\}$, define $\bb^{\Omega}_{\qq;i}$ to be the image of $\lonegen^\Omega_{\qq;i}$ under the composition
\begin{align*}
\bar V_{\Omega} \overset{\id}{\longrightarrow} \bar V_\sk \overset{\wh{\Phi}_{\sk}}{\longrightarrow} \bar V^\can_\sk \approx \ovl{S}\K[t].
\end{align*}
Any degree $\qq$ element $\bb \in \ovl{S}\K[t]$ can be written uniquely as a linear combination $c_1\bb^{\Omega}_{\qq;1} + \dots + c_{\frakp(\qq)}\bb^{\Omega}_{\qq;\frakp(\qq)}$ for some $c_1,\dots,c_{\frakp(\qq)} \in \K$, and we have
$$\calA_{\Omega}(c_1\tau^{\Omega}_{\qq;1} + \dots + c_{\frakp(\qq)}\tau^{\Omega}_{\qq;\frakp(\qq)}) = \max\{\calA_{\Omega}(\tau_{\qq;i}^\Omega)\;:\; c_i \neq 0\}.$$
This means that $\gapac_\bb(X_\Omega)$ is constant 
as $\bb$ varies along any stratum of the full flag
$$ \K\langle \bb^{\Omega}_{\qq;1}\rangle \subset \K\langle \bb^{\Omega}_{\qq;1},\bb^{\Omega}_{\qq;2}\rangle \subset \dots \subset \K\langle \bb^{\Omega}_{\qq;1},\dots,\bb^{\Omega}_{\qq;\frakp(\qq)} \rangle.$$

We can similarly associate to $X_{\Omega'}$ the degree $\qq$ elements $\bb^{\Omega'}_{\qq;1},\dots,\bb^{\Omega'}_{\qq;\frakp(\qq)} \in \ovl{S}\K[t]$
and corresponding actions $\calA_{\Omega'}(\tau^{\Omega'}_{\qq;1}) \leq \dots \leq \calA_{\Omega'}(\tau^{\Omega'}_{\qq;\frakp(\qq)}) \in \R_{\geq 0}$ for each $d := -2-2q$ with $q \in \Z_{\geq 1}$.
Note that in general the list of elements $\bb^{\Omega'}_{\qq;1},\dots,\bb^{\Omega'}_{\qq;\frakp(\qq)}$ will be different from $\bb^{\Omega}_{\qq;1},\dots,\bb^{\Omega}_{\qq;\frakp(\qq)}$ under the natural identification $\bar V_{\Omega'} \cong \bar V_{\Omega}$, and consequently the associated full flags could coincide, intersect transversely, or neither.
The following immediate proposition summarizes the finite number of comparisons which are sufficient in each degree:
\begin{prop}\label{prop:comparisons}
Fix $q \in \Z_{\geq 1}$ and put $d := -2-2q$.
In the context of Theorem~\ref{thm:chscII_main}, for each $i  \in \{1,\dots, \frakp(\qq)\}$ we have
$$\gapac_{\bb^{\Omega'}_{\qq;i}}(X_{\Omega}) \leq \gapac_{\bb^{\Omega'}_{\qq;i}}(X_{\Omega'}).$$
Moreover, these inequalities imply that we have $\gapac_\bb(X_\Omega) \leq \gapac_\bb(X_{\Omega'})$ for all degree $d$ elements $\bb \in \ovl{S}\K[t]$.
\end{prop}
\NI Note that if we write $\bb^{\Omega'}_{\qq;i} = c_1\bb^\Omega_{\qq;1} + \dots + c_{\frakp(\qq)}\bb^\Omega_{\qq;\frakp(\qq)}$ for some  $c_1,\dots,c_{\frakp(\qq)} \in \K$, we have $$\gapac_{\bb_{\qq;i}^{\Omega'}}(X_{\Omega}) = \max\{\gapac_{\bb_{\qq;i}^{\Omega}}(X_\Omega)\;:\; c_i \neq 0\}.$$

\begin{proof}[Proof of Proposition~\ref{prop:comparisons}]
The inequality $\gapac_{\bb^{\Omega'}_{\qq;i}}(X_{\Omega}) \leq \gapac_{\bb^{\Omega'}_{\qq;i}}(X_{\Omega'})$ is automatic from the fact that  $\gapac_{\bb^{\Omega'}_{\qq;i}}$ is a symplectic capacity.
Conversely, assuming that we have $\gapac_{\bb_{\qq;i}^{\Omega'}}(X_\Omega) \leq \gapac_{\bb_{\qq;i}^{\Omega'}}(X_{\Omega'})$ for all $i$, we need to establish $\gapac_\bb(X_\Omega) \leq \gapac_\bb(X_{\Omega'})$ for all $\bb$. 
Fix some arbitrary $\bb_0 \in \ovl{S}\K[t]$ of degree $d$, which we can write as a linear combination $\bb_0 = \sum_{i=1}^{\frakp(\qq)}c_i\bb^{\Omega'}_{\qq;i}$ for some $c_1,\dots,c_{\frakp(\qq)} \in \K$.
Let $m$ denote the maximal $i$ such that $c_i \neq 0$.
Observe that we have
$\gapac_{\bb_0}(X_{\Omega'}) = \gapac_{\bb_{\qq;m}^{\Omega'}}(X_{\Omega'})$
and $\gapac_{\bb_0}(X_\Omega) \leq \max\limits_{i=1,\dots,m} \gapac_{\bb_{\qq;i}^{\Omega'}}(X_\Omega),$
so it suffices to establish the inequality
$$\max_{i=1,\dots,m} \gapac_{\bb_{\qq;i}^{\Omega'}}(X_\Omega) \leq  \gapac_{\bb_{\qq;m}^{\Omega'}}(X_{\Omega'}).$$
Now note that for each $i=1,\dots,m$ we have
$$\gapac_{\bb_{\qq;i}^{\Omega'}}(X_\Omega) \leq \gapac_{\bb_{\qq;i}^{\Omega'}}(X_{\Omega'}) \leq \gapac_{\bb_{\qq;m}^{\Omega'}}(X_{\Omega'}),$$
from which the desired inequality follows.
\end{proof}

\subsection{The case of ellipsoids}\label{subsec:ellipsoids}

In the case that $X_\Omega$ is the four-dimensional ellipsoid $E(a,b)$, we can use the maps $\Phi_{a,b}$ and $\Psi_{a,b}$ from \S\ref{subsec:phi_and_psi} to replace the filtered $\Li$ algebra $V_{a,b}$ with the canonical model $V_{a,b}^\can$. 
This means that in order to compute the spectral invariant of a class $A \in H(\bar V_{a,b})$
we can consider its image under 
$H(\wh{\Phi}_{a,b}): H(\bar V_{a,b}) \rightarrow H(\bar V_{a,b}^\can)$ and then directly read off its action.
We arrive at the following refined version of Theorem~\ref{thm:gb_for_ctd} which gives a more direct computation of the capacities $\gapac_\bb$:
\begin{thm}
If $E(a,b)$ is the four-dimensional ellipsoid with area parameters $a,b \in \R_{> 0}$, for any $\bb \in \ovl{S}\K[t]$ we have
$$\gapac_\bb(E(a,b)) = \calA_{a,b}((\wh{\Phi}_{a,b}\circ \wh{\Psi}_{\sk})(\bb)).$$	\end{thm}

\begin{remark}
The basis $\{A_1,A_2,A_3,\dots\}$ for $V_{a,b}^\can$ naturally induces a basis for $\bar V_{a,b}^\can$. Since the differential of $\bar V_{a,b}^\can$ is trivial, each element of the latter basis is a left endpoint of a semi-infinite barcode as in \S\ref{subsec:persistent}.
\end{remark}

\begin{proof}[Proof of Theorem~\ref{thm:a_b_a'_b'_ineq}]
Given a symplectic embedding $E(a,b) \times \C^N \shookrightarrow E(a',b') \times \C^N$ for some $N \in \Z_{\geq 0}$, the induced filtered $\Li$ homomorphism 
$Q: V_{a',b'} \rightarrow V_{a,b}$ from Corollary~\ref{cor:model_indep} is well-defined up to unfiltered $\Li$ homotopy, and hence so is
$\Phi_{a,b} \circ Q \circ \Psi_{a',b'}: V^\can_{a',b'} \rightarrow V^\can_{a,b}$. 
Since the $\Li$ operations on $V^\can_{a,b}$ and $V^\can_{a',b'}$ are trivial, $\Li$ homotopies have no effect, i.e. the compositions
$\Phi_{a,b} \circ Q \circ \Psi_{a',b'}$ and $\Phi_{a,b} \circ \Psi_{a',b'}$ are $\Li$ homotopic and therefore must be equal.
In particular, each nonzero structure coefficient of the $\Li$ homomorphism 
$\Phi_{a,b} \circ \Psi_{a',b'}$
gives rise to an action inequality, and we can readily compute these structure coefficients using 
\eqref{eq:first_rec} and \eqref{eq:sec_rec}.
\end{proof}

\subsection{Beyond capacities}\label{subsec:beyond}

The basic observation underlying this subsection is as follows. Suppose that we have a filtered $\Li$ homomorphism $Q: V_{\Omega'} \rightarrow V_{\Omega}$.
Suppose that for some basis elements $v_1',\dots,v_k' \in V_{\Omega'}$ and $v_1,\dots,v_l \in V_{\Omega'}$,
the induced map on bar complexes $\wh{Q}: \bar V_{\Omega'} \rightarrow \bar V_{\Omega}$ has a nonzero structure coefficient 
$\langle \wh{Q}(v_1' \odot \dots \odot v_k'),v_1 \odot \dots \odot v_l\rangle \neq 0$.
Then by filtration considerations we must have the inequality 
\begin{align}\label{eq:total_ineq}
\sum_{i=1}^k\calA_{\Omega'}(v_i') \geq \sum_{j=1}^l \calA_{\Omega}(v_j).
\end{align}
However, $\wh{Q}$ is not an arbitrary filtered chain map, since it comes from the filtered $\Li$ homomorphism $Q$.
In fact, there must be surjective set map $M: \{1,\dots,l\} \rightarrow \{1,\dots,k\}$ such that for $j = 1,\dots,k$ we have
\begin{align}\label{eq:indiv_ineq}
\sum_{i \in M^{-1}(j)}\calA_{\Omega'}(v_i') \geq \calA_{\Omega}(v_j).
\end{align}
Note that the bar complex spectral invariants only have access to total inequalities of the form ~\eqref{eq:total_ineq}.
In the case of ellipsoids, by triviality of the $\Li$ operations for $V^\can_{a,b}$ these give the same obstructions as the individual inequalities ~\eqref{eq:indiv_ineq} (c.f. \cite[\S 6.3]{HSC}). However, for more general domains the latter inequalities could in principle give stronger obstructions.

We now illustrate this phenomenon by proving Theorem~\ref{prop:poly_into_cube} and Theorem~\ref{prop:poly_into_ball} from the introduction.
Before proving these results, we need the following:
\begin{prop} \label{prop:nonzero_coeffs}
Put $\Omega = \Omega_{P(a,b)}$ for $a \leq b$.
\begin{enumerate} 
\item
For all $d \in Z_{\geq 1}$, the unique representative $x$ of $(\odot^{d-1} \beta_{1,0}) \odot \beta_{0,1}$ provided by Lemma~\ref{lem:am_rep} satisfies $$\langle x,\beta_{2d-1,0}\rangle \neq 0,$$ 

\item For all $d \in Z_{\geq 1}$, the unique representative $x$ of $\odot^d\beta_{1,1}$ provided by Lemma~\ref{lem:am_rep} satisfies $$\langle x, \beta_{3d-1,0}\rangle \neq 0.$$  
\end{enumerate}
\end{prop}
\begin{proof}
To prove (2), observe that a generator of $\beta_{i,j}$ is action minimal for $V_{\Omega_{P(a,b)}}$ if and only if it is action minimal for $V_{\Omega_{E_\sk}}$.
By Corollary~\ref{cor:T_equals_S}, we have $\langle x, \beta_{3d-1}\rangle = S_d = d! T_d$.
According to \cite[Cor. 4.1.3]{McDuffSiegel_counting}, this is positive for all $d \in \Z_{\geq 1}$.

As for (1), we give an example computation, the general case being manifest from this.
We have
\begin{align*}
\beta_{1,0}\odot \beta_{1,0}\odot \beta_{1,0} \odot \beta_{0,1} &\sim \beta_{1,0}\odot \beta_{1,0}\odot \beta_{1,0} \odot \beta_{1,0} +3\beta_{1,0}\odot \beta_{1,0} \odot\beta_{2,1} \\
\beta_{1,0} \odot \beta_{1,0} \odot \beta_{2,1} &\sim 3\beta_{1,0} \odot \beta_{1,0} \odot \beta_{3,0} + 2\beta_{1,0}\odot \beta_{4,1}\\
\beta_{1,0}\odot \beta_{4,1} &\sim 5\beta_{1,0} \odot \beta_{5,0} + \beta_{6,1}\\
\beta_{6,1} &\sim 7\beta_{7,0},
\end{align*}
and hence we have
$$\langle x, \beta_{7,0} \rangle = (3)(2)(1)(7) \neq 0.$$
\end{proof}
\begin{remark}
One could also try to prove (1) in Proposition~\ref{prop:nonzero_coeffs} directly from the combinatorial definition of $S_d$, but this appears to be much less straightforward than the analogous computation for (2). We have used computer calculations to independently verify $S_d > 0$ for $d = 1,\dots, 26$.
\end{remark}

\begin{proof}[Proof of Theorem~\ref{prop:poly_into_cube}]

Let $Q: V_{\Omega'} \rightarrow V_\Omega$ be a filtered $\Li$ homomorphism for $\Omega' := \Omega_{P(c,c)}$ and $\Omega := \Omega_{P(1,a)}$ as guaranteed by Corollary~\ref{cor:model_indep}.
Suppose that we have a nonzero structure coefficient
$$\langle \wh{Q}(\odot^{d-1}\beta_{1,0}\odot \beta_{0,1}),\beta_{i_1,j_1}\odot \dots \odot \beta_{i_k,j_k}\rangle \neq 0$$
for some $d \in \Z_{\geq 1}$ and $\beta_{i_1,j_1},\dots,\beta_{i_k,j_k} \in V_\Omega$.
Then from the definition of $\wh{Q}$ we can find $d_1,\dots,d_s \in \Z_{\geq 1}$ 
such that either $\langle Q^{d_s}(\odot^{d_s}\beta_{1,0}), \beta_{i_s,j_s}\rangle \neq 0$ or $\langle Q^{d_s}(\odot^{d_s-1}\beta_{1,0} \odot \beta_{0,1}), \beta_{i_s,j_s}\rangle \neq 0$ for each $s = 1,\dots,k$.
Either way, by action and index considerations we must then have 
\begin{align*}
cd_s \geq i_s + aj_s \;\;\;\;\;\text{     and     } \;\;\;\;\;i_s+j_s = 2d_s-1
\end{align*}
for each $s = 1,\dots,k$.
Eliminating $i_s$, this means we have
$$ c \geq \frac{2d_s-1-j_s+aj_s}{d_s}$$
for $s = 1,\dots k$.

Now suppose by contradiction that we have both $c < 2$ and $c < a$.
We claim that $j_s = 0$ for $s = 1,\dots, k$.
Assuming this claim, it follows that $\wh{Q}(\odot^{d-1}\beta_{1,0}\odot \beta_{0,1})$ 
is a linear combination of tensor products of action minimal basis elements of $V_\Omega$ as in Lemma~\ref{lem:am_rep}.
Since $\wh{Q}$ induces the identity map $H(\bar V_{\Omega'}) \rightarrow H(\bar V_\Omega)$ on homology,
we also have that $\wh{Q}(\odot^{d-1}\beta_{1,0}\odot \beta_{0,1})$ is homologous to $\odot^{d-1}\beta_{1,0}\odot \beta_{0,1}$ in $\bar V_\Omega$.
Then by (1) in Proposition~\ref{prop:nonzero_coeffs}, we have
$$\langle \wh{Q}(\odot^{d-1}\beta_{1,0}\odot \beta_{0,1}), \beta_{2d-1,0}\rangle \neq 0.$$
Action considerations then imply
$cd \geq 2d-1,$ and since $d \in \Z_{\geq 1}$ is arbitrary this gives $c \geq 2$, which is a contradiction.

We now justify the above claim.
First suppose that we have $a \geq 2$.
Using the inequality
$$ \frac{2d_s -1 - j_s + aj_s}{d_s} < 2$$
we get $j_s(a-1) \leq 1$, and hence $j_s = 0$ as claimed.

Now suppose that we have $a < 2$.
Using the inequality
$$ \frac{2d_s -1 - j_s + aj_s}{d_s} < a$$
we get $$j_s < \frac{d_s(a-2)+1}{a-1}.$$
To conclude that $j_s = 0$, it suffices to show that we have
$$\frac{d_s(a-2) + 1}{a-1} \leq 1,$$
i.e. 
$$ d_s(a-2) \leq a - 2,$$
which holds since $a < 2$ and $d_s \geq 1$.

Finally, to establish the sharpness claim, observe that there is a naive inclusion $P(1,a) \subset P(a,a)$,
so this must be optimal for $a \leq 2$. 
For $a \geq 2$, we instead use the symplectic embedding $P(1,a) \times \C^N \shookrightarrow P(2,2) \times \C^N$
which exists for all $N \geq 1$ by \cite[Thm. 1.4]{hind2015some}.

\end{proof}
\begin{remark}\hspace{1cm}
\begin{enumerate}
\item In the case $a \geq 2$, Proposition~\ref{prop:poly_into_cube} is subsumed by \cite{irvine2019stabilized}, which also covers the much more general case of target $P(c,d)$.
\item
In the four-dimensional case, the obstructions in Proposition~\ref{prop:poly_into_cube} also show that the naive inclusion $P(1,a) \shookrightarrow P(a,a)$ is optimal for $a \leq 2$.
This is a special case of \cite[Thm. 1.6]{hutchings2016beyond}, which also covers more general target polydisks.
For $a \geq 2$, 
According to \cite[Prop. 4.4.4]{Schlenk_embedding_problems}, for $a \geq 2$ symplectic folding gives $P(1,a) \shookrightarrow P(c,c)$ for any $c > 1 + a/2$, and for $a > 4$ multiple symplectic folding gives an even better embedding. 
We do not know to what extent these are optimal.
\end{enumerate}\end{remark}

\begin{proof}[Proof of Theorem~\ref{prop:poly_into_ball}]
This is similar to the proof of Proposition~\ref{prop:poly_into_cube}.
By Corollary~\ref{cor:model_indep}, we have a filtered $\Li$ homomorphism 
$Q: V_{\Omega'} \rightarrow V_\Omega$, now with $\Omega' := \Omega_{B^4(c)}$ and $\Omega := P(1,a)$.
Suppose that we have a nonzero structure coefficient 
$$\langle \wh{Q}(\odot^{d}\beta_{1,1}),\beta_{i_1,j_1}\odot \dots \odot \beta_{i_k,j_k}\rangle \neq 0$$
for some $d \in \Z_{\geq 1}$ and $\beta_{i_1,j_1},\dots,\beta_{i_k,j_k} \in V_\Omega$.
Then we can find $d_1,\dots,d_s \in \Z_{\geq 1}$ such that $\langle Q^{d_s}(\odot^{d_s} \beta_{1,1}),\beta_{i_s,j_s}\rangle \neq 0$
for $s = 1,\dots,d$.
By action and index considerations we must have 
\begin{align*}
cd_s \geq i_s + aj_s \;\;\;\;\;\text{     and     } \;\;\;\;\ i_s+j_s = 3d_s-1,
\end{align*}
and eliminating $i_s$ gives 
$$ c \geq \frac{3d_s-1-j_s+aj_s}{d_s}$$
for $s = 1,\dots,k$.

Suppose by contradiction that we have both $c < 3$ and $c < a+1$.
We claim that $j_s = 0$ for $s = 1,...,k$.
Assuming this claim, it follows as in the proof of Proposition~\ref{prop:nonzero_coeffs}, except using (2) instead of (1) in Proposition~\ref{prop:nonzero_coeffs}, 
that we have $\langle \wh{Q}(\odot^l \beta_{1,1},\beta_{3d-1}\rangle \neq 0$.
Action considerations then give $cd \geq 3d-1$, and since $d \in Z_{\geq 1}$ is arbitrary we get $c \geq 3$, which is a contradiction.

To justify the claim, first suppose that $a \geq 2$.
Then the inequality
\begin{align*}
\frac{3d_s - 1 - j_s + aj_s}{d_s} < 3
\end{align*}
gives $j_s(a-1) < 1$, and hence $j_s = 0$ as claimed.

Now suppose that we have $a < 2$.
Then the inequality
\begin{align*}
\frac{3d_s-1 - j_s + aj_s}{d_s} < a+1
\end{align*}
gives $j_s < \frac{d_s(a-2)+1}{a-1}$.
To conclude that $j_s = 0$, it suffices to observe that 
we have $\frac{d_s(a-2)+1}{a-1} \leq 1$,
which holds since we have $1 \leq a \leq 2$ and $d_s \geq 1$.

Finally, to establish the sharpness claim, note that this is a naive inclusion $P(1,a) \subset B^4(a+1)$, so this must be optimal for $a \leq 2$.
For $a \geq 2$, we instead use the symplectic embedding $P(1,a) \times \C^N \shookrightarrow B^4(3)\times \C^N$ which exists for all $N \geq 1$
by \cite[Thm. 1.3]{hind2015some}.
\end{proof}
\begin{remark}
In the case $a \geq 2$, Proposition~\ref{prop:poly_into_ball} is covered by \cite[Thm. 3.2]{hind2017stabilized}.
The combinatorics of our proof is formally similar to and inspired by the works \cite{hind2015symplectic,hind2019squeezing,hind2017stabilized}.
\end{remark}

\section{Enumerative implications}\label{sec:enum_imps}

{In this section we explore to what extent the formulas from \S\ref{sec:family} be interpreted as computations of enumerative invariants.
For concreteness we mostly restrict the discussion to four-dimensional ellipsoids.

\subsection{Ambiguities in $\Li$ homomorphisms}\label{subsec:ambiguities}

Consider the filtered $\Li$ algebra $\chlin(E(a,b))$.
The underlying $\K$-module is freely generated by the Reeb orbits of $\bdy E(a,b)$.
Assuming that $a$ and $b$ are rationally independent, these Reeb orbits are of the form $\gamma_{\sht;k}$ and $\gamma_{\lng;k}$
with actions $ak$ and $bk$ respectively, for $k \in \Z_{\geq 1}$.
If we write out the Reeb orbits in order of increasing action, the Conley--Zehnder\footnote{Here we are computing Conley--Zehnder indices with respect to a global trivialization of the contact vector bundle over $\bdy E(a,b)$. This trivialization is denoted by $\tau_{\op{ex}}$ in \cite{McDuffSiegel_counting}, and it has the property that the first Chern class term in the SFT index formula disappears for ellipsoids.} index of the $k$th one is $n-1 + 2k$.
For the filtered $\Li$ algebra $\chlin(E(a,b))$, the underlying $\K$-module is freely generated by the Reeb orbits of $\bdy E(a,b)$.
With our $\Li$ grading conventions as in \S\ref{subsec:recollections}, the grading of a Reeb orbit $\gamma$ is $n - \cz(\gamma) - 3$, where $n = 2$ is half the ambient dimension.
In the sequel we will sometimes implicitly identify $V^\can_{a,b}$ and $\chlin(E(a,b))$ by associating $A_i$ with the $i$th Reeb orbit in the above list.
This identification preserves the degree and action of generators, while all of the $\Li$ operations for both $\chlin(E(a,b))$ and $V^\can_{a,b}$ vanish for degree parity reasons.

Now suppose we have a symplectic embedding $E_{a,b} \times \C^N \shookrightarrow E_{a',b'} \times \C^N$ for some $N \in \Z_{\geq 0}$. By Theorem~\ref{thm:chscII_main} and its proof in \cite{chscII}, we have the following diagram of filtered $\Li$ homomorphisms, which commutes up to unfiltered $\Li$ homotopy:
\begin{align}\label{big_diagram}
\xymatrix{
\chlin(E(a',b'))\ar^{\Xi}[d] \ar@/^/^{\hspace{.8cm}G_{a',b'}}[r] & V_{a',b'}\ar^{\id}[d] \ar@/^/^{\hspace{.8cm}F_{a',b'}}[l] \ar@/^/^{\Phi_{a',b'}}[r] & V_{a',b'}^\can \ar^{\Phi_{a,b}\circ\Psi_{a',b'}}[d] \ar@/^/^{\Psi_{a',b'}}[l] \\
\chlin(E(a,b)) \ar@/^/^{\hspace{.8cm}G_{a,b}}[r] & V_{a,b} \ar@/^/^{\hspace{.8cm}F_{a,b}}[l] \ar@/^/^{\Phi_{a,b}}[r] & V_{a,b}^\can \ar@/^/^{\Psi_{a,b}}[l]
}
\end{align}
Here the left vertical map\footnote{Note that each of the arrows in \eqref{big_diagram} represents an $\Li$ homomorphism as in Definition~\ref{def:L-inf_homo}, i.e. a sequence of $k$-to-$1$ maps for $k \in \Z_{\geq 1}$, or alternatively a single map on the level of bar complexes.} is the SFT cobordism map (c.f. \cite[\S3.4]{HSC}), and the $\Phi$ and $\Psi$ maps are the ones we constructed in \S\ref{subsec:phi_and_psi}. The $F$ and $G$ maps also come from certain auxiliary SFT cobordism maps (see \S\ref{subsec:rounding} below).

Our aim is to understand the map $\Xi$, which, as we recall in \S\ref{subsec:rounding} below, enumerates (at least in favorable situations) curves in $E(a',b') \setminus E(a,b)$. 
The upshot of the above diagram is that the $\Li$ homomorphism $\Xi$ is identified with $$F_{a,b}\circ \Psi_{a,b} \circ (\Phi_{a,b} \circ \Psi_{a',b'}) \circ \Phi_{a',b'} \circ G_{a',b'}.$$
That is, we have $\Xi = \Phi_{a,b} \circ \Psi_{a',b'}$ {\em up to pre-composing and post-composing with filtered $\Li$ self homotopy equivalences of $\chlin(E(a',b'))$ and $\chlin(E(a,b))$ respectively.}

\sss

In order to better understand the above ambiguity in the filtered $\Li$ homomorphism $V_{a',b'}^\can \rightarrow V^\can_{a,b}$, it is convenient to introduce the following partial order on the basis elements of $\bar V^\can_{a,b}$ (and similarly for $\bar V^\can_{a',b'}$).
Note that this also induces a partial order on the basis elements of $\bar \chlin(E(a,b))$ via the identification
$\chlin(E(a,b)) \approx V^\can_{a,b}$.
\begin{definition}\label{def:partial_order}
We define a partial order on the basis elements of $\bar V^\can_{a,b}$ as follows.
Firstly, for basis elements $v_1 \odot \dots \odot v_k \in \bar V^\can_{a,b}$ and $v \in V^\can_{a,b} = \bar^{\leq 1}V^\can_{a,b} \subset \bar V^\can_{a,b}$, we put $v \leqq v_1 \odot \dots \odot v_k$ if the following two conditions hold:
\begin{enumerate}
\item $\sum_{i=1}^k \calA_{a,b}(v_i) \geq \calA_{a,b}(v)$
\item $\sum_{i=1}^k |v_i| = |v|$.
\end{enumerate}
More generally, for $v_1' \odot \dots \odot v_l' \in \bar V^\can_{a,b}$, we put $v_1' \odot \dots \odot v_l' \leqq v_1 \odot \dots \odot v_k$ if there exists a surjective set map $\{1,\dots,k\} \rightarrow \{1,\dots,l\}$ such that
for each $i \in \{1,\dots,l\}$ we have
$v_i \leqq \underset{j \in f^{-1}(i)}\odot v_j.$
\end{definition}
\begin{remark}
The conditions (1) and (2) above are precisely the action and index conditions needed for a filtered $\Li$ homomorphism $\chi: V^\can_{a,b} \rightarrow V^\can_{a,b}$ to have a nonzero structure coefficient $\langle \chi^k(v_1 \odot \dots \odot v_k),v\rangle$.
We note the similarity to the partial order defined in \cite{HuT}. 
\end{remark}

It turns out that we only need to worry about $\Li$ self homotopy equivalences which are the identity at the linear level:
\begin{definition}\label{def:linearly_identical}
An $\Li$ homomorphism $\chi$ from an $\Li$ algebra to itself
is {\bf linearly identical} if the linear term $\chi^1$ is the identity map.
\end{definition}

\NI We then have:
\begin{lemma}\label{lem:pre_and_post_comp}
Let $Q: V^\can_{a',b'} \rightarrow V^\can_{a,b}$ be a filtered $\Li$ homomorphism.
Assume that $v_1\odot \dots \odot v_k \in \bar V^\can_{a',b'}$ is minimal with respect to the above partial order, and similarly that $v \in \bar V^\can_{a,b}$ is maximal.
Then the structure coefficients $\langle Q^k(v_1 \odot \dots \odot v_k),v\rangle$ are unchanged if we pre-compose or post-compose $Q$ by linearly identical filtered $\Li$ self homotopy equivalences of $V^\can_{a',b'}$ and $V^\can_{a,b}$ respectively.
\end{lemma}

\subsection{Well-defined curve counts in cobordisms}\label{subsec:well_def_counts}

We now adopt a more geometric perspective and consider counts of punctured pseudoholomorphic curves in a symplectic cobordism of the form $E(a',b') \setminus E(a,b)$.
Let $J$ be a generic almost complex structure on the symplectic completion of $E(a',b') \setminus E(a,b)$.
Fix a collection of Reeb orbits $\Gamma^+  = (\gamma_1^+,\dots,\gamma^+_k)$ in $\bdy E(a',b')$ and $\Gamma^- = (\gamma_1^-,\dots,\gamma_l^-)$ in $\bdy E(a,b)$, and let $\calM^J_{E(a',b') \setminus E(a,b)}(\Gamma^+;\Gamma^-)$ denote the moduli space of genus zero $J$-holomorphic curves in $E(a',b') \setminus E(a,b)$ with $k$ positive ends asymptotic to $\gamma_1^+,\dots,\gamma_k^+$ and $l$ negative ends asymptotic to $\gamma_1^-,\dots,\gamma_l^-$.
More precisely, following \cite[\S3]{HSC}, $\calM^J_{E(a',b')\setminus E(a,b)}(\Gamma^+;\Gamma^-)$ is defined with the following features:
\begin{itemize}
	\item each puncture of a curve has a freely varying asymptotic marker which is required to map to a chosen basepoint on the image of the corresponding Reeb orbit
\item those punctures asymptotic to the same Reeb orbit are ordered
\item each curve is {\em unparametrized}, i.e. we quotient by the group of biholomorphic reparametrizations.
\end{itemize}

Near any somewhere injective curve $u$, $\calM^J_{E(a',b') \setminus E(a,b)}(\Gamma^+;\Gamma^-)$ is a smooth manifold of dimension 
\begin{align}
\ind\;\calM^J_{E(a',b') \setminus E(a,b)}(\Gamma^+;\Gamma^-) = (2-3)(2-k-l) + \sum_{i=1}^k \cz(\gamma_i^+) - \sum_{j=1}^l \cz(\gamma_j^-).
\end{align}
On the other hand, multiply covered curves in $E(a',b') \setminus E(a,b)$ tend to appear with higher-than-expected dimension, necessitating abstract perturbations to achieve transversality.
In general the SFT compactification $\ovl{\calM}^J_{E(a',b') \setminus E(a,b)}(\Gamma^+;\Gamma^-)$ (see \cite[\S3.3]{HSC}) will include boundary strata consisting of pseudoholomorphic buildings with negative expected dimension. Examples \ref{ex:bad_ex1} and \ref{ex:bad_ex2} at the end of this subsection illustrate situations where a naive counting of regular curves is not available.

\sss

Recall that the $k$-to-$1$ part $\Xi^k$ of the cobordism map $\Xi: \chlin(E(a',b')) \rightarrow \chlin(E(a,b))$ counts index $0$ rational curves in $E(a',b') \setminus E(a,b)$ with $k$ positive ends and one\footnote{Strictly speaking we should count {\em anchored} curves, but anchors do not appear for ellipsoids since the natural contact forms on their boundaries are dynamically convex.} negative end. 
More precisely, for $\Gamma^+ = (\gamma_1^+,\dots,\gamma_k^+)$ a collection of Reeb orbits in $\bdy E(a',b')$ and $\Gamma^- := (\gamma^-)$ a single Reeb in $\bdy E(a,b)$, assume that we have $\ind\,\calM^J_{E(a',b') \setminus E(a,b)}(\Gamma^+;\Gamma^-) = 0$.
We introduce the following combinatorial factors:
\begin{itemize}
	\item $\kappa_{\gamma_i^+}$ is the covering multiplicity of the Reeb orbit $\gamma_i^+$, and we put $\kappa_{\Gamma^+} := \kappa_{\gamma^+_1}\dots \kappa_{\gamma^+_k}$
	\item $\mu_{\Gamma^+}$ is the number of ways of ordering the punctures asymptotic to each Reeb orbit (see \cite[\S3.4.1]{HSC}).
\end{itemize}
By definition, the structure coefficients are given\footnote{In this paper we find it convenient to use a slightly different convention with respect to the $\kappa_{\gamma}$ factors compared to \cite[\S3.4.2]{HSC}, which instead puts $\langle \Xi^k(\gamma_1^+ \odot \dots \odot \gamma_k^+),\gamma^-\rangle = \frac{1}{\kappa_{\Gamma^-}}
\#\ovl{\calM}^J_{E(a',b') \setminus E(a,b)}(\Gamma^+;\Gamma^-)$. These differ by the change of basis $\gamma \leftrightarrow \kappa_{\gamma}\gamma$.} by
\begin{align}\label{eq:str_coeff}
\langle \Xi^k(\gamma_1^+ \odot \dots \odot \gamma_k^+),\gamma^-\rangle = \frac{1}{\kappa_{\Gamma^+}}
\#\ovl{\calM}^J_{E(a',b') \setminus E(a,b)}(\Gamma^+;\Gamma^-).
\end{align}

 \begin{remark}\label{rmk:counting_sw_and_mult}
If $u$ is a somewhere injective curve in $\calM^J_{E(a',b')\setminus E(a,b)}(\Gamma^+;\Gamma^-)$, then the underlying curve with unordered punctures and without asymptotic markers makes a total contribution of $\kappa_{\gamma^-}\mu_{\Gamma^+}$ to the structure coefficient $\langle \Xi^k(\gamma_1^+,\dots,\gamma_k^+),\gamma^-\rangle$.
This is because there are precisely $\kappa_{\gamma_i^+}$ possible placements of the asymptotic marker at the $i$th positive puncture, and the number of possible orderings of the positive punctures is precisely $\mu_{\Gamma^+}$.
More generally, if $u$ is a regular multiply covered curve in $\calM^J_{E(a',b')\setminus E(a,b)}(\Gamma^+;\Gamma^-)$ with covering multiplicity $\mult_u$, then each underlying curve after ignoring the asymptotic markers and orderings of the punctures contributes $\kappa_{\gamma^-} \mu_{\Gamma^+} / \mult_u$.  
\end{remark}

\begin{remark}
The structure coefficients of the cobordism map $\Xi$ are only canonically defined up to pre-composing and post-composing with linearly identical filtered $\Li$ self homotopy equivalences of $\chlin(E(a',b'))$ and $\chlin(E(a,b))$ respectively.
Indeed, although negative index curves in the symplectizations of $\bdy E(a',b')$ and $\bdy E(a,b)$ do not arise, we do have index $0$ multiple covers of trivial cylinders.
These curves can appear in the SFT compactifications of moduli spaces of curves with two or more positive ends, and hence lead to ambiguities in counting problems (c.f. \cite[Rmk. 3.2.2]{McDuffSiegel_counting} and Example~\ref{ex:bad_ex2} below).
On the other hand, since there is at most one Reeb orbit of $\bdy E(a,b)$ or $\bdy E(a',b')$ in any given degree, these ambiguities do not arise at the linear level.
\end{remark}

\sss

In favorable situations, one can show that the compactified moduli space $\ovl{\calM}^J_{E(a',b') \setminus E(a,b)}(\Gamma^+;\Gamma^-)$ consists only of regular curves for generic $J$, and hence can be defined without recourse to any virtual perturbation techniques. 
The idea is to use the fact that somewhere injective curves are regular for generic $J$, and then use action and index considerations to argue that no multiple covers or nontrivial pseudoholomorphic buildings can appear. One can then use the SFT compactness theorem to argue that the count is finite and independent of $J$ (provided that it is generic). The following two lemmas illustrate this point.

\begin{lemma}[\cite{Mint,McDuffSiegel}]\label{lem:well_def_count1}
Let $\Gamma^+ = (\gamma_1^+,\dots,\gamma_k^+)$ be a collection of Reeb orbits in $\bdy E(a',b')$ which is {\em minimal} with respect to partial order from Definition~\ref{def:partial_order}.
Suppose that $\gamma^- = \gamma_{\sht;m}$ is the $m$-fold iterate of the short simple Reeb orbit of $\bdy E(a,b)$ for some $m \in \Z_{\geq 1}$ such that $m < b/a$. Let $J$ be a generic SFT-admissible almost complex structure on the symplectic completion of $E(a',b') \setminus \eps E(a,b)$ for $\eps > 0$ sufficiently small.
Then assuming its index is zero, the compactified moduli space
$\ovl{\calM}^J_{E(a',b') \setminus \eps E(a,b)}(\Gamma^+;(\gamma^-))$ consists entirely of regular curves
and coincides with the uncompactified moduli space $\calM^J_{E(a',b') \setminus \eps E(a,b)}(\Gamma^+;(\gamma^-))$. Moreover, the count $\#\calM^J_{E(a',b') \setminus \eps E(a,b)}(\Gamma^+;(\gamma^-))$ is finite and independent of the choice of $\eps $ and generic $J$.
\end{lemma}

\begin{lemma}[\cite{McDuffSiegel}]\label{lem:well_def_count2}
Let $\Gamma^+ = (\gamma_{\lng},\dots,\gamma_{\lng})$ be a collection of $d$ copies of the long simple Reeb orbit $E(1,1+\delta)$ for $\delta > 0$ sufficiently small.
Suppose that $\gamma^-$ is a Reeb orbit of $\bdy E(a,b)$ which is {\em maximal} with respect to the partial order from Definition~\ref{def:partial_order}.
Let $J$ be a generic SFT-admissible almost complex structure on the symplectic completion of $E(1,1+\delta) \setminus \eps E(a,b)$ for $\eps > 0$ sufficiently small.
Then assuming its index is zero, the compactified moduli space
$\ovl{\calM}^J_{E(1,1+\delta) \setminus \eps E(a,b)}(\Gamma^+;(\gamma^-))$ consists entirely of regular curves
and coincides with the uncompactified moduli space $\calM^J_{E(1,1+\delta) \setminus \eps E(a,b)}(\Gamma^+;(\gamma^-))$. Moreover, the count $\#\calM^J_{E(1,1+\delta) \setminus \eps E(a,b)}(\Gamma^+;(\gamma^-))$ is finite and independent of the choice of $\delta,\eps$ and generic $J$.
\end{lemma}
\begin{remark} \label{rmk:closing_up_long_orbits}\hspace{1cm}
\begin{enumerate}
\item In Lemma~\ref{lem:well_def_count1}, the condition $m < b/a$ means we can replace $E(a,b)$ with the skinny ellipsoid $E_\sk$. These are precisely the types of counts appearing in the definition of $\gapac_\bb(E(a,b))$.
\item In Lemma~\ref{lem:well_def_count2}, we can equivalently count degree $d$ $J$-holomorphic planes in $\CP^2 \setminus E(a,b)$ with negative end asymptotic to $\gamma^-$. These are precisely the types of counts appearing in the works \cite{HK, CGH, Ghost, Mint} on the restricted stabilized ellipsoid embedding problem.	
\end{enumerate}
\end{remark}

\sss

The following lemma, which is essentially a special case of Proposition~\ref{prop:Mint_counts} from the introduction, gives a useful setting in which Lemma~\ref{lem:well_def_count2} applies. We provide a proof for the sake of completeness.
\begin{lemma}\label{lem:minimality}
Suppose that $p+q = 3d$ for some $p,q,d \in \Z_{\geq 1}$. 
Assume also that there is no partition $k_1 + \dots + k_m = p$ of $p$
for some $m \in \Z_{\geq 2}$ and $k_1,\dots,k_m \in \Z_{\geq 1}$ such that
$\sum_{i=1}^m (k_i + \lceil k_iq/p\rceil) = 3d$.
Then for $x = p/q+\delta$ with $\delta > 0$ sufficiently small,
$\gamma_{\sht;p}$ is maximal in $\bar \chlin(E(1,x))$ with respect to the partial order from Definition~\ref{def:partial_order}. 
\end{lemma}

\begin{remark}
Note that the hypothesis about no such partitions existing holds for example if $\gcd(p,q) = 1$. Indeed, such a partition can only exist if $k_iq/p$ is an integer for $i = 1,\dots, m$, since otherwise we have 
$$\sum_{i=1}^m (k_i + \lceil k_iq/p\rceil) > \sum_{i=1}^m(k_i + k_iq/p) = p + q = 3d.$$
If $p$ and $q$ are relatively prime, this necessitates $k_1\dots,k_m \geq p$, and hence $m=1$.
\end{remark}
\begin{proof}[Proof of Lemma~\ref{lem:minimality}]
The actions of Reeb orbits in $\bdy E(1,x)$ are given by
\begin{align*}
\calA_{1,x}(\gamma_{\sht;k}) = k,\;\;\;\;\; \calA_{1,x}(\gamma_{\lng;k}) = kx,\;\;\;\;\; k \in \Z_{\geq 1}.
\end{align*}
With our conventions, the degree of the element in $V^\can_{a,b}$ corresponding to the Reeb orbit $\gamma$ is given by $n - \cz(\gamma) - 3$ with $n = 2$, and we have:
\begin{align*}
\cz(\gamma_{\sht;k}) &= 1 + 2(k + \lfloor k/x\rfloor) \\
\cz(\gamma_{\lng;k}) &= 1 + 2(k + \lfloor kx\rfloor)
\end{align*}
for $k \in \Z_{\geq 1}$.

Now suppose that, for some $a,b \in \Z_{\geq 1}$ and $k_1,\dots,k_a,l_1,\dots,l_b \in \Z_{\geq 1}$, 
the following index and action conditions hold:
\begin{enumerate}
\item
$\sum_{i=1}^a\left( k_i + \lfloor k_i/x\rfloor\right) + \sum_{j=1}^b\left(l_i + \lfloor l_i x\rfloor\right) + a + b - 1 = 3d-1$
\item
$\sum_{i=1}^a k_i + \sum_{j=1}^b xl_j \geq p$
\end{enumerate}
It suffices to show that we must have $b = 0$ and $a = 1$ with $k_1 = p$.

We have
\begin{align*}
3d &= \sum_{i=1}^a\left( k_i + \lceil k_i/x\rceil \right) + \sum_{j=1}^b\left(l_j + \lceil l_jx\rceil \right)\\
&> (1+1/x)\left( \sum_{i=1}^ak_i + \sum_{j=1}^b xl_j\right)\\
&\geq (1+1/x)p = p + (1+\delta q/p)^{-1}q.
\end{align*}
Observe that $l_j(p/q+\delta)$ is not an integer for $\delta$ sufficiently small,
and so we have
$$\lceil l_j x \rceil - l_jx = 1+ \lfloor l_jx\rfloor - l_jx,$$
which approaches $1 + \lfloor l_jp/q\rfloor - l_jp/q > 0$ as $\delta \rightarrow 0$.
This shows that the strict inequality above is false for $\delta > 0$ sufficiently small, unless we have $b = 0$.

\end{proof}

\sss

To end this subsection, we give some examples which help clarify the necessity of the assumptions in Lemmas \ref{lem:well_def_count1} and \ref{lem:well_def_count2}.
\begin{example}\label{ex:bad_ex1}
Consider the moduli space $\calM^J_{E(1,1+\delta) \setminus \eps E(1,1+\delta')}(\Gamma^+;\Gamma^-)$ with $\Gamma^+ = (\gamma_{\sht},\gamma_{\sht})$ and
$\Gamma^- = (\gamma_{\sht;2})$ for $\delta,\delta' > 0$ sufficiently small, which has expected dimension zero.
Since there is a cylinder in $E(1,1+\delta) \setminus \eps E(1,1+\delta')$ positively asymptotic to $\gamma_{\sht}$ and negatively asymptotic to $\gamma_{\sht}$, any double branched cover with one branched point at the negative puncture and one branched point in the interior gives rise to an element of $\calM^J_{E(1,1+\eps) \setminus \eps E(1,1+\eps')}(\Gamma^+;\Gamma^-)$.
Since these covers appear in a two-dimensional family due to the moveable branch point, this shows that the moduli space $\calM^J_{E(1,1+\delta) \setminus \eps E(1,1+\delta')}(\Gamma^+;\Gamma^-)$ appears with higher-than-expected dimension.
\end{example}

\begin{example}\label{ex:bad_ex2}
Consider the moduli space $\calM^J_{E(1,1+\delta) \setminus \eps E(1,3+\delta')}(\Gamma^+,\Gamma^-)$ with $\Gamma^+ = (\gamma_{\sht},\gamma_{\sht})$ and $\Gamma^- = (\gamma_{\sht;3})$, for $\delta,\delta' > 0$ sufficiently small, which has expected dimension zero.
A sequence of curves in this moduli space could in principle degenerate into a two level pseudoholomorphic building with
\begin{itemize}
\item top level in the symplectization $\R \times \bdy E(1,1+\delta)$ consisting of a rational curve with two positive ends both asymptotic to $\gamma_{\sht}$ and one negative end asymptotic to $\gamma_{\sht;2}$
\item bottom level in $E(1,1+\delta) \setminus \eps E(1,3+\delta')$ consisting of a cylinder with positive end asymptotic to $\gamma_{\sht;2}$ and negative end asymptotic to $\gamma_{\sht;3}$.
\end{itemize}
Note that the curves in both levels have index zero.
\end{example}

\subsection{Rounding and partially compactifying}\label{subsec:rounding}

The starting point for the proof of Theorem~\ref{thm:chscII_main} is to replace the convex toric domain $X_\Omega$ with a ``fully rounded'' convex toric domain $X_{\wt{\Omega}}$ of the same dimension. 
The following definition appears implicitly in \cite[Lem. 2.7]{Gutt-Hu}:
\begin{definition}
Let $X_{\Omega} \subset \C^n$ be a convex toric domain, and put $\Sigma := \bdy \Omega \cap \R_{> 0}^n$. 
We say that $X_{\Omega}$ is {\bf $\eps$ fully rounded} if 
\begin{enumerate}
\item $\Sigma$ is a smooth hypersurface in $\R^n$
\item the Gauss map $\gauss:\Sigma \rightarrow S^{n-1}$ is a smooth embedding
\item $\bdy X_{\Omega}$ is a smooth hypersurface 
\item for $i \in \{1,\dots n\}$ and any point $p \in \Sigma \cap \{(z_1,\dots,z_n) \in \C^n\; :\; z_i = 0\}$,
the $i$th component of $\gauss(p)$ is less than $\eps$.
\end{enumerate}
\end{definition}

We typically take $\eps > 0$ to be some sufficiently small constant and suppress it from the notation, in which case we simply say that $X_{\Omega}$ is {\bf fully rounded}.
Given any convex toric domain $X_\Omega$, we can replace its moment map image $\Omega$ with a $C^0$-small perturbation $\wt{\Omega}$ such that $X_{\wt{\Omega}}$ is fully rounded.
Figure~\ref{fig:tri_smooth} illustrates this in the case of a four-dimensional ellipsoid $X_\Omega = E(a,b)$, with its fully rounded version denoted by $\wt{E}(a,b)$.
\begin{figure}
  \includegraphics[scale=2]{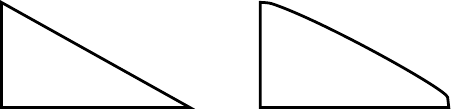}
  \caption{Perturbing the ellipsoid $E(a,b)$ to the fully rounded convex toric domain $\wt{E}(a,b)$.}
  \label{fig:tri_smooth}
\end{figure}
From the point of view of capacities or symplectic embedding obstructions this perturbation has essentially no effect. On the other hand, the fully rounding process can have a drastic effect on the Reeb dynamics of $\bdy X_{\Omega}$.
Indeed, whereas $\bdy E(a,b)$ has just two simple Reeb orbits (assuming $a$ and $b$ are rationally independent), $\bdy \wt{E}(a,b)$ has infinitely many simple Reeb orbits. In general, for each $p \in \Sigma$ such that $\gauss(p)$ is a rational direction, the moment map fiber $\mu^{-1}(p)$ is fibered by simple Reeb orbits forming a $\mathbb{T}^{n-1}$-family (see \cite{hutchings2016beyond,Gutt-Hu} for more details). Here we say that a vector in $S^{n-1}$ is a {\bf rational direction} if it is a positive rescaling of a vector in $\Z^n$, and we denote the set of rational directions in $S^{n-1}$ by $S^{n-1}_\Q$.
	
We fix a preferred perfect Morse function $f_{\mathbb{T}^{n-1}}: \mathbb{T}^{n-1} \rightarrow \R$. 
Using the perturbation scheme\footnote{Alternatively, we could directly apply Morse--Bott techniques to these families of generators.}described in \cite[Lem. 5.4]{hutchings2016beyond}, we can find a $C^0$-close Liouville domain whose simple Reeb orbits (up to an arbitrarily high action cutoff) are of the form $\gamma_{(q,c)}$, indexed by pairs $(q,c)$ such that:
\begin{enumerate}
	\item $q = (q_1,\dots,q_n) \in \Z_{\geq 0}^n$ is nonzero and primitive 
	\item $c$ is a critical point of $f_{\mathbb{T}^{n-1}}$.
	\end{enumerate}
Moreover, up to a small discrepancy, the action of $\gamma_{(q,c)}$ is given by $||q||_{\Omega}^*$, and its Conley--Zehnder index is given by
$\cz(\gamma_{(q,c)}) = n-1 - |c| + 2\sum_{i=1}^n q_i$, where $|c|$ denotes the Morse index of the critical point $c$.
We can formally extend this description of the simple Reeb orbits to all Reeb orbits by allowing pairs $(q,c)$ with $q \in \Z_{\geq 0}^n$ nonzero but not necessarily primitive, and we still have $\calA(\gamma_{(q,c)}) = ||q||_{\Omega}^*$ and $\cz(\gamma_{(q,c)}) = n-1 -|c|+ 2\sum_{i=1}^n q_i$.
From now on we will assume that such a perturbation has been performed and we suppress it from the notation.

Although the fully rounding process can introduce additional Reeb orbits, it turns out that the relevant curves become easier to describe. 
Indeed, at least in the four-dimensional case, the above description gives a natural bijective correspondence between Reeb orbits of $\bdy X_{\wt{\Omega}}$ (up to an arbitrarily large action cutoff) and basis elements of $V_{\Omega}$.
Moreover, for any $X_\Omega$ 
this correspondence preserves degree and (up to small discrepancies) action.
The proof of Theorem~\ref{thm:chscII_main} in \cite{chscII} extends this identification on basis elements to the level of curves by deforming $X_{\wt{\Omega}}$ to a situation where curves can be explicitly enumerated.

\sss

Up to minor rescalings, we have inclusions $X_\Omega \subset X_{\wt{\Omega}}$ and $X_{\wt{\Omega}} \subset X_\Omega$.
These induce filtered $\Li$ homomorphisms 
$\chlin(X_{\wt{\Omega}}) \rightarrow \chlin(X_\Omega)$ and $\chlin(X_\Omega) \rightarrow \chlin(X_{\wt{\Omega}})$
which are identified with the maps $F_\Omega$ and $G_\Omega$ respectively from \S\ref{subsec:ambiguities}.

Now suppose that we have an ellipsoid embedding $E(a,b) \shookrightarrow E(a',b')$. For enumerative purposes, we can assume this is an inclusion, after possibly replacing $E(a,b)$ with $\eps E(a,b)$ for $\eps > 0$ sufficiently small.
Up to the ambiguities considered in Lemma~\ref{lem:pre_and_post_comp},  the $\Li$ homomorphism $\Phi_{a,b} \circ \Psi_{a',b'}: V^\can_{a',b'} \rightarrow V^\can_{a,b}$ coincides with the induced cobordism map $\Xi: \chlin(E(a',b')) \rightarrow \chlin(E(a,b))$, provided that we choose the constants $C_{q;a,b}$ and $C_{q;a',b'}$ in Construction~\ref{phi_psi_constr} so that we have agreement at the linear level, i.e. $(\Phi_{a,b} \circ \Psi_{a',b'})^1 = \Xi^1$. 
Since $\Phi_{s,t}$ and $F_{s,t}$ are $\Li$ homotopy inverses of $\Psi_{s,t}$ and $G_{s,t}$ respectively, 
and $\calA_{s,t}(\beta_{i,k-i}) > \calA_{s,t}(A_{k})$ unless $\beta_{i,k-i} = \beta_{\ii(k),\jj(k)}$,
it suffices to choose the constants $C_{q;s,t}$ such that $$\langle G^1_{s,t}(A_k),\beta_{\ii(k),\jj(k)}\rangle = \langle \Psi^1_{s,t}(A_k),\beta_{\ii(k),\jj(k)}\rangle$$ for all $s,t \in \R_{> 0}$ and $q \in \Z_{\geq 1}$.

In \S\ref{subsec:unique_cylinders}, we show that there is a unique cylinder $u$ (ignoring asymptotic markers) in the symplectic cobordism $E(s,t) \setminus \eps \wt{E}(s,t)$ which is positively asymptotic to $A_k$ and negatively asymptotic to $\beta_{\ii(k),\jj(k)}$. We denote by $\kappa_{A_k}$ the covering multiplicity of the Reeb orbit corresponding to $A_k$ in $\bdy E(s,t)$.
We have:
	\begin{itemize}
		\item the covering multiplicity of the Reeb orbit $\beta_{\ii(k),\jj(k)}$ at the negative end of $u$ is $\gcd(\ii(k),\jj(k))$
		\item the covering multiplicity of $u$ is $\gcd(\ii(k),\jj(k),\kappa_{A_k})$.
	\end{itemize}
Using the conventions described in Remark~\ref{rmk:counting_sw_and_mult}, this translates into
$$\langle G^1_{s,t}(A_k),\beta_{\ii(k),\jj(k)}\rangle = \frac{\gcd(\ii(k),\jj(k))}{\gcd(\ii(k),\jj(k),\kappa_{A_k})}.$$
We have thus proved the following theorem.
Recall that we are identifying each Reeb orbit $\gamma$ in $\bdy E(s,t)$ with the corresponding generator $A_i \in V^\can_{s,t}$ with $|A_i| = n-3 - \cz(\gamma)$, and that the structure coefficients of $\Xi$ are given by \eqref{eq:str_coeff}.
\begin{thm}
 In Construction~\ref{phi_psi_constr}, put\footnote{We emphasize that $\kappa_{A_k}$ and the pair $(\ii(k),\jj(k))$ depend sensitively on $s,t$. More precisely, they depend on the ratio $t/s$. Here we add the $s,t$ superscripts into the notation as a reminder.} $C_{q;s,t} := \frac{\gcd(\ii^{s,t}(q),\jj^{s,t}(q))}{\gcd(\ii^{s,t}(q),\jj^{s,t}(q),\kappa_{A^{s,t}_q})}$
 for all $s,t \in \R_{> 0}$ and $q \in \Z_{\geq 1}$.
 Then in the context of Lemma~\ref{lem:well_def_count1} we have
 $$\tfrac{1}{\kappa_{\Gamma^+}} \# \calM^J_{E(a',b') \setminus \eps E(a,b)}((\gamma_1^+,\dots,\gamma_k^+);(\gamma_{\sht;m})) = \langle (\Phi_{a,b}\circ \Psi_{a',b'})^k(\gamma_1^+,\dots,\gamma_k^+),A_m\rangle.$$
 Similarly, in the context of Lemma~\ref{lem:well_def_count2}, we have
 $$\# \calM^J_{E(1,1+\delta) \setminus \eps E(a,b)}((\underbrace{\gamma_\lng,\dots,\gamma_\lng}_d);(\gamma^-)) = \langle (\Phi_{a,b}\circ \Psi_{1,1+\delta})^d(A_2,\dots,A_2),\gamma^-\rangle.$$
\end{thm}

\subsection{Uniqueness of cylinders}\label{subsec:unique_cylinders}

Our goal is to characterize the cylinders
in the symplectic cobordism $E(s,t) \setminus \eps\wt{E}(s,t)$. 
For each $k \in \Z_{\geq 1}$, a Fredholm index zero cylinder with positive asymptotic $A_k$ must have negative asymptotic $\beta_{\ii(k),\jj(k)}$ by action and index considerations. 
Similarly, there cannot be any pseudoholomorphic building with more than one level representing such a cylinder.
We will show that there is a unique such cylinder $u$ for each $k \in \Z_{\geq 1}$, and this is a multiple cover if $\gcd(\ii(k),\jj(k),\kappa_{A_k}) > 1$.

The basic idea is to apply the relative adjunction formula for a punctured curve $u$:
$$ c_\tau(u) = \chi(u) + Q_\tau(u) + \wr_\tau(u) - 2\delta(u).$$ 
Here the subscripted invariants depend on a choice $\tau$ of trivialization of the contact vector bundle over each asymptotic Reeb orbit of $u$: $c_\tau(u)$ is the relative first Chern class of $u$, $\chi(u)$ is the Euler characteristic of $u$, $Q_\tau(u)$ is the relative intersection pairing of $u$, $\wr_\tau(u)$ is the difference of writhes at the top and bottom of $u$, and $\delta$ is a count of singularities of $u$.
We refer the reader to \cite[\S 3.3]{Hlect} for more details.

Using writhe bounds (see \cite[\S3.2]{McDuffSiegel_counting} and the references therein), we first show in Lemma~\ref{lem:no_mult_cover_cyl} that $\gcd(\ii(k),\jj(k),\kappa_{A_k}) > 1$ contradicts the relative adjunction inequality, meaning that $u$ cannot be somewhere injective. 
We then consider the somewhere injective case, and in Lemma~\ref{lem:cyl_unique} we argue similarly for a union of two cylinders with the same asymptotics to conclude that $u$ must be unique.

\sss

Suppose that $u$ is a cylinder as above. We put $m := \kappa_{A_k}$, so that the Reeb orbit $\gamma$ in $\bdy E(s,t)$ corresponding to $A_k$ is either $\gamma_{\sht;m}$ or $\gamma_{\lng;m}$.
Let $\theta$ denote the rotation angle of $\gamma$ 
(see \cite[\S3.2]{Hlect}).
Since $\ind(u) = 0$, the negative asymptotic is then either $\beta_{i,j} = \beta_{m,\lfloor m\theta\rfloor}$ or $\beta_{i,j} = \beta_{\lfloor m\theta \rfloor,m}$ respectively.
 According to \cite[\S5.3]{hutchings2016beyond}, we can take the rotation angle of $\beta_{i,j}$ to be positive and arbitrarily close to zero.
Note that $\beta_{i,j}$ is a $g$-fold cover of its underlying simple orbit, where we put $g := \gcd(i,j) = \gcd(m,\lfloor m\theta \rfloor)$.

\begin{lemma}\label{lem:no_mult_cover_cyl}
If the cylinder $u$ is somewhere injective, then the Reeb orbit $\beta_{i,j}$ must be simple.
\end{lemma}
\begin{proof}
Using the ``split'' trivialization $\tau_{\op{sp}}$ for $\bdy E(s,t)$ from \cite[\S3.2]{McDuffSiegel_counting} and the trivialization $\tau$ for $\wt{E}(s,t)$ from \cite{hutchings2016beyond}, we have
\begin{itemize}
\item
$c(\gamma) = m$
\item
$c(\beta_{i,j}) = i + j = m + \lfloor m\theta \rfloor$
\item $\chi(u) = 0$
\item $Q(\gamma) = 0$
\item $Q(\beta_{i,j}) = ij = m\lfloor m\theta \rfloor$
\item $\wr^+(u) \leq \lfloor m\theta\rfloor (m-1)$
\item $\wr^-(u) \geq g-1$.
\end{itemize}
The relative adjunction inequality then gives
\begin{align*}
\wr^+(u) - \wr^-(u) \geq c(\gamma) - c(\beta_{i,j}) - \chi(u) - Q(u),
\end{align*}
so we must have
\begin{align*}
\lfloor m\theta\rfloor (m-1) - (g-1) \geq - \lfloor m\theta \rfloor + m\lfloor m\theta \rfloor,
\end{align*}
which is a contradiction unless $g = 1$.
\end{proof}

Similarly, we prove that when $u$ is somewhere injective, it is the unique representative of its moduli space.
\begin{lemma}\label{lem:cyl_unique}
Suppose we have $g = 1$, and that $u$ and $u'$ are two cylinders with the same asymptotics as above.
Then $u = u'$.
\end{lemma}
\begin{proof}
Let $C$ denote the union of $u$ and $u'$.
For this disconnected curve we have
\begin{itemize}
\item
$c(C) = 2m - 2(m + \lfloor m\theta \rfloor)$
\item $Q(C) = 0 - 4m\lfloor m\theta \rfloor$ (see \cite[\S5.3]{hutchings2016beyond})
\item $\chi(C) = 0$
\item $\wr^+(C) \leq 4m\lfloor m\theta \rfloor - 2\lfloor m\theta \rfloor$
\item $\wr^-(C) \geq 4g-2$
\end{itemize}
From the relative adjunction inequality we get
\begin{align*}
4m\lfloor m\theta \rfloor - 2\lfloor m\theta \rfloor - (4g-2) \geq 2m - 2(m + \lfloor m\theta \rfloor) + 4m\lfloor m\theta \rfloor,
\end{align*}
i.e.
\begin{align*}
4g-2 \leq 0,
\end{align*}
which is a contradiction.
\end{proof}

\addtocontents{toc}{\protect\setcounter{tocdepth}{2}} 

\subsection{Mirror symmetry interpretation}\label{subsec:ms}

In this somewhat speculative subsection, we give an alternative description of the filtered $\Li$ algebra $V_\Omega$ and extend its definition to arbitrary dimensions. We also give a sketch proof of Theorem~\ref{thm:chscII_main} based on some expected structural properties of symplectic cohomology. 
This approach could be viewed as a version of quantitative closed string mirror symmetry for convex toric domains in $\C^n$.
At the end we arrive at an explicit algebraic description of $V_\Omega$. However, the proof of Theorem~\ref{thm:chscII_main} in \cite{chscII} instead computes $\chlin(X_\Omega)$ by a  more direct curve counting argument.

Let $X_\Omega \subset \C^n$ be a convex toric domain. We formulate this approach using $\chlinsc(X_\Omega)$ in place of $\chlin(X_\Omega)$. 
Our sketch computation is based on the following steps:
\begin{itemize}
    \item [\bf{Step 1}] Compute $\sc(D_\Omega)$ as a filtered homotopy Batalin--Vilkovisky (BV) algebra, where $D_\Omega$ is a smooth Lagrangian torus fibration over $\Omega$ which partially compactifies to $X_\Omega$.
    \item [\bf{Step 2}] Quotient out the action zero generators and perform an algebraic $S^1$-quotient to arrive at the filtered $\Li$ algebra $\sc_{S^1,+}(D_\Omega)$.
    \item [\bf{Step 3}] Deform $\sc_{S^1,+}(D_\Omega)$ by the Cieliebak--Latschev Maurer--Cartan element $m \in \sc_{S^1,+}(D_\Omega)$ to obtain $\sc_{S^1,+,m}(D_\Omega)$. 
    \item [\bf{Step 4}] Finally, after passing to a certain $\Li$ subalgebra corresponding to ``first orthant generators'', we obtain a model for $\sc_{S^1,+}(X_\Omega)$.
\end{itemize}

\NI We now elaborate on each of these steps.  

\subsubsection{The smooth Lagrangian torus fibration $D_\Omega$ and its BV algebra structure}

Firstly, in order to define $D_\Omega$, consider the translation $\Omega + \frakt$ of $\Omega$ by a vector $t \in \R_{>0}^n$ which is small in each coordinate, and put
$$D_\Omega := \mu^{-1}(\Omega + \mathfrak{t}).$$
Note that $D_\Omega$ is fibered by the smooth Lagrangian tori $\mu^{-1}(p)$ for $p \in \Omega + \frakt$, and (after a slight shrinkening) we have a natural inclusion $D_\Omega \subset X_\Omega$ (see e.g. \cite{landry2015symplectic}).

The Liouville domain $D_\Omega$ can be viewed as a unit disk cotangent bundle of $\mathbb{T}^n$ with respect to a Finsler norm determined by $\Omega$.  
In particular, it is Liouville deformation equivalent to the standard unit disk cotangent bundle $D^*\mathbb{T}^n$, and by Viterbo's theorem \cite{abouzaid2013symplectic} we have an isomorphism of BV algebras
$$\sh^*(D_\Omega) \cong H_{n-*}(\calL \mathbb{T}^n).$$
We note that $T^*\mathbb{T}^n$ is symplectomorphic to $(\C^*)^n$, whose SYZ mirror is itself, and accordingly we have an isomorphism
$$SH^*(D_\Omega) \cong \pvf((\C^*)^n),$$
where $\pvf((\C^*)^n)$ denotes the BV algebra of algebraic polyvector fields on $(\C^*)^n$ with coefficients in $\K$.
Here the BV operator on $\pvf((\C^*)^n)$ is of the form 
$$\Delta = \vol(-)^{-1} \circ d \circ \vol(-),$$
where $\vol$ denotes the volume form $\frac{1}{z_1\dots z_n}dz_1 \wedge \dots \wedge dz_n$ on $(\C^*)^n$, $\vol(-)$ denotes the induced isomorphism sending a $k$-vector field on $(\C^*)^n$ to an $(n-k)$-form, and $\vol(-)^{-1}$ denotes the inverse isomorphism.
Recall that the product and BV operator together determine a Lie bracket (namely the Schouten bracket on polyvector fields) via the BV relation
$$ \{a,b\} = \Delta(a\cdot b) - \Delta(a)\cdot b - (-1)^{|a|}a\cdot \Delta(b).$$

In the case $n = 1$, $\pvf(\C^*)$ has a basis of the form $\{z^k,z^k\bdy_{z} \;:\; k \in \Z\}$, where $z$ is the coordinate on $\C^*$. With our $\Li$ conventions from \S\ref{subsec:recollections}, the functions $z^k$ have degree $-2$ and the vector fields $z^k \bdy_z$ have degree $-1$.
The product is characterized by $$(z^k)\cdot (z^l) = z^{k+l},\;\;\;\;\; 
z^k \cdot z^l\bdy_z = z^l\bdy_z \cdot z^k = z^{k+l}\bdy_z,\;\;\;\;\; z^k \bdy_z \cdot z^l \bdy_z = 0.$$
The BV operator $\Delta: \pvf(\C^*) \rightarrow \pvf(\C^*)$ is characterized by $$\Delta(z^k\bdy_z) = (k-1)z^{k-1},\;\;\;\;\;\Delta(z^k) = 0.$$

For general $n$, the basis elements of $\pvf((\C^*)^n)$ are of the form
$z_1^{k_1}\dots z_n^{k_n} \bdy_{z_{i_1}} \wedge \dots \wedge \bdy_{z_{i_m}}$
of degree $m-2$,
for $k_1,\dots,k_n \in \Z$ and $\{i_1,\dots,i_m\}$ a subset of $\{1,\dots,n\}$ of size $m \in \{0,\dots,n\}$.
In accordance with the K\"unneth theorem, we can also view the BV algebra $\pvf((\C^*)^n)$ as the tensor product of $n$ copies of the BV algebra $\pvf(C^*)$:
$$ \pvf((\C^*)^n) \cong \underbrace{\pvf(\C^*) \otimes \dots \otimes \pvf(\C^*)}_n.$$

\begin{remark}\label{rmk:string_top_descr}
Alternatively, we have the description (c.f. \cite[\S 6.2]{tonk})
$$ H_{*}(\calL \mathbb{T}^n;\K) \cong \K[z_1^{\pm},\dots,z_n^{\pm}] \otimes H_*(\mathbb{T}^n;\K),$$
where $\calL \mathbb{T}^n$ denotes the free loop space of $\mathbb{T}^n$.
Here the first factor records the first homology class of families of loops and roughly corresponds to the based loop space of $\mathbb{T}^n$ (in particular each $z_i$ has degree zero), while the second factor corresponds to the constant loops.
For $\vec{k} = (k_1,\dots,k_n) \in \Z^n$ put $z^{\vec{k}} = z^{k_1}\cdots z^{k_n}$.
Under this identification, the loop product is of the form 
$$(z^{\vec{k}} \otimes \la) \cdot (z^{\vec{l}} \otimes \mu) = z^{\vec{k} + \vec{l}} \otimes (\la \cdot \mu)$$
for $\vec{k},\vec{l} \in \Z^n$ and $\la,\mu \in H_*(\mathbb{T}^n;\K)$, 
where $\la \cdot \mu \in H_*(\mathbb{T}^n;\K)$ denotes the intersection product on homology.

Let $e_1,\dots,e_n$ denote the standard basis for $H_1(\mathbb{T}^n;\K)$, and let $e_1^{\vee},\dots,e_n^{\vee}$ denote the dual basis for $H^1(\mathbb{T}^n;\K)$.
We identify $H_*(\mathbb{T}^n;\K)$ as a $\K$-module with the exterior algebra on $\K\langle e_1,\dots,e_n\rangle$, and we identify $H^*(\mathbb{T}^n;\K)$ as a $\K$-algebra with the exterior algebra on $\K\langle e_1^{\vee},\dots,e_n^{\vee}\rangle$.
Then the intersection product is given explicitly by
$$\la \cdot \mu = \star^{-1}(\star(\la) \wedge \star(\mu)),$$
where the Poincar\'e duality isomorphism $\star: H_*(\mathbb{T}^n;\K) \rightarrow H^{n-*}(\mathbb{T}^n;\K)$ is characterized by $\star(e_{i_1}\wedge \cdots\wedge e_{i_k}) = \op{sgn}(\sigma) e_{\ovl{i}_1}^{\vee} \wedge \cdots \wedge e_{\ovl{i}_{n-k}}^{\vee}$.
Here for $1 \leq i_1 < \dots < i_k \leq n$ we take $1 \leq \ovl{i}_1< \dots < \ovl{i}_{n-k} \leq n$ to correspond to the complementary subset of $\{1,\dots,n\}$, 
and $\op{sgn}(\sigma)$ denotes the sign of the permutation $\sigma = (i_1,\dots,i_k,\ovl{i}_1,\dots,\ovl{i}_{n-k})$.
As for the BV operator, it is given by
$$\Delta(z^{\vec{k}} \otimes \la) = z^{\vec{k}} \otimes ((k_1e_1+\dots+k_ne_n) \wedge \la).$$

Comparing this with the previous model, the identification $\pvf((\C^*)^n) \approx H_{n-*}(\calL \mathbb{T}^n;\K)$ sends 
$z^{\vec{k}} \bdy_{z_{i_1}}\wedge \dots \wedge \bdy_{z_{i_m}}$ to $z^{\vec{k}}z_{i_1}^{-1}\cdots z_{i_m}^{-1} \otimes \star(e_{i_1} \wedge \cdots \wedge e_{i_m})$.
We will freely switch between these two equivalent models.
\end{remark}

\subsubsection{Filtered homotopy BV algebras}

For any Liouville domain $X$, we expect the BV algebra $\sh(X)$ to admit a natural filtered chain-level refinement, giving $\sc(X)$ the structure of a filtered homotopy BV algebra (see \cite{homotopy_BV}).
Moreover, we expect the Viterbo isomorphism $\sh(T^*\mathbb{T}^n) \cong \pvf((\C^*)^n)$ to extend to a homotopy equivalence of filtered homotopy BV algebras.
At least in the absence of filtrations, a significant step in this direction appears in the work \cite{cohencalabi}.
By a version of Kontsevich's formality theorem (see e.g.\cite{tamarkin-tsygan,campos2017bv}), it makes sense to view the BV algebra $\pvf((\C^*)^n)$ as a homotopy BV algebra which happens to be a differential graded Batalin--Vilkovisky algebra (DGBV) with trivial differential.
If $X_\Omega \subset \C^n$ is a convex toric domain, we endow the DGBV algebra $\pvf((\C^*)^n)$ with its filtration by putting (in terms of the model from Remark~\ref{rmk:string_top_descr})
$$ \calA_\Omega(z^{\vec{k}} \otimes \la) := \max\{ \langle \vec{k},w\rangle\;:\; w \in \Omega\}.$$
Note that we have $\calA_{\Omega}(z^{\vec{k}} \otimes \la) = || \vec{k}||_{\Omega}^*$ as in Definition~\ref{def:dual_norm}) whenever $\vec{k} \in \Z_{\geq 0}^n$, whereas we have $\calA_{\Omega}(z^{\vec{k}} \otimes \la) = 0$ whenever $\vec{k} \in \Z_{\leq 0}^n$.

\subsubsection{Quotienting out the circle action}

Recall that we are most interested in the invariant $\sc_{S^1,+}$. 
To pass from $\sc(D_\Omega)$ to $\sc_{+}(D_\Omega)$, we quotient out the subcomplex generated by elements of the form $z^{\vec{0}} \otimes \lam$, which correspond to constant orbits.
As for the $S^1$-quotient, this appears to be rather complicated for a general homotopy BV algebra, since the quotient must behave like a homotopy quotient in order to have appropriate functoriality properties. 
The general formulation should be given in terms of cyclic homology (see e.g. \cite[\S2]{ganatra1cyclic} for the linear case and \cite{westerland2008equivariant,campos2018gravity} for higher structures).

If $X$ is a Liouville domain, the homotopy quotient $\sc_{S^1,+}(X)$ should inherit the structure of a filtered homotopy gravity algebra (see e.g. \cite{getzler1994two,getzler1995operads,westerland2008equivariant,campos2018gravity}). However, we are only concerned with a small part of this structure, namely the filtered $\Li$ structure on $\sc_{S^1,+}(X)$.
Fortunately, since our model for $\sc(D_\Omega)$ happens to be DGBV such that the BV operator $\Delta$ is acyclic, we can instead realize the $S^1$-quotient as a naive quotient by restricting to the image of $\Delta$.
In other words, we get a model for the filtered $\Li$ algebra $\sc_{S^1,+}(D_\Omega)$ by simply restricting the differential, BV operator, and filtration of $\sc_{+}(D_\Omega)$ to $\im(\Delta) \subset \sc_+(D_\Omega)$. Note that this means that the $\Li$ operations on $\sc_{S^1,+}$ are once-shifted compared with those of $\sc_{+}$, which is to be expected on general grounds (c.f. \cite{westerland2008equivariant}).

Explicitly, in terms of the model from Remark~\ref{rmk:string_top_descr},
$\im(\Delta)$ is generated as a $\K$-module by all elements of the form
$z^{\vec{k}} \otimes \left((k_1e_1+\dots+k_ne_n) \wedge \la\right)$ for $\vec{k} \in \Z^n$ and $\la \in H_*(\mathbb{T}^n;\K)$.
Note that we can also view the subspace of elements of the form 
$z^{\vec{k}} \otimes \left((k_1e_1+\dots+k_ne_n) \wedge \la\right)$ for $\la \in H_*(\mathbb{T}^n;\K)$ as describing the homology of the $(n-1)$-torus $\mathbb{T}^n / S^1_{\vec{k}}$, where $S^1_{\vec{k}}$ denotes the circle acting on $\mathbb{T}^n$ by rotations in the direction $\vec{k}$.

\subsubsection{The Cieliebak--Latschev deformation and first orthant subalgebra}

The inclusion $D_\Omega \subset X_\Omega$ gives rise to a Cieliebak--Latschev Maurer--Cartan element $m \in \sc_{S^1,+}(D_\Omega)$ as in \cite[\S 4]{HSC}.
Using this Maurer--Cartan element, we can define the deformed $\Li$ algebra $\sc_{S^1,+,m}(D_\Omega)$. This is the target of the induced transfer map $\Pi: \sc_{S^1,+}(X_\Omega) \rightarrow \sc_{S^1,+,m}(D_\Omega)$,
and since the action of $m$ is arbitrarily close to zero, this is in fact a filtered $\Li$ homomorphism.
We show in \cite{chscII} that up to filtered gauge equivalence we have 
$m = z_1^{-1} + \dots + z_n^{-1}$. Note that this resembles the superpotential of the Clifford torus in $\C^n$, and indeed the partial compactification $D_\Omega \rightsquigarrow X_\Omega$ is mirror to turning on a superpotential
 (c.f. \cite{Tduality}).

In general, the cobordism map $\Pi$ need not be an $\Li$ homotopy equivalence.
Observe that in our case we also have (up to small shrinkenings) a symplectic embedding in the other direction
$X_\Omega \shookrightarrow D_\Omega$,
defined via the ``Traynor trick'' (see e.g. \cite{landry2015symplectic}).
The induced cobordism map takes the form $\Upsilon: \chlinsc(D_\Omega) \rightarrow \chlinsc(X_\Omega)$, which can be deformed to a map $\Upsilon_m: \sc_{S^1,+,m}(D_\Omega) \rightarrow \chlinsc(X_\Omega)$ (a priori the target is twisted by the push forward of the Maurer--Cartan element $m$, but this turns out to be trivial).

Even in the case $n=1$ the map $\Pi$ cannot be a homotopy equivalence, owing to the fact that the target has generators corresponding to any homology class $\vec{k} \in \Z^n \setminus \{\vec{0}\}$, whereas the source has only generators corresponding to $\vec{k} \in \Z_{\geq 0}^n$.
In fact, the target of $\Pi$ can be restricted to the $\Li$ subalgebra generated by elements of the form 
$$ z^{\vec{k}} \left( (k_1e_1+\dots+k_ne_n) \wedge e_{i_1}\wedge\dots\wedge e_{i_a} \wedge \la\right)$$
for $\vec{k} \in \Z_{\geq 0}$ and $\la \in \H_*(\mathbb{T}^n;\K)$, where $i_1,\dots,i_a$ denotes those indices $i \in \{1,\dots,n\}$ for which $k_i = 0$.
After restricting the target of $\Pi$ and the source of $\Upsilon_m$ to this $\Li$ subalgebra, these maps become inverse $\Li$ homotopy equivalences.

For example, in the case $n=2$, this leaves two generators for each $\vec{k} = (i,j) \in \Z_{\geq 1}^2$ (these corresponding to $\alpha_{i,j},\beta_{i,j}$ in Defintiion~\ref{def:V}), and one generator for each $\vec{k}$ of the form $(i,0)$ or $(0,i)$ with $i \in \Z_{\geq 1}$ (this corresponds to $\beta_{i,0}$ or $\beta_{0,i}$).

\subsubsection{Description of $V_\Omega$}

The above discussion motivates the following.
\begin{definition}
Let $X_\Omega \subset \C^n$ be a convex toric domain.
As a $\K$-module, $V_{\Omega}$ is the subspace of $\pvf((\C^*)^n) = \K[z_1^\pm,\dots,z_n^\pm] \otimes \Lam(\K\langle e_1,\dots,e_n\rangle)$ consisting of all elements of the form $z^{\vec{k}} \otimes \left((k_1e_1+\dots+k_ne_n) \wedge e_{i_1}\wedge\dots\wedge e_{i_a} \wedge \la\right)$, where
\begin{itemize}
\item $\vec{k}  = (k_1,\dots,k_a) \in \Z_{\geq 0}^n$
\item
 $\la \in \Lam(\K\langle e_1,\dots,e_n\rangle)$
 \item
$i_1,\dots,i_a$ are the vanishing indices of $\vec{k}$.
\end{itemize}
As an $\Li$ algebra, $V_{\Omega}$ has bracket $\ell^2$ given by restricting the bracket of $\pvf((\C^*)^n)$, the differential $\ell^1$ is given by
$\ell^1(-) = \ell^2(-,m)$ for $m = \sum_{i=1}^n z_i^{-1}\otimes (e_1\wedge\cdots\wedge e_n)$,
and we have $\ell^k \equiv 0$ for $k \geq 3$.
The filtration is given by 
$\calA_{\Omega}(z^{\vec{k}} \otimes \la) = \max\{ \langle \vec{k},w\rangle\;:\; w \in \Omega\}$.
\end{definition}

\begin{remark}
In the case $n = 2$, the above definition is equivalent to Definition~\ref{def:the_filtration}.
Indeed, the BV operator in $\pvf((\C^*)^2)$ is given by
\begin{align*}
&\Delta(z_1^kz_2^l) = 0\\	
&\Delta(z_1^kz_2^l \bdy_{z_1}) = (k-1)z_1^{k-1}z_2^l\\
&\Delta(z_1^kz_2^l \bdy_{z_2}) = (l-1)z_1^kz_2^{l-1}\\
&\Delta(z_1^kz_2^l \bdy_{z_1} \wedge \bdy_{z_2}) = (k-1)z_1^{k-1}z_2^l \bdy_{z_2} - (l-1)z_1^kz_2^{l-1}\bdy_{z_1},
\end{align*}
and we put $\alpha_{i,j} := iz_1^iz_2^{j+1}\bdy_{z_2} - j z_1^{i+1}z_2^j\bdy_{z_1}$ and $\beta_{i,j} := z^iz^j$.
We then have for example 
\begin{align*}
\bdy(\alpha_{i,j}) &= [\alpha_{i,j},m] \\&= [iz_1^iz_2^{j+1}\bdy_{z_2} - j z_1^{i+1}z_2^j\bdy_{z_1},z_1^{-1} + z_2^{-1}]\\
&= jz_1^{i-1}z_2^j - iz_1^iz_2^{j-1}\\
&= j\beta_{i-1,j} - i\beta_{i,j-1}.
\end{align*}
One can similarly check that we have 
$[\alpha_{i,j},\beta_{k,l}] = (il-jk)\beta_{i+k,j+l}$ and so on.
\end{remark}

\subsection{Computations}

The computations presented in Table~\ref{table:pq} were performed with the aid of a computer program.
Recall that we put
$$S_{d;1,x} := \tfrac{1}{d!k}\langle \Phi_{1,x} \circ \Psi_{1,1}(\odot^d A_2),A_{3d-1}\rangle.$$
Here we consider the case $x = p/q > \tau^4$ with $p + q = 3d$ for small $d$, which is relevant for the restricted stabilized ellipsoid embedding problem in light of Lemma~\ref{lem:suff_cond_for_RSEP}.
\begin{table}[]
\caption{Computations of $S_{d;1,x}$ for low degree $d$.}
\label{table:pq}
\begin{tabular}{|l|l|l|lll}
\hline
$\bf{d}$ & $\bf{x}$ & $\bf{S_{d;1,x}}$ & \multicolumn{1}{l|}{$\bf{d}$} & \multicolumn{1}{l|}{$\bf{x}$} & \multicolumn{1}{l|}{$\bf{S_{d;1,x}}$} \\ \hline
4 & 11 & 26 & \multicolumn{1}{l|}{5} & \multicolumn{1}{l|}{14} & \multicolumn{1}{l|}{217} \\ \hline
6 & 17 & 2110 & \multicolumn{1}{l|}{7} & \multicolumn{1}{l|}{20} & \multicolumn{1}{l|}{22744} \\ \hline
7 & 19/2 & 117 & \multicolumn{1}{l|}{8} & \multicolumn{1}{l|}{23} & \multicolumn{1}{l|}{264057} \\ \hline
8 & 7 & 3 & \multicolumn{1}{l|}{9} & \multicolumn{1}{l|}{26} & \multicolumn{1}{l|}{3242395} \\ \hline
9 & 25/2 & 15789 & \multicolumn{1}{l|}{10} & \multicolumn{1}{l|}{29} & \multicolumn{1}{l|}{41596252} \\ \hline
10 & 9 & 645 & \multicolumn{1}{l|}{11} & \multicolumn{1}{l|}{32} & \multicolumn{1}{l|}{552733376} \\ \hline
11 & 31/2 & 2464347 & \multicolumn{1}{l|}{11} & \multicolumn{1}{l|}{10} & \multicolumn{1}{l|}{9573} \\ \hline
11 & 29/4 & 13 & \multicolumn{1}{l|}{12} & \multicolumn{1}{l|}{35} & \multicolumn{1}{l|}{7559811021} \\ \hline
13 & 38 & 105919629403 & \multicolumn{1}{l|}{13} & \multicolumn{1}{l|}{37/2} & \multicolumn{1}{l|}{430078369} \\ \hline
13 & 12 & 1780887 & \multicolumn{1}{l|}{13} & \multicolumn{1}{l|}{35/4} & \multicolumn{1}{l|}{3944} \\ \hline
14 & 41 & 1514674166755 & \multicolumn{1}{l|}{14} & \multicolumn{1}{l|}{13} & \multicolumn{1}{l|}{25031754} \\ \hline
14 & 37/5 & 68 & \multicolumn{1}{l|}{15} & \multicolumn{1}{l|}{44} & \multicolumn{1}{l|}{22043665219240} \\ \hline
15 & 43/2 & 81723958013 & \multicolumn{1}{l|}{15} & \multicolumn{1}{l|}{41/4} & \multicolumn{1}{l|}{1124321} \\ \hline
16 & 47 & 325734154669786 & \multicolumn{1}{l|}{16} & \multicolumn{1}{l|}{15} & \multicolumn{1}{l|}{4885433892} \\ \hline
16 & 43/5 & 25820 & \multicolumn{1}{l|}{17} & \multicolumn{1}{l|}{50} & \multicolumn{1}{l|}{4877954835706120} \\ \hline
17 & 49/2 & 16592202689939 & \multicolumn{1}{l|}{17} & \multicolumn{1}{l|}{16} & \multicolumn{1}{l|}{69877241271} \\ \hline
17 & 47/4 & 246944186 & \multicolumn{1}{l|}{17} & \multicolumn{1}{l|}{46/5} & \multicolumn{1}{l|}{569847} \\ \hline
17 & 15/2 & 399 & \multicolumn{1}{l|}{18} & \multicolumn{1}{l|}{53} & \multicolumn{1}{l|}{73914684068584441} \\ \hline
18 & 49/5 & 10149686 & \multicolumn{1}{l|}{19} & \multicolumn{1}{l|}{56} & \multicolumn{1}{l|}{1131820243084746628} \\ \hline
19 & 55/2 & 3551018422750862 & \multicolumn{1}{l|}{19} & \multicolumn{1}{l|}{18} & \multicolumn{1}{l|}{14488663452960} \\ \hline
19 & 53/4 & 56189198531 & \multicolumn{1}{l|}{19} & \multicolumn{1}{l|}{52/5} & \multicolumn{1}{l|}{159165980} \\ \hline
19 & 17/2 & 177552 & \multicolumn{1}{l|}{19} & \multicolumn{1}{l|}{50/7} & \multicolumn{1}{l|}{68} \\ \hline
20 & 59 & 17494508772311055354 & \multicolumn{1}{l|}{20} & \multicolumn{1}{l|}{19} & \multicolumn{1}{l|}{212393113297755} \\ \hline
20 & 53/7 & 2530 & \multicolumn{1}{l|}{21} & \multicolumn{1}{l|}{62} & \multicolumn{1}{l|}{272708269111411142397} \\ \hline
21 & 61/2 & 793261958194583682 & \multicolumn{1}{l|}{21} & \multicolumn{1}{l|}{59/4} & \multicolumn{1}{l|}{12285503088082} \\ \hline
21 & 58/5 & 38704647011 & \multicolumn{1}{l|}{21} & \multicolumn{1}{l|}{55/8} & \multicolumn{1}{l|}{3} \\ \hline
22 & 65 & 4283702718045699720655 & \multicolumn{1}{l|}{22} & \multicolumn{1}{l|}{21} & \multicolumn{1}{l|}{46445543517245355} \\ \hline
22 & 61/5 & 641758837605 & \multicolumn{1}{l|}{22} & \multicolumn{1}{l|}{59/7} & \multicolumn{1}{l|}{1267385} \\ \hline
23 & 68 & 67758554214673087717096 & \multicolumn{1}{l|}{23} & \multicolumn{1}{l|}{67/2} & \multicolumn{1}{l|}{183608196484302103721} \\ \hline
23 & 22 & 696398605853414442 & \multicolumn{1}{l|}{24} & \multicolumn{1}{l|}{65/4} & \multicolumn{1}{l|}{2779384665030742} \\ \hline
23 & 64/5 & 10053727601979 & \multicolumn{1}{l|}{23} & \multicolumn{1}{l|}{21/2} & \multicolumn{1}{l|}{25263997971} \\ \hline
23 & 62/7 & 36487435 & \multicolumn{1}{l|}{23} & \multicolumn{1}{l|}{61/8} & \multicolumn{1}{l|}{16965} \\ \hline
24 & 71 & 1078626379578534031088536 & \multicolumn{1}{l|}{24} & \multicolumn{1}{l|}{67/5} & \multicolumn{1}{l|}{151264325566672} \\ \hline
24 & 65/7 & 798939718 &  &  &  \\ \cline{1-3}
\end{tabular}
\end{table}

\bibliographystyle{math}

\bibliography{computing_hsc}

\end{document}